\documentclass{article}
\usepackage{arxiv}
\usepackage{microtype}
\usepackage[utf8]{inputenc} 
\usepackage[T1]{fontenc}    
\usepackage{nicefrac}
\usepackage{lipsum}
\usepackage{graphicx}
\usepackage{subfigure}
\usepackage{amsfonts}
\usepackage{dirtytalk}
\usepackage{mathrsfs}
\usepackage[T1]{fontenc}
\usepackage[T3,T1]{fontenc}
\usepackage{commath}
\usepackage{amsmath}
\allowdisplaybreaks
\usepackage{amsthm}
\usepackage{mathtools}
\usepackage[shortlabels]{enumitem}

\usepackage{natbib}

\DeclareMathOperator*{\argmin}{arg\,min}

\newtheorem{theorem}{Theorem}[section]
\newtheorem{lemma}[theorem]{Lemma}
\newtheorem{proposition}[theorem]{Proposition}
\newtheorem{corollary}[theorem]{Corollary}
\newtheorem{example}[theorem]{Example}

\theoremstyle{definition}
\newtheorem{definition}[theorem]{Definition}

\usepackage{tikz-cd}
\usepackage{booktabs} 
\usepackage{hyperref}
\DeclarePairedDelimiter\ceil{\lceil}{\rceil}

\title{Towards Empirical Process Theory for Vector-Valued Functions: Metric Entropy of Smooth Function Classes}
\author{Junhyung Park\thanks{Corresponding author: \texttt{junhyung.park@tuebingen.mpg.de}}\\
	Max Planck Institute for Intelligent Systems\\
	T\"ubingen, Germany\\
	\And
	Krikamol Muandet\\
	Max Planck Institute for Intelligent Systems\\
	T\"ubingen, Germany}
\date{}
\begin{document}
\maketitle
\begin{abstract}
	This paper provides some first steps in developing empirical process theory for functions taking values in a vector space. Our main results provide bounds on the entropy of classes of smooth functions taking values in a Hilbert space, by leveraging theory from differential calculus of vector-valued functions and fractal dimension theory of metric spaces. We demonstrate how these entropy bounds can be used to show the uniform law of large numbers and asymptotic equicontinuity of the function classes, and also apply it to statistical learning theory in which the output space is a Hilbert space. We conclude with a discussion on the extension of Rademacher complexities to vector-valued function classes. 
\end{abstract}
\section{Introduction}\label{Sintro}
Empirical process theory is an important branch of probability theory that deals with the empirical measure \(P_n=\frac{1}{n}\sum^n_{i=1}\delta_{X_i}\) based on random independent and identically distributed (i.i.d.) copies \(X_1,...,X_n\) of a random variable \(X\) on a domain \(\mathcal{X}\), and stochastic processes of the form \(\{P_nf-Pf:f\in\mathcal{F}\}\), where \(\mathcal{F}\) is a class of functions \(\mathcal{X}\rightarrow\mathbb{R}\). Due to its very nature, the theory has found a wealth of applications in statistics \citep{vandervaart1996weak,vandegeer2000empirical,kosorok2008introduction,shorack2009empirical,dudley2014uniform}. In particular, it has been the major tool in analysing properties of estimators in supervised learning, both in regression and classification \citep{gyorfi2006distribution,steinwart2008support,shalev2014understanding}. 

In the traditional (and still dominant) supervised learning setting, the output space is (a subset of) \(\mathbb{R}\), but there is a rapidly growing literature in machine learning and statistics on learning vector-valued functions \citep{micchelli2005learning,alvarez2012kernels}. This occurs, for example, in multi-task or multi-output learning \citep{evgeniou2005learning,yousefi2018local,xu2019survey,reeve2020optimistic}, functional response models \citep{morris2015functional,kadri2016operator,brault2017large,saha2020learning}, kernel conditional mean embeddings \citep{grunewalder2012conditional,park2020measure} or structured prediction \citep{ciliberto2020general,laforgue2020duality}, among others.

There are valuable works analysing the properties of vector-valued regressors with specific algorithms, notably integral operator techniques in vector-valued reproducing kernel Hilbert space regression \citep{caponnetto2006optimal,kadri2016operator,singh2019kernel,park2020regularised,cabannes2021fast}, and in the form of (local) Rademacher complexities, empirical process theoretic techniques have been applied to cases where the output space is finite dimensional \citep{yousefi2018local,li2019learning,reeve2020optimistic,wu2021fine}. However, as general empirical process theory is developed, to the best of our knowledge, exclusively for classes of real-valued functions, the powerful armoury of empirical process theory has not been utilised fully to analyse vector-valued learning problems. The aim of this paper is to provide some first steps towards developing a theory of empirical processes with vector-valued functions. 

An indispensable object in empirical process theory is metric entropy of function classes, and one of the most frequently used function classes is that of smooth functions. In our main results in Section \ref{Ssmooth}, we investigate how we can bound the entropy of classes of smooth vector-valued functions. When the output space is infinite-dimensional, bounding the entropy becomes far less trivial, compared to the case of real-valued function classes. For example, seemingly benign function classes such as the classes of constant functions onto the unit ball clearly has infinite entropy with respect to any reasonable metric, since the unit ball in an infinite-dimensional Hilbert space is not totally bounded \citep[p.62, Corollary 6]{bollobas1999linear}. 

This requires us to look for other ways to restrict the functions than in the norm sense, and in the main results of this paper in Section \ref{Ssmooth}, we propose considering subsets of the output space with specific geometric features. We leverage notions from dimension theory of metric spaces \citep{heinonen2001lectures,robinson2010dimensions}, specifically the entropy-based upper box-counting and Assouad dimensions. These, along with other dimensions such as the Hausdorff or packing dimensions, are inherently fractal and thus are studied extensively in fractal geometry \citep{edgar2007measure,massopust2014fractal,fraser2020assouad}, but we do not make explicit use of their fractal nature. Rather, we investigate how restricting our function classes to subsets of the output space with these properties can help us bound their entropies. We use these entropy bounds to show uniform law of large numbers and asymptotic equicontinuity, and in Section \ref{Sstatisticallearningtheory}, we demonstrate applications in statistical learning theory, and discuss the generalisation of the popular Rademacher complexity to the vector-valued setting. 

\subsection{Mathematical Preliminaries \& Notations}\label{SSmaths}
Let \(\mathcal{Y}\) be a separable Hilbert space over \(\mathbb{R}\), with its inner product and norm denoted by \(\langle\cdot,\cdot\rangle_\mathcal{Y}\) and \(\lVert\cdot\rVert_\mathcal{Y}\) respectively. We denote by \(\mathscr{Y}\) the Borel \(\sigma\)-algebra of \(\mathcal{Y}\), i.e. the \(\sigma\)-algebra generated by the open subsets of \(\mathcal{Y}\). Let \((\mathcal{X},\mathscr{X})\) be a measurable set, and \(Q\) a probability measure on it. 

\paragraph{Bochner Integration} A function \(g:\mathcal{X}\rightarrow\mathcal{Y}\) is said to be \textit{Bochner-integrable} with respect to \(Q\) if \(g\) is strongly measurable and if \(\lVert g\rVert_\mathcal{Y}\) is \(Q\)-integrable \citep[p.15, Definition 35]{dinculeanu2000vector}, and denote its Bochner integral by \(\int gdQ\in\mathcal{Y}\). We denote the space of Bochner \(Q\)-integrable functions by \(L^1(\mathcal{X},Q;\mathcal{Y})\). Further, for \(1\leq p<\infty\), we denote by \(L^p(\mathcal{X},Q;\mathcal{Y})\) the space of functions \(g:\mathcal{X}\rightarrow\mathcal{Y}\) such that \(\int\lVert g\rVert^p_\mathcal{Y}dQ<\infty\), and denote the corresponding seminorm by \(\lVert g\rVert_{p,Q}^p=\int\lVert g\rVert^p_\mathcal{Y}dQ\). The case \(p=2\) is a special case, where \(L^2(\mathcal{X},Q;\mathcal{Y})\) can be equipped with a semi-inner product \(\langle g_1,g_2\rangle_{2,Q}=\int\langle g_1,g_2\rangle_\mathcal{Y}dQ\). Finally, we denote by \(L^\infty(\mathcal{X};\mathcal{Y})\) the space of functions \(g:\mathcal{X}\rightarrow\mathcal{Y}\) such that the uniform norm \(\left\lVert g\right\rVert_\infty=\sup_{x\in\mathcal{X}}\left\lVert g(x)\right\rVert_\mathcal{Y}\) is bounded. Following \citet[p.16]{vandegeer2000empirical}, we do not consider the essential supremum (which depends on the measure \(Q\)), but the supremum over \textit{all} \(x\in\mathcal{X}\), so that the uniform norm does not depend on any measure. 

\paragraph{Taylor's Theorem for Vector-Valued Functions} The notions of (partial) differentiation and smoothness for \(\mathcal{Y}\)-valued functions are central in our bounds for entropy of smooth functions (see Appendix \ref{Sdiffcalcfull} for more details, and \citet{cartan1967calcul,coleman2012calculus} for full expositions). Suppose that \(U\) is an open subset of \(\mathbb{R}^d\), and denote the Euclidean norm in \(\mathbb{R}^d\) by \(\lVert\cdot\rVert\). For \(m\in\mathbb{N}\) and an \(m\)-times differentiable function \(g:U\rightarrow\mathcal{Y}\), we write \(g^{(m)}\) for the \(m^\text{th}\) derivative of \(g\), an \(m\)-linear operator from \(\mathbb{R}^d\) into \(\mathcal{Y}\) (see Appendix \ref{Sdiffcalcfull}). We state the extension of Taylor's theorem to functions with values in \(\mathcal{Y}\), with Lagrange's form of the remainder. To this end, for \(a,b\in\mathbb{R}^d\), define the \textit{segment} joining \(a\) and \(b\) as the set \([a,b]=\{x\in\mathbb{R}^d:x=va+(1-v)b,v\in[0,1]\}\) \citep[p.51]{coleman2012calculus}. 
\begin{theorem}[{\citet[p.77, Th\'eor\`eme 5.6.2]{cartan1967calcul}}]\label{Ttaylorlagrange}
	Suppose that \(g:U\rightarrow\mathcal{Y}\) is \((m+1)\)-times differentiable, that the segment \([a,a+h]\) is contained in \(U\) and that, for some \(K>0\), we have \(\lVert g^{(m+1)}(x)\rVert_\textnormal{op}\leq K\) for all \(x\in U\). Then
	\[\Bigg\lVert g(a+h)-\sum^m_{k=0}\frac{1}{k!}g^{(k)}(a)((h)^k)\Bigg\rVert_\mathcal{Y}\leq K\frac{\left\lVert h\right\rVert^{m+1}}{(m+1)!},\]
	where we wrote \((h)^k=(h,...,h)\in(\mathbb{R}^d)^k\) for \(k=1,...,m\).
\end{theorem}
Write \(\mathbb{N}_0=\{0,1,2,...\}\), and for \(p=(p_1,...,p_d)\in\mathbb{N}^d_0\), write \([p]\vcentcolon=p_1+...+p_d\). Then we denote the \(p^\text{th}\) partial derivative \(\partial^{p_1}_1...\partial^{p_d}_dg(a)\) of \(g\) at \(a\in U\) as \(D^pg(a)\in\mathcal{Y}\). For each \(k=1,...,m+1\), \(g^{(k)}(a)((h)^k)=\sum^d_{l_1,...,l_k=1}h_{l_1}...h_{l_k}\partial_{l_1}...\partial_{l_k}g(a)=\sum_{[p]=k}\frac{k!h^p}{p!}D^pg(a)\), where we wrote \(h^p\) as a shorthand for \(h_1^{p_1}...h_d^{p_d}\) and \(p!\) for \(p_1!...p_d!\). Hence, using partial derivatives, we can express Taylor's theorem above as
\[\Bigg\lVert g(a+h)-\sum_{[p]\leq m}\frac{h^p}{p!}D^pg(a)\Bigg\rVert_\mathcal{Y}\leq K\frac{\left\lVert h\right\rVert^{m+1}}{(m+1)!}.\]

\paragraph{Metric Spaces, Covering Numbers and Dimensions} Finally, we introduce some notions from the theory of metric spaces. In particular, covering numbers play a central role in entropy discussions, and different notions of dimensions based on covering numbers will be used to restrict the range of partial derivatives of functions, leading up to entropy bounds in our main results (Section \ref{Ssmooth}).

Suppose \((\mathcal{Z},\rho)\) is a metric space. For \(r>0\) and \(z_0\in\mathcal{Z}\), the \textit{ball of radius \(r\) centred at \(z_0\)} is \(\mathcal{B}(z_0,r)=\{z\in\mathcal{Z}:\rho(z,z_0)\leq r\}\). For any \(\delta>0\), the \textit{\(\delta\)-covering number} of \((\mathcal{Z},\rho)\), denoted by \(N(\delta,\mathcal{Z},\rho)\), is the minimum number of balls of radius \(\delta\) with centres in \(\mathcal{Z}\) required to cover \(\mathcal{Z}\), i.e. the minimal \(N\) such that there exists a set \(\{z_1,...,z_N\}\subset\mathcal{Z}\) such that for all \(z\in\mathcal{Z}\), there exists a \(j=j(z)\in\{1,...,N\}\) with \(\rho(z,z_j)\leq\delta\) (we take \(N(\delta,\mathcal{Z},\rho)=\infty\) if no finite covering by closed balls with radius \(\delta\) exists). We say that \(\mathcal{Z}\) is \textit{totally bounded} if \(N(\delta,\mathcal{Z},\rho)<\infty\) for all \(\delta>0\). We define the \textit{\(\delta\)-entropy} as \(H(\delta,\mathcal{Z},\rho)=\log N(\delta,\mathcal{Z},\rho)\). 

Let \(E\) be a subset of \((\mathcal{Z},\rho)\). The \textit{upper box-counting dimension} of \(E\) is
\[\tau_\text{box}(E)\vcentcolon=\limsup_{\delta\rightarrow0}\frac{H(\delta,E,\rho)}{-\log\delta}\]
\citep[p.32, Definition 3.1]{robinson2010dimensions}. It is immediate from the definition \citep[p.32, (3.3)]{robinson2010dimensions} that if \(\tau>\tau_\text{box}(E)\), then there exists \(\delta_0>0\) such that for all \(\delta<\delta_0\),
\[N(\delta,E,\rho)<\delta^{-\tau}.\tag{box}\]
A subset \(E\) of \((\mathcal{Z},\rho)\) is said to be \textit{\((M,\tau)\)-homogeneous} (or simply \textit{homogeneous}) if the intersection of \(E\) with any closed ball of radius \(R\) can be covered by at most \(M\left(\frac{R}{r}\right)^\tau\) closed balls of smaller radius \(r\), i.e. \(N(r,\mathcal{B}(z,R)\cap E,\rho)\leq M\left(\frac{R}{r}\right)^\tau\) for all \(z\in E\) and \(R>r\) \citep[p.83, Definition 9.1]{robinson2010dimensions}. The \textit{Assouad dimension} \citep[p.85, Definition 9.5]{robinson2010dimensions}, sometimes also known as the \textit{doubling dimension}, of \(E\) is
\[\tau_\text{asd}(E)\vcentcolon=\inf\{\tau:E\text{ is }(M,\tau)\text{-homogeneous for some }M\geq1\}.\]

\section{Empirical Process Theory for Functions Taking Values in a Hilbert Space}\label{Sempiricalprocesstheory}
Take \((\Omega,\mathscr{F},\mathbb{P})\) as the underlying probability space. Let \(X:\Omega\rightarrow\mathcal{X}\) be a random variable, and let \(X_1,X_2,...\) be i.i.d. copies of \(X\). Denote by \(P\) its distribution, i.e. for \(A\in\mathscr{X}\), \(P(A)=\mathbb{P}(X^{-1}(A))\), and by \(P_n\) the empirical measure on \(\mathcal{X}\) based on \(X_1,...,X_n\), i.e.
\[P_n=\frac{1}{n}\sum^n_{i=1}\delta_{X_i},\qquad\text{where, for }A\in\mathscr{X},\delta_{X_i}(A)=\begin{cases}0&\text{if }X_i\notin A\\1&\text{if }X_i\in A\end{cases}.\]
For a function \(g\in L^1(\mathcal{X},Q;\mathcal{Y})\), we adopt the notation \(Qg=\int gdQ\). Hence,
\[Pg=\int gdP\qquad\text{and}\qquad P_ng=\frac{1}{n}\sum^n_{i=1}g(X_i).\]
Note that the integral \(Pg\) is a Bochner integral, and that we have \(Pg,P_ng\in\mathcal{Y}\). Now, for fixed \(g\), the law of large numbers in Hilbert (more generally, Banach) spaces \citep{mourier1953elements} tells us that \(P_ng\) converges to \(Pg\). One of the pillars of empirical process theory is to consider the convergence of \(P_ng\) to \(Pg\) not for a fixed \(g\), but uniformly over a class of functions. Let \(\mathcal{G}\subset L^1(\mathcal{X},P;\mathcal{Y})\). For a measure \(Q\) on \(\mathcal{X}\), we denote \(\left\lVert Q\right\rVert_\mathcal{G}\vcentcolon=\sup_{g\in\mathcal{G}}\left\lVert Qg\right\rVert_\mathcal{Y}\). 
\begin{definition}\label{Dglivenkocantelli}
	We say that the class \(\mathcal{G}\) is a \textit{Glivenko Cantelli (GC)} class, or that it satisfies the \textit{uniform law of large numbers} (with respect to the measure \(P\)) if \(\left\lVert P_n-P\right\rVert_\mathcal{G}=\sup_{g\in\mathcal{G}}\left\lVert P_ng-Pg\right\rVert_\mathcal{Y}\stackrel{P}{\rightarrow}0\). 
\end{definition}
Definition \ref{Dglivenkocantelli} could have been defined in terms of the weak convergence in Hilbert spaces, i.e. \(y_n\rightarrow y_0\) if \(\langle y,y_n\rangle_\mathcal{Y}\rightarrow\langle y,y_0\rangle_\mathcal{Y}\) for every \(y\in\mathcal{Y}\). In this paper, we only consider strong (norm) convergence. Next, we define the empirical process and the asymptotic equicontinuity.
\begin{definition}\label{Dempiricalprocess}
	We regard \(\left\{\nu_n(g)=\sqrt{n}\left(P_n-P\right)g:g\in\mathcal{G}\right\}\) as a stochastic process with values in \(\mathcal{Y}\) indexed by \(\mathcal{G}\), and call it the \textit{empirical process}. 
	
	We say that the empirical process \(\left\{\nu_n(g):g\in\mathcal{G}\right\}\) is \textit{asymptotically equicontinuous} at \(g_0\in\mathcal{G}\) if, for every sequence \(\left\{\hat{g}_n\right\}\subset\mathcal{G}\) with \(\left\lVert\hat{g}_n-g_0\right\rVert_{2,P}\stackrel{P}{\rightarrow}0\), we have \(\left\lVert\nu_n\left(\hat{g}_n\right)-\nu_n\left(g_0\right)\right\rVert_\mathcal{Y}\stackrel{P}{\rightarrow}0\). 
\end{definition}

Some of the first steps in empirical process theory are the symmetrisation and chaining techniques, and using them to prove uniform law of large numbers and asymptotic equicontinuity for classes of functions that satisfy certain entropy conditions. We provide the adaptation of some of these results for vector-valued function classes but defer them to Appendix \ref{Sempiricalprocessesfull}, because, while strictly speaking novel, the statements and proofs of these results carry over from the case of real-valued function classes with only minor adjustments, in particular with concentration inequalities for vector-valued random variables \citep{pinelis1992approach}. 

A more challenging task, as mentioned in the Introduction, is to bound entropies of vector-valued function classes, and the main results of this paper will focus on this problem (Section \ref{Ssmooth}). In the usual theory of empirical processes with real-valued functions, there are two major tools. The first is to consider the entropy with respect to the empirical measure \(P_n\). One usually requires this entropy to be uniformly bounded over all realisations of the samples \(X_1,...,X_n\), and the most widely-used example of function classes that satisfy this property are the celebrated \textit{Vapnik-Chervonenkis (VC) subgraph classes} \citep[Section 2.6]{vandervaart1996weak}. The second tool is what is known as \textit{entropy with bracketing} with respect to the underlying measure \(P\) (see, for example, \citet[p.122, Theorem 2.4.1 and p.129, Section 2.5.2]{vandervaart1996weak}, \citet[Sections 3.1 and 5.5]{vandegeer2000empirical} and \citet[Chapter 7]{dudley2014uniform}). However, both VC subgraph classes and entropy with bracketing make explicit use of the fact that the output space \(\mathbb{R}\) is \textit{totally-ordered}, and makes use of objects such as \(\{x\in\mathcal{X}:x\leq g(x_0)\}\) and \(\{x\in\mathcal{X}:g_1(x_0)\leq x\leq g_2(x_0)\}\), where \(g,g_1,g_2\in\mathcal{G}\) and \(x_0\in\mathcal{X}\). A direct extension is clearly not possible when our output space \(\mathcal{Y}\) has any dimension greater than 1, and an attempt at an extension is even more difficult when \(\mathcal{Y}\) is infinite-dimensional. In this paper, we do not investigate whether it is possible to obtain meaningful results by extending these ideas, and leave it for future work. 

We mention that in this work, we overlook the problem of measurability, which arise as we take suprema over possibly uncountable sets. This is commonly done in works treating statistical applications of empirical processes (see, e.g. \citet[p.21, Section 2.5]{vandegeer2000empirical}, \citet[p.7, first paragraph of Section 2]{bartlett2005local} and \citet[p.5, last paragraph of Section 1]{yousefi2018local}). We either assume that function classes and underlying distributions satisfy conditions that ensure measurability, or that notions of outer probabilities and expectations are used instead, as in \citet{vandervaart1996weak} and \citet{kosorok2008introduction}. 

\section{Entropy of Classes of Smooth Vector-Valued Functions}\label{Ssmooth}
In the usual empirical process theory with real-valued functions, classes of smooth functions on compact domains are some of the most frequently used examples that satisfy good entropy conditions \citep[p.154, Example 9.3.2]{vandegeer2000empirical}, \citep[Section 2.7.1]{vandervaart1996weak}, \citep[Section 8.2]{dudley2014uniform}. In this section, we give analogues of these results when the output space is the (not necessarily finite-dimensional) Hilbert space \(\mathcal{Y}\). 

Let \(m\in\mathbb{N}\); this will determine the smoothness of our function class. Let \(d\geq1\), and let us take as our input space the unit cube in \(\mathbb{R}^d\), \(\mathcal{X}=\{x\in\mathbb{R}^d:0\leq x_j\leq 1,j=1,...,d\}\); this is only to simplify the exposition, and the subsequent results will clearly hold for any compact subsets of \(\mathbb{R}^d\).  

In order to bound the entropy of classes of smooth real-valued functions, one bounds the absolute values of the range and partial derivatives of the function class. When the output space is \(\mathcal{Y}\), in particular, if \(\mathcal{Y}\) has infinite dimensions, bounding the norm of the range and partial derivatives is useless, because balls in infinite-dimensional spaces are not totally bounded. Therefore, to have any hope of bounding the entropy of function classes taking values in \(\mathcal{Y}\), the very least we need to do is to find a totally bounded subset \(B\subset\mathcal{Y}\), and restrict our range and partial derivatives therein. 

Denote by \(\mathcal{G}_B^m\) the set of \(m\)-times differentiable functions \(g:\mathcal{X}\rightarrow\mathcal{Y}\) whose partial derivatives \(D^pg:\mathcal{X}\rightarrow\mathcal{Y}\) of orders \([p]\leq m\) exist everywhere on the interior of \(\mathcal{X}\), and such that \(D^pg(x)\in B\) for all \(x\in\mathcal{X}\) and \([p]\leq m\), where \(D^0g=g\). We present three results bounding \(H(\delta,\mathcal{G}^m_B,\lVert\cdot\rVert_\infty)\) for \(\delta>0\) sufficiently small, each with different assumptions on \(B\). Theorem \ref{Tentropyassouad} assumes that \(B\) is homogeneous, i.e. we impose local entropy conditions. In Theorems \ref{Tentropybox} and \ref{Tentropyexponential}, we impose global entropy conditions on \(B\), the former with finite upper box-counting dimension, and the latter with \(N(\delta,B,\lVert\cdot\rVert_\mathcal{Y})\) allowed to grow exponentially as \(\delta\) decreases. Proofs are deferred to Section \ref{SSproofs}. 
\begin{theorem}\label{Tentropyassouad}
	Let \(B\subset\mathcal{Y}\) be totally bounded and \((M,\tau_\textnormal{asd})\)-homogeneous. Then for sufficiently small \(\delta>0\), there exists some constant \(K\) depending on \(K_B\), \(m\), \(d\), \(M\) and \(\tau_\textnormal{asd}\) such that
	\[H\left(\delta,\mathcal{G}^m_B,\lVert\cdot\rVert_\infty\right)\leq K\delta^{-\frac{d}{m}}.\]
\end{theorem}
Theorem \ref{Tentropyassouad} gives the same rate for \(\mathcal{G}^m_B\) as for smooth real-valued function classes \citep[p.288, Theorem 8.4(a)]{dudley2014uniform}, which is a special case of the set-up in Theorem \ref{Tentropyassouad}, since any bounded subset of \(\mathbb{R}\) is a homogeneous subset (with Assouad dimension at most 1). In fact, \citet[Theorem 8.4(a)]{dudley2014uniform} shows that this rate of \(\delta^{-\frac{d}{m}}\) cannot be improved, so the rate given in Theorem \ref{Tentropyassouad} is also optimal. We will later see from the proof that the dependence on \(\tau_\text{asd}\) is linear. 
\begin{theorem}\label{Tentropybox}
	Let \(B\) be a subset of \(\mathcal{Y}\) with finite upper box-counting dimension \(\tau_\textnormal{box}\). Then for sufficiently small \(\delta>0\), there exists some constant \(K\) depending on \(K_B\), \(m\), \(d\) and \(\tau_\textnormal{box}\) such that
	\[H\left(\delta,\mathcal{G}^m_B,\lVert\cdot\rVert_\infty\right)\leq K\delta^{-\frac{d}{m}}\log\left(\frac{1}{\delta}\right).\]
\end{theorem}
\begin{theorem}\label{Tentropyexponential}
	Let \(B\) be a subset of \(\mathcal{Y}\) with \(N(\epsilon,B,\lVert\cdot\rVert_\mathcal{Y})\leq\exp\{M\epsilon^{-\tau_\textnormal{exp}}\}\) for some \(M,\tau_\textnormal{exp}>0\). Then for sufficiently small \(\delta>0\), there is some constant \(K\) depending on \(K_B\), \(m\), \(d\), \(M\) and \(\tau_\textnormal{exp}\) such that
	\[H\left(\delta,\mathcal{G}^m_B,\lVert\cdot\rVert_\infty\right)\leq K\delta^{-\left(\frac{d}{m}+\tau_\textnormal{exp}\right)}.\]
\end{theorem}
We can use results stated and proved in Appendix \ref{Sempiricalprocessesfull} to show that we have uniform law of large numbers over \(\mathcal{G}^m_B\), where \(B\) satisfies the conditions in any one of Theorems \ref{Tentropyassouad}, \ref{Tentropybox} or \ref{Tentropyexponential}. 
\begin{corollary}
	The function class \(\mathcal{G}^m_B\), where \(B\) is either homogeneous, has finite upper box-counting dimension or satisfies \(N(\epsilon,B,\lVert\cdot\rVert_\mathcal{Y})\leq\exp\{M\epsilon^{-\tau_\textnormal{exp}}\}\) for some \(\tau_\textnormal{exp}>0\), is Glivenko-Cantelli. 
\end{corollary}
Further, the empirical process defined by \(\mathcal{G}^m_B\) (c.f. Definition \ref{Dempiricalprocess}) is asymptotically equicontinuous. 
\begin{corollary}
	Suppose that \(B\) is either homogeneous, has finite upper box-counting dimension or satisfies \(N(\epsilon,B,\lVert\cdot\rVert_\mathcal{Y})\leq\exp\{M\epsilon^{-\tau_\textnormal{exp}}\}\) for some \(\tau_\textnormal{exp}>0\). Then the empirical process \(\{\nu_n(g)=\sqrt{n}(P_n-P)g:g\in\mathcal{G}^m_B\}\) defined by \(\mathcal{G}^m_B\) is asymptotically equicontinuous.
\end{corollary}

\subsection{Examples}\label{SSexamples}
With these results in hand, it is now of interest to investigate which interesting examples of output space \(\mathcal{Y}\) and subsets \(B\) satisfy the conditions of Theorems \ref{Tentropyassouad}, \ref{Tentropybox} and \ref{Tentropyexponential}. 
\begin{example}\label{Emulticlass}
	Suppose that \(\mathcal{Y}\) is a finite-dimensional Hilbert space, say with dimension \(d_\mathcal{Y}\). Then balls are totally bounded, so we can let \(B\) be of the form \(B=\{y\in\mathcal{Y}:\lVert y\rVert_\mathcal{Y}\leq K\}\) for any \(K>0\). Moreover, subsets of finite-dimensional spaces are homogeneous with Assouad dimension at most \(d_\mathcal{Y}\) \citep[p.85, Lemma 9.6(iii)]{robinson2010dimensions}, and so we can apply Theorem \ref{Tentropyassouad}. The case \(\mathcal{Y}=\mathbb{R}\) corresponds to the usual regression with real-valued output. If \(\mathcal{Y}=\mathbb{R}^{d_\mathcal{Y}}\), it corresponds to the multi-task learning setting \citep{evgeniou2005learning,yousefi2018local,xu2019survey}. 
\end{example}
A prominent application of vector-valued output spaces will be when we have functional responses; example data sets include speech, diffusion tensor imaging, mass spectrometry and glaucoma (see \citet{morris2015functional,kadri2016operator} and references therein). Let \(\mathcal{X}'\) be a domain, and \(\mathcal{Y}=L^2(\mathcal{X}',P';\mathbb{R})\) the space of real-valued functions that are square-integrable with respect to some distribution \(P'\) on \(\mathcal{X}'\). By considering interesting subsets of \(\mathcal{Y}\), we can derive bounds on the entropy \(H(\delta,\mathcal{G}^m_B,\lVert\cdot\rVert_\infty)\) using Theorems \ref{Tentropyassouad}, \ref{Tentropybox} and \ref{Tentropyexponential}. The next 4 examples are considered in this set-up. 
\begin{example}\label{Efindim}
	Suppose that \(\psi_1,...,\psi_r\in\mathcal{Y}\), and let \(B=\{f=\theta_1\psi_1+...+\theta_r\psi_r:\theta=(\theta_1,...,\theta_r)^T\in\mathbb{R}^r,\lVert f\rVert_{2,P'}\leq R\}\)
	Then \citet[p.20, Lemma 2.5]{vandegeer2000empirical} tells us that \(B\) is homogeneous, and so Theorem \ref{Tentropyassouad} applies. This corresponds to the case where the responses are finite-dimensional functions, or adopting the nomenclature of \citet[p.152, Example 9.3.1]{vandegeer2000empirical}, \say{linear regressors}.
\end{example}
\begin{example}\label{Eassouad}
	More generally, function classes with finite Assouad dimensions have been considered in classification problems, and their generalisation properties analysed \citep{li2007learnability,bshouty2009using}. If these functions form the responses of a regression problem, then Theorem \ref{Tentropyassouad} can again be applied. Examples of such function classes include halfspaces with respect to the uniform distribution (i.e. where \(P'\) is the uniform distribution) \citep[Proposition 6]{bshouty2009using}. 
\end{example}
\begin{example}\label{Esmooth}
	Let \(\mathcal{X}'\) be compact in \(\mathbb{R}^{d'}\) (in general, \(d\neq d'\)), and suppose that \(B\subset\mathcal{Y}\) consists of smooth functions. More specifically, for some \(m'\in\mathbb{N}\) and \(M>0\), let \(B\) be the set of all \(m'\)-times differentiable functions \(f:\mathcal{X}'\rightarrow\mathbb{R}\) whose partial derivatives \(D^qf:\mathcal{X}'\rightarrow\mathbb{R}\) of orders \([q]\leq m'\) exist everywhere on the interior of \(\mathcal{X}'\), and such that \(\lvert D^qf(x')\rvert\leq M\) for all \(x'\in\mathcal{X}'\) and \([q]\leq m'\). Then applying the result for real-valued function classes \citep[p.288, Theorem 8.4]{dudley2014uniform} (or Theorem \ref{Tentropyassouad} with \(\mathcal{Y}=\mathbb{R}\) and \(B\) being the ball of radius \(M\)), we have \(N(\delta,B,\lVert\cdot\rVert_\infty)\leq\exp\{K'\delta^{-\frac{d'}{m'}}\}\) for some constant \(K'>0\). This in turn allows us to apply Theorem \ref{Tentropyexponential} to bound the entropy of \(\mathcal{G}^m_B\) as
	\[H(\delta,\mathcal{G}^m_B,\lVert\cdot\rVert_\infty)\leq K\delta^{-\left(\frac{d}{m}+\frac{d'}{m'}\right)}\]
	for some constant \(K>0\). So when the output space is itself a class of smooth (real-valued) functions, the smoothness of the two function classes simply add in the negative exponent of \(\delta\) in the entropy. 
\end{example}
\begin{example}\label{Erkhs}
	Let \(B\) be a ball in a reproducing kernel Hilbert space (RKHS) with a \(\mathcal{C}^\infty\) Mercer kernel (see \citet{cucker2002mathematical} for details), then \citet[Theorem D]{cucker2002mathematical} tells us that for some constant \(K'\), we have \(N(\delta,B,\lVert\cdot\rVert_\infty)\leq\exp\{K'\delta^{-\frac{2d}{h}}\}\) for any \(h>d\). Then we can again apply Theorem \ref{Tentropyexponential} to bound \(H(\delta,\mathcal{G}^m_B,\lVert\cdot\rVert_\infty)\) by \(K\delta^{-(\frac{d}{m}+\frac{2d}{h})}\) for some constant \(K\) and any \(h>d\). 
\end{example}

\subsection{Proofs of the Main Results}\label{SSproofs}
We now prove Theorems \ref{Tentropyassouad}, \ref{Tentropybox} and \ref{Tentropyexponential}.
The idea is to approximate smooth functions by piecewise polynomials \citep{kolmogorov1955bounds}. We start with some development shared by the three Theorems. 

As \(B\) is totally bounded, for some \(K_B>0\), \(\lVert y\rVert_\mathcal{Y}\leq K_B\) for all \(y\in B\). Let \(g\in\mathcal{G}^m_B\), \(x\in\mathcal{X}\), \(x+h\in\mathcal{X}\) and \(p\in\mathbb{N}^d_0\) with \([p]\leq m-1\). Then \(D^pg\) is \((m-[p])\)-times differentiable, and \(\lVert (D^pg)^{(m-[p])}(x)\rVert_\text{op}=\lVert\sum_{[q]\leq m-[p]}\frac{(m-[p])!}{q!}D^{p+q}g(x)\rVert_\mathcal{Y}\leq d^{m-[p]}K_B\). Hence, 
\[D^pg(x+h)=\sum_{[q]\leq m-1-[p]}\frac{h^q}{q!}D^{p+q}g(x)+R_p(g,x,h)\tag{*}\]
by Taylor's Theorem (Theorem \ref{Ttaylorlagrange}), where \(\lVert R_p(g,x,h)\rVert_\mathcal{Y}\leq d^{m-[p]}K_B\frac{\lVert h\rVert^{m-[p]}}{(m-[p])!}\). So there is a constant \(K_1=K_1(K_B,m,d)\geq1\) such that, for all \(g\in\mathcal{G}^m_B\), \(x\in\mathcal{X}\), \(x+h\in\mathcal{X}\) and \(p\in\mathbb{N}^d_0\) with \([p]\leq m-1\),
\[\left\lVert R_p(g,x,h)\right\rVert_\mathcal{Y}\leq K_1\left\lVert h\right\rVert^{m-[p]}.\tag{**}\]	
Let \(\Delta\vcentcolon=(\frac{\delta}{4K_1})^{\frac{1}{m}}\), and \(x_{(1)},...,x_{(L)}\) a \(\frac{\Delta}{2}\)-net in \(\mathcal{X}\), i.e. \(\sup_{x\in\mathcal{X}}\{\inf_{1\leq l\leq L}\lVert x-x_{(l)}\rVert\}\leq\frac{\Delta}{2}\). By decomposing \(\mathcal{X}\) into cubes of side \(\ceil*{\frac{d^{1/2}}{\Delta}}^{-1}\) and taking the \(x_{(l)}\) as the centres thereof, we can take
\[L\leq K_2\delta^{-\frac{d}{m}}\tag{\(\dagger\)}\]
for some constant \(K_2=K_2(d,K_1)\). Now, for each \(k=0,1,...,m-1\), define \(\delta_k=\frac{\delta}{2\Delta^ke^d}\). We construct a cover of \(B\) as follows. First, to ease the notation, write \(N_k=N(\frac{1}{2}\delta_k,B,\lVert\cdot\rVert_\mathcal{Y})\), and find a set \(\{a^k_j,j=1,...,N_k\}\subset B\) such that \(\mathcal{B}(a^k_j,\frac{1}{2}\delta_k)\) cover \(B\). Then define
\[A^k_1=\mathcal{B}(a^k_1,\frac{1}{2}\delta_k),A^k_2=\mathcal{B}(a^k_2,\frac{1}{2}\delta_k)\backslash\mathcal{B}(a^k_1,\frac{1}{2}\delta_k),...,A^k_{N_k}=\mathcal{B}(a^k_{N_k},\frac{1}{2}\delta_k)\backslash\cup_{j=1}^{N_k-1}\mathcal{B}(a^k_j,\frac{1}{2}\delta_k).\]
Then \(\mathscr{A}_k\vcentcolon=\{A^k_j,j=1,...,N_k\}\) is a cover of \(B\) of cardinality \(N_k\), whose sets \(A^k_j\) have diameter at most \(\delta_k\) and are disjoint. For each \(l=1,...,L\), \(g\in\mathcal{G}^m_B\) and \(p\in\mathbb{N}^d_0\) with \([p]\leq m-1\), define \(A_{l,p}(g)\) as the unique set in \(\mathscr{A}_{[p]}\) such that \(D^pg(x_{(l)})\in A_{l,p}(g)\), and \(a_{l,p}(g)\) as the centre of the ball from which \(A_{l,p}(g)\) was created, so that \(\lVert a_{l,p}(g)-D^pg(x_{(l)})\rVert_\mathcal{Y}\leq\frac{1}{2}\delta_{[p]}\). Then if \(g_1,g_2\in\mathcal{G}^m_B\) are such that \(A_{l,p}(g_1)=A_{l,p}(g_2)\) for all \(l=1,...,L\) and all \(p\in\mathbb{N}^d_0\) with \([p]\leq m-1\), then
\[\lVert D^p(g_1-g_2)(x_{(l)})\rVert_\mathcal{Y}\leq\delta_{[p]},\tag{***}\]
since the diameter of \(A_{l,p}(g_1)=A_{l,p}(g_2)\) is at most \(\delta_{[p]}\). For each \(x\in\mathcal{X}\), take \(x_{(l)}\) such that \(\lVert x-x_{(l)}\rVert\leq\frac{\Delta}{2}\). Then we have, by putting \(p=0\) into (*),
\begin{alignat*}{2}
	&\left\lVert(g_1-g_2)(x)\right\rVert_\mathcal{Y}\\
	&=\left\lVert R_0(g_1,x_{(l)},x-x_{(l)})-R_0(g_2,x_{(l)},x-x_{(l)})+\sum_{[p]\leq m-1}\frac{(x-x_{(l)})^p}{p!}D^p(g_1-g_2)(x_{(l)})\right\rVert_\mathcal{Y}\\
	&\leq2K_1\lVert x-x_{(l)}\rVert^m+\sum_{[p]\leq m-1}\delta_{[p]}\frac{\lVert x-x_{(l)}\rVert^p}{p!}\qquad\text{by (**) with }p=0\text{ and (***)}\\
	&\leq2K_1\Delta^m+\sum^{m-1}_{k=0}\delta_k\Delta^k\left(\sum_{[p]=k}\frac{1}{p!}\right)\leq\frac{\delta}{2}+\left(\max_{k\leq m-1}\delta_k\Delta^k\right)\sum_{k=0}^{m-1}\frac{d^k}{k!}\leq\frac{\delta}{2}+\frac{\delta}{2e^d}e^d=\delta.
\end{alignat*}
It follows that the \(\delta\)-covering number \(N(\delta,\mathcal{G}^m_B,\lVert\cdot\rVert_\infty)\) with respect to the supremum norm is bounded by the number of distinct possibilities for \(\{A_{l,p}(g):l=1,...,L,g\in\mathcal{G}^m_B,p\in\mathbb{N}^d_0,[p]\leq m-1\}\). 

\begin{proof}[Proof of Theorem \ref{Tentropyassouad}]
	Let \(x_{(l)}\) be ordered so that for \(1<l\leq L\), \(\lVert x_{(l')}-x_{(l)}\rVert\leq\Delta\) for some \(l'<l\). Suppose \(g\in\mathcal{G}^m_B\). For each \(l=1,...,L\) and \(p\in\mathbb{N}^d_0\) with \([p]\leq m-1\), we write \(\mathcal{A}_{l,p}(g)\) for the number of possibilities of \(A_{l,p}(g)\), and for each \(l=1,...,L\), we write \(\mathcal{A}_l(g)\) for the number of possibilities of \(A_{l,p}(g)\) as \(p\in\mathbb{N}^d_0\) varies with \([p]\leq m-1\). For \(l=1\), we have \(D^pg(x_{(1)})\in B\)	for each \(p\in\mathbb{N}^d_0\) with \([p]\leq m-1\). So
	\[\mathcal{A}_{1,p}(g)\leq N_{[p]}=N\left(\frac{1}{4e^d}\delta^{\frac{m-[p]}{m}}(4K_1)^{\frac{[p]}{m}},B,\lVert\cdot\rVert_\mathcal{Y}\right)\leq N\left(\frac{\delta}{4e^d},B,\lVert\cdot\rVert_\mathcal{Y}\right),\]
	where the last upper bound follows since  \(N(\cdot,B,\lVert\cdot\rVert_\mathcal{Y})\) is a decreasing function, and we have \(K_1\geq1\) and \(0<\delta<1\). This upper bound has no dependence on \(p\). The number of different \(p\in\mathbb{N}^d_0\) with \([p]\leq m-1\) is equal to \(\binom{m+d-1}{d}\), which is bounded above by \(m^d\), and so \(\mathcal{A}_1(g)\leq N(\frac{\delta}{4e^d},B,\lVert\cdot\rVert_\mathcal{Y})^{m^d}\). Since \(B=B\cap\mathcal{B}(0,K_B)\) is \((M,\tau_\text{asd})\)-homogeneous, \(N(\frac{\delta}{4e^d},B,\lVert\cdot\rVert_\mathcal{Y})\leq M(\frac{4e^dK_B}{\delta})^{\tau_\text{asd}}\),	and so
	\[\mathcal{A}_1(g)\leq M^{m^d}\left(\frac{4e^dK_B}{\delta}\right)^{\tau_\text{asd}m^d}.\tag{\(\dagger\dagger\)}\]
	Now, for \(1<l\leq L\), suppose that \(A_{l',q}(g)\) is given for all \(l'<l\) and all \(q\in\mathbb{N}^d_0\) with \([q]\leq m-1\). Choose \(l'<l\) such that \(\lVert x_{(l')}-x_{(l)}\rVert\leq\Delta\), and write \(y_{l,p}(g)\vcentcolon=\sum_{[q]\leq m-1-[p]}\frac{(x_{(l')}-x_{(l)})^q}{q!}a_{l',p+q}(g)\). Then for any \(p\in\mathbb{N}_0^d\) with \([p]\leq m-1\), (*) tells us that
	\begin{alignat*}{2}
		&\left\lVert D^pg(x_{(l)})-y_{l,p}(g)\right\rVert_\mathcal{Y}\\
		&=\left\lVert R_p(g,x_{(l')},x_{(l)}-x_{(l')})\right\rVert_\mathcal{Y}+\sum_{[q]\leq m-1-[p]}\frac{\lVert x_{(l')}-x_{(l)}\rVert^q}{q!}\left\lVert D^{p+q}g(x_{(l')})-a_{l',p+q}(g)\right\rVert_\mathcal{Y}\\
		&\leq K_1\Delta^{m-[p]}+\sum_{[q]\leq m-1-[p]}\delta_{[p+q]}\frac{\Delta^q}{q!}=K_1\frac{\Delta^m}{\Delta^{[p]}}+\delta_{[p]}\sum_{k=0}^{m-1-[p]}\delta_k\Delta^k\left(\sum_{[q]=k}\frac{1}{q!}\right)\leq\frac{e^d+1}{2}\delta_{[p]}.
	\end{alignat*}
	As \(a_{l',p+q}(g)\) is given for all \([q]\leq m-1-[p]\), \(y_{l,p}(g)\) is a fixed point in \(\mathcal{Y}\). So \(\mathcal{A}_{l,p}(g)\) is bounded by the number of sets in \(\mathscr{A}_{[p]}\) that intersect with \(B_{l,p}(g)\vcentcolon=B\cap\mathcal{B}\left(y_{l,p}(g),\frac{e^d+1}{2}\delta_{[p]}\right)\). Define \(\mathscr{A}_{l,p}(g)\vcentcolon=\{A\in\mathscr{A}_{[p]}:A\cap B_{l,p}(g)=\emptyset\}\) and \(\mathscr{A}'_{l,p}(g)\vcentcolon=\{A\in\mathscr{A}_{[p]}:A\cap B_{l,p}(g)\neq\emptyset\}\), so that \(\mathscr{A}_{[p]}=\mathscr{A}_{l,p}(g)\cup\mathscr{A}'_{l,p}(g)\), \(N_{[p]}=\lvert\mathscr{A}_{[p]}\rvert=\lvert\mathscr{A}_{l,p}(g)\rvert+\lvert\mathscr{A}'_{l,p}(g)\rvert\) and \(\mathcal{A}_{l,p}(g)\leq\lvert\mathscr{A}'_{l,p}(g)\rvert\). Now, write \(B^+_{l,p}(g)\vcentcolon=B\cap\mathcal{B}(y_{l,p}(g),\frac{e^d+3}{2}\delta_{[p]})\). Then we have \(A\subset B^+_{l,p}(g)\) for all \(A\in\mathscr{A}'_{l,p}(g)\). Let \(\mathscr{A}^+_{l,p}(g)\) be a \(\frac{1}{2}\delta_{[p]}\)-cover of \(B^+_{l,p}(g)\) with minimal cardinality \(N(\frac{1}{2}\delta_{[p]},B^+_{l,p}(g),\lVert\cdot\rVert_\mathcal{Y})\). Since \(B\) is \((M,\tau_\text{asd})\)-homogeneous, \(N(\frac{1}{2}\delta_{[p]},B^+_{l,p}(g),\lVert\cdot\rVert_\mathcal{Y})\leq M(e^d+3)^{\tau_\text{asd}}\). By taking the union \(\mathscr{A}^+_{l,p}(g)\) with \(\mathscr{A}_{l,p}(g)\), we have a \(\frac{1}{2}\delta_{[p]}\)-cover of \(B\) with cardinality at most \(\lvert\mathscr{A}_{l,p}(g)\rvert+M(e^d+3)^{\tau_\text{asd}}\). So if \(\lvert\mathscr{A}'_{l,p}(g)\rvert>M(e^d+3)^{\tau_\text{asd}}\), then we have found a \(\frac{1}{2}\delta_{[p]}\)-cover of \(B\) with cardinality strictly less than \(N_{[p]}\), contradicting its minimality. Hence, we must have \(\mathcal{A}_{l,p}(g)\leq\lvert\mathscr{A}'_{l,p}(g)\rvert\leq M(e^d+3)^{\tau_\text{asd}}\). But the latter quantity is a constant that does not depend on \(\delta\) or \(p\). Thus
	\[\mathcal{A}_l(g)\leq\prod_{[p]\leq m-1}\mathcal{A}_{l,p}(g)\leq M^{m^d}\left(e^d+3\right)^{\tau_\text{asd}m^d}.\tag{\(\dagger\dagger\dagger\)}\]
	Putting together (\(\dagger\)), (\(\dagger\dagger\)) and (\(\dagger\dagger\dagger\)), we arrive at
	\[N\left(\delta,\mathcal{G}^m_B,\lVert\cdot\rVert_\infty\right)\leq\prod^L_{l=1}\mathcal{A}_l(g)\leq M^{m^d}\left(\frac{4e^dK_B}{\delta}\right)^{\tau_\text{asd}m^d}M^{m^dK_2\delta^{-\frac{d}{m}}}\left(e^d+3\right)^{\tau_\text{asd}m^dK_2\delta^{-\frac{d}{m}}},\]
	and so
	\begin{alignat*}{2}
		H\left(\delta,\mathcal{G}^m_B,\lVert\cdot\rVert_\infty\right)&\leq\delta^{-\frac{d}{m}}\log\left(M^{m^dK_2}\left(e^d+3\right)^{\tau_\text{asd}m^dK_2}\right)+m^d\log\left(M\left(\frac{4e^dK_B}{\delta}\right)^{\tau_\text{asd}}\right)\\
		&\leq K\delta^{-\frac{d}{m}},
	\end{alignat*}
	where \(K\) is a constant depending on \(M,m,d,K_2,\tau_\text{asd}\) and \(K_B\). With the second term, we bounded \(\log\left(\frac{1}{\delta}\right)\) by a constant times \(\delta^{-\frac{d}{m}}\). 
\end{proof}
\begin{proof}[Proof of Theorem \ref{Tentropybox}]
	Suppose \(g\in\mathcal{G}^m_B\). With notation as in the proof of Theorem \ref{Tentropyassouad},  for each \(l=1,...,L\) and each \(p\in\mathbb{N}^d_0\) with \([p]\leq m-1\), we have
	\[\mathcal{A}_{l,p}(g)\leq N_{[p]}=N\left(\frac{\delta}{2\Delta^{[p]}e^d},B,\lVert\cdot\rVert_\mathcal{Y}\right)\leq N\left(\frac{\delta}{2e^d},B,\lVert\cdot\rVert_\mathcal{Y}\right)\leq\left(\frac{\delta}{2e^d}\right)^{-(\tau_\text{box}+1)},\]
	where the second last upper bound follows since \(N(\cdot,B,\lVert\cdot\rVert_\mathcal{Y})\) is a decreasing function, and we have \(K_1\geq1\) and \(0<\delta<1\), and the last upper bound follows from Equation (box) in Section \ref{SSmaths}. This upper bound has no dependence on \(l\) or \(p\). So for each \(l=1,...,L\), \(\mathcal{A}_l(g)\leq\left(\frac{2e^d}{\delta}\right)^{(\tau_\text{box}+1)m^d}\). Putting this together with (\(\dagger\)), we arrive at
	\[N\left(\delta,\mathcal{G}^m_B,\lVert\cdot\rVert_\infty\right)\leq\prod^L_{l=1}\mathcal{A}_l(g)\leq\left(\frac{2e^d}{\delta}\right)^{(\tau_\text{box}+1)m^dK_2\delta^{-\frac{d}{m}}},\]
	and so
	\[H\left(\delta,\mathcal{G}^m_B,\lVert\cdot\rVert_\infty\right)\leq(\tau_\text{box}+1)m^dK_2\delta^{-\frac{d}{m}}\log\left(\frac{2e^d}{\delta}\right)\leq K\delta^{-\frac{d}{m}}\log\left(\frac{1}{\delta}\right),\]
	where \(K\) is a constant depending on \(m,d,K_2\) and \(\tau_\text{box}\). 
\end{proof}
\begin{proof}[Proof of Theorem \ref{Tentropyexponential}]
	Suppose \(g\in\mathcal{G}^m_B\). With notation as in the proof of Theorem \ref{Tentropyassouad}, for each \(l=1,...,L\) and each \(p\in\mathbb{N}^d_0\) with \([p]\leq m-1\), we have
	\[\mathcal{A}_{l,p}(g)\leq N_{[p]}=N\left(\frac{\delta}{2\Delta^{[p]}e^d},B,\lVert\cdot\rVert_\mathcal{Y}\right)\leq N\left(\frac{\delta}{2e^d},B,\lVert\cdot\rVert_\mathcal{Y}\right)\leq\exp\left\{M\left(\frac{\delta}{2e^d}\right)^{-\tau_\text{exp}}\right\},\]
	where the second last upper bound follows since \(N(\cdot,B,\lVert\cdot\rVert_\mathcal{Y})\) is a decreasing function, and we have \(K_1\geq1\) and \(0<\delta<1\). This upper bound has no dependence on \(l\) or \(p\). So for each \(l=1,...,L\),
	\[\mathcal{A}_l(g)\leq\exp\left\{M\left(\frac{\delta}{2e^d}\right)^{-\tau_\text{exp}}m^d\right\}.\]
	Putting this together with (\(\dagger\)), we arrive at
	\[N\left(\delta,\mathcal{G}^m_B,\lVert\cdot\rVert_\infty\right)\leq\prod^L_{l=1}\mathcal{A}_l(g)\leq\exp\left\{M\left(\frac{\delta}{2e^d}\right)^{-\tau_\text{exp}}m^dK_2\delta^{-\frac{d}{m}}\right\},\]
	and so
	\[H\left(\delta,\mathcal{G}^m_B,\lVert\cdot\rVert_\infty\right)\leq M\left(\frac{1}{2e^d}\right)^{-\tau_\text{exp}}m^dK_2\delta^{-\frac{d}{m}-\tau_\text{exp}}\leq K\delta^{-\left(\frac{d}{m}+\tau_\text{exp}\right)},\]
	where \(K\) is a constant depending on \(m,d,M,K_2\) and \(\tau_\text{exp}\). 
\end{proof}
\section{Applications to Statistical Learning Theory}\label{Sstatisticallearningtheory}
In this Section, we discuss the application of the above main results to statistical learning theory. 

\subsection{Least-Squares Regression with Fixed Design}\label{SSpeeling}
We first consider problem of least squares regression with fixed design, whereby the covariates \(x_1,...,x_n\in\mathcal{X}\) are considered fixed. Let \(Y_1,...,Y_n\) be random variables taking values in \(\mathcal{Y}\) satisfying
\[Y_i=g_0(x_i)+\varepsilon_i,\qquad i=1,...,n,\]
where \(\varepsilon_i\) are independent (Hilbert space) Gaussian noise terms with zero mean and covariance with trace 1 (see Section \ref{SSgaussian} for details), and \(g_0\) is the unknown regression function in a given class \(\mathcal{G}\) of functions \(\mathcal{X}\rightarrow\mathcal{Y}\). We assume that the following least squares estimator exists:
\[\hat{g}_n\vcentcolon=\argmin_{g\in\mathcal{G}}\frac{1}{n}\sum^n_{i=1}\left\lVert Y_i-g(x_i)\right\rVert_\mathcal{Y}^2,\]
and we are interested in the convergence of \(\lVert\hat{g}_n-g_0\rVert_{2,P_n}\) to 0. Theorem \ref{Tleastsquares} in Appendix \ref{SSpeelingappendix}, whose proof is based on the \say{peeling device} \citep{vandegeer2000empirical} and concentration of Gaussian measures in Hilbert spaces as discussed in Section \ref{SconcentrationHilbfull}, tells us that \(\lVert\hat{g}_n-g_0\rVert_{2,P_n}=\mathcal{O}_P(\delta_n)\), with \(\delta_n\) satisfying
\[\sqrt{n}\delta_n^2\geq8\left(J(\delta_n)+4\delta_n\sqrt{1+t}+\delta_n\sqrt{8t/3}\right),\]
where \(J(\delta)\vcentcolon=4\int^\delta_0\sqrt{2H(u,\mathcal{B}_{2,P_n}(\delta),\lVert\cdot\rVert_{2,P_n})}du\) and \(\mathcal{B}_{2,P_n}(\delta)\vcentcolon=\{g\in\mathcal{G}:\left\lVert g\right\rVert_{2,P_n}\leq\delta\}\). 

As an example, let us return to the setting of Example \ref{Esmooth}, where \(\mathcal{Y}=L^2(\mathcal{X}',P';\mathbb{R})\), and \(B\subset\mathcal{Y}\) is a class of \(m'\)-times differentiable functions. We saw that \(H(\delta,\mathcal{G}^m_B,\lVert\cdot\rVert_\infty)\leq K\delta^{-(\frac{d}{m}+\frac{d'}{m'})}\) by Theorem \ref{Tentropyexponential}. Thus, for another constant \(K'>0\), \(J(\delta)\leq K'\delta^{1-\frac{1}{2}(\frac{d}{m}+\frac{d'}{m'})}\), and it can be shown that
\[\left\lVert\hat{g}_n-g_0\right\rVert_{2,P_n}=\mathcal{O}_P(n^{-1/(2+\frac{d}{m}+\frac{d'}{m'})}).\]
For smooth real-valued function classes, the rate is \(n^{-1/(2+\frac{d}{m})}\) \citep[p.40, Theorem 1.6]{tsybakov2008introduction}, so we can see that the terms \(\frac{d}{m}\) and \(\frac{d'}{m'}\) that correspond to the smoothness of \(\mathcal{G}^m_B\) and \(B\) simply add up in the exponent. Note that as \(m\rightarrow\infty\) and \(m'\rightarrow\infty\), we have \(\left\lVert\hat{g}_n-g_0\right\rVert_{2,P_n}=\mathcal{O}_P(n^{-\frac{1}{2}})\). 

\subsection{Empirical Risk Minimisation with Lipschitz Loss with Random Design}\label{SSrandomdesign}
We discuss the random design setting with bounded \(c\)-Lipschitz loss function \(\mathcal{L}:\mathcal{Y}\times\mathcal{Y}\rightarrow\mathbb{R}\). The population and empirical risks for \(g\in\mathcal{G}\) are given by \(\mathcal{R}(g)=\mathbb{E}[\mathcal{L}(Y,g(X))]\) and \(\hat{\mathcal{R}}_n(g)=\frac{1}{n}\sum^n_{i=1}\mathcal{L}(Y_i,g(X_i))\) respectively. We assume that the empirical risk minimiser \(\hat{g}_n=\argmin_{g\in\mathcal{G}}\hat{\mathcal{R}}_n(g)\) exists, and are interested in the convergence of \(\mathcal{R}(\hat{g}_n)-\mathcal{R}(g^*)\) to 0. The following decomposition is well-known, e.g. \citet[p.329, Eqn. (26.10)]{shalev2014understanding}:
\[\mathcal{R}(\hat{g}_n)-\mathcal{R}(g^*)\leq\sup_{g\in\mathcal{G}}\left\{\mathcal{R}(g)-\hat{\mathcal{R}}_n(g)\right\}+\hat{\mathcal{R}}_n(g^*)-\mathcal{R}(g^*).\]
For the last two terms, we can apply Hoeffding's inequality for real random variables (Proposition \ref{Phoeffdingreal}), since \(g^*\in\mathcal{G}\) is fixed. For the supremum, the class
\(\mathcal{L}\circ\mathcal{G}\vcentcolon=\left\{(x,y)\mapsto\mathcal{L}(y,g(x)):g\in\mathcal{G}\right\}\) of \(\mathcal{X}\times\mathcal{Y}\rightarrow\mathbb{R}\) functions satisfies the entropy contraction property \(H(\delta,\mathcal{L}\circ\mathcal{G},\lVert\cdot\rVert_\infty)\leq H(\frac{1}{c}\delta,\mathcal{G},\lVert\cdot\rVert_\infty)\) (Lemma \ref{Llipschitzentropy}), thanks to the Lipschitz property of \(\mathcal{L}\). If the entropy is at most polynomial in \(\delta\), as is the case for all of the useful cases, then we have \(H(\delta,\mathcal{L}\circ\mathcal{G},\lVert\cdot\rVert_{2,P_n})\leq KH(\delta,\mathcal{G},\lVert\cdot\rVert_{2,P_n})\) for some constant \(K\), which in turn is bounded by \(KH(\delta,\mathcal{G},\lVert\cdot\rVert_\infty)\), which has been the main topic of this paper. Thence, we can apply the usual chaining argument for real-valued function classes to bound \(\sup_{g\in\mathcal{G}}\{\mathcal{R}(g)-\hat{\mathcal{R}}_n(g)\}\) in probability by an expression involving \(H(\delta,\mathcal{G},\lVert\cdot\rVert_\infty)\). 

\subsection{Discussion on Rademacher Complexity for Vector-Valued Function Classes}\label{SSrademacher}

As well as metric entropy, another common measure of complexity of function classes is the Rademacher complexity, the empirical version of which, for real-valued function classes \(\mathcal{F}\), is \(\hat{\mathfrak{R}}_n(\mathcal{F})=\mathbb{E}[\sup_{f\in\mathcal{F}}\lvert\frac{1}{n}\sum^n_{i=1}\sigma_if(X_i)\rvert\mid X_1,...,X_n]\), where \(\sigma_i\) are independent Rademacher variables \citep[Definition 2]{bartlett2002rademacher}. In this Section, we briefly discuss Rademacher complexities for vector-valued function classes, and due to space constraints, defer a fuller discussion to Appendix \ref{SSrademacherappendix}. Define the \say{Rademacher complexity} of a class \(\mathcal{G}\) of vector-valued functions as
\[\hat{\mathfrak{R}}_n(\mathcal{G})=\mathbb{E}\Bigg[\sup_{g\in\mathcal{G}}\Big\lVert\frac{1}{n}\sum^n_{i=1}\sigma_ig(X_i)\Big\rVert_\mathcal{Y}\mid X_1,...,X_n\Bigg].\]
Indeed, in Section \ref{SSsymmetrisation}, we use the symmetrised empirical measure \(\frac{1}{n}\sum^n_{i=1}\sigma_i\delta_{X_i}\), which suggests the use of the above definition. However, there is a critical issue with this definition. Rademacher complexities are almost always used in conjunction with a loss function, i.e. what we end up using is the Rademacher complexity of the class \(\mathcal{L}\circ\mathcal{G}\) (c.f. Section \ref{SSrandomdesign}). With real-valued function classes \(\mathcal{F}\), \citet[p.112, Theorem 4.12]{ledoux1991probability} shows that for bounded Lipschitz losses \(\mathcal{L}\), we have \(\hat{\mathfrak{R}}_n(\mathcal{L}\circ\mathcal{F})\leq K\hat{\mathfrak{R}}_n(\mathcal{F})\) for a constant \(K\), so it is meaningful to work with \(\hat{\mathfrak{R}}_n(\mathcal{F})\). However, the proof makes use of the fact that the output space is \(\mathbb{R}\), and \citet[Section 6]{maurer2016vector} shows via a counterexample that contraction no longer holds for the above definition of \(\hat{\mathfrak{R}}_n(\mathcal{G})\). \citet{maurer2016vector} in fact shows a contraction result for what we call the \textit{coordinate-wise Rademacher complexity}:
\[\hat{\mathfrak{R}}_n^\text{coord}(\mathcal{G})=\mathbb{E}\Bigg[\sup_{g\in\mathcal{G}}\sum^n_{i=1}\sum_k\sigma_i^kg_k(X_i)\mid X_1,...,X_n\Bigg],\]
where a particular basis of \(\mathcal{Y}\) is fixed, \(k\) is the index on the coordinates of \(\mathcal{Y}\) with respect to this basis and \(g_k\) are real-valued functions that map to each coordinate of \(g\). Notice that in this case, we need a separate Rademacher variable for each coordinate, as well as for each sample. This has been used for finite-dimensional multi-task learning \citep{yousefi2018local,li2019learning}. While we recognise its usefulness, the coordinate-wise Rademacher complexity, by definition, relies on a choice of basis of \(\mathcal{Y}\), and we show in Appendix \ref{SSrademacherappendix} that \(\hat{\mathfrak{R}}^\text{coord}_n\) is actually not independent of the choice of basis. We regard this as a critical issue in using \(\hat{\mathfrak{R}}^\text{coord}_n\) as a \say{complexity measure of a function class}, since it is intuitively clear that the complexity should not depend on the choice of basis of the output space. 

A common way to bound the Rademacher complexity is to use Dudley's chaining and uniform entropy condition, in precisely the same manner as in Section \ref{SSchaining}. In this case, we propose a workaround that avoids using either \(\hat{\mathfrak{R}}_n(\mathcal{G})\) or \(\hat{\mathfrak{R}}^\text{coord}_n(\mathcal{G})\). For a bounded Lipschitz loss function \(\mathcal{L}\), \(\hat{\mathfrak{R}}_n(\mathcal{L}\circ\mathcal{G})\) can be bounded by an expression involving the integral (with respect to \(\delta\)) of the entropy \(H(\delta,\mathcal{L}\circ\mathcal{G},\lVert\cdot\rVert_{2,P_n})\) (this is a standard result; see, for example, \citet[p.338, Lemma 27.4]{shalev2014understanding}; we show a vector-valued analogue for \(\hat{\mathfrak{R}}_n(\mathcal{G})\) in Theorem \ref{Trademacherboundentropy}, using vector-valued Hoeffding-type inequality). But as discussed in Section \ref{SSrandomdesign}, the entropies satisfy a simple contraction property given in Lemma \ref{Llipschitzentropy}. So applying the same argument, we can bound \(\hat{\mathfrak{R}}_n(\mathcal{L}\circ\mathcal{G})\) by an expression involving \(H(\delta,\mathcal{G},\lVert\cdot\rVert_\infty)\), which has been the main topic of this paper. This does not contradict the counterexample of \citet[Section 6]{maurer2016vector}, since the latter is the space of linear operators between infinite-dimensional Hilbert spaces, and hence has infinite entropy. 

\section{Discussion \& Future Directions}\label{Sdiscussion}

To summarize, we took some first steps towards establishing a theory of empirical processes for vector-valued functions. In particular, we investigated the metric entropy of smooth functions, by restricting the partial derivatives to take values in totally bounded subsets with specific properties, leveraging theory from fractal geometry, and demonstrated its application in empirical risk minimisation. 

There is a plethora of possible future research directions. Considering other classes of functions than those of smooth functions is a natural next step. Also, we let \(\mathcal{Y}\) be a Hilbert space, primarily because some simplifications occur for Hoeffding's inequality and Gaussian measures (Appendix \ref{SconcentrationHilbfull}), but extensions to Banach spaces should be possible. Moreover, we used compact subsets of \(\mathbb{R}^d\) as our input space due to the ease in considering partial derivatives, but interesting applications exist for which the input space \(\mathcal{X}\) is a subset of an infinite-dimensional space \citep{li2020neural,nelsen2021random,lu2021learning}. On the more theoretical side, measurability questions for empirical processes and uniform central limit theorems in vector spaces are interesting questions. With empirical risk minimisation, extensions to more general noise with vector-valued Bernstein's inequality or misspecified models are important. 

\section*{Acknowledgments}
We thank Shashank Singh and SImon Buchholz at the Max Planck Institute for Intelligent Systems, T\"ubingen for their helpful comments to the initial drafts. 

This work was inspired by the lecture course \say{Empirical Process Theory and Applications} given by Sara van de Geer at Seminar f\"ur Statistik, ETH Z\"urich. JP is extremely grateful to Sara van de Geer for the lectures, as well as readily engaging in post-lecture discussions. 

\bibliography{ref}
\bibliographystyle{unsrtnat}
\vspace{.7cm}

\newpage
\appendix

The appendix is structured as follows. In empirical process theory, concentration inequalities are an essential tool, and in Appendix \ref{SconcentrationHilbfull}, we state and prove concentration inequalities in Hilbert spaces, in particular, the extensions of Hoeffding's inequality and Gaussian measures to Hilbert spaces. In Appendix \ref{Sdiffcalcfull}, we develop the theory of differential calculus between Banach spaces, which plays a vital role in our paper in considering smooth functions. In Appendix \ref{Sempiricalprocessesfull}, we develop the theory of empirical process theory for vector-valued functions, briefly introduced in Section \ref{Sempiricalprocesstheory}, in more detail. In particular, we develop the symmetrisation technique (Appendix \ref{SSsymmetrisation}); we establish the uniform law of large numbers for vector-valued function classes with bounded entropy (Appendix \ref{SSulln}); we develop the chaining technique for vector-valued functions and use it to establish asymptotic equicontinuity of empirical processes satisfying uniform entropy condition (Appendix \ref{SSchaining}); we use the chaining technique in tandem with the peeling device for least-squares regression, as briefly discussed in Section \ref{SSpeeling} (Appendix \ref{SSpeelingappendix}); and finally, we discuss in full connections with the popular Rademacher complexity, as briefly touched upon in Section \ref{SSrademacher} (Appendix \ref{SSrademacherappendix}).

\section{Concentration Inequalities in Hilbert Spaces}\label{SconcentrationHilbfull}
First, we state Markov's inequality, on which all subsequent results are based. 
\begin{proposition}[Markov's inequality]\label{Pmarkov}
	For any non-negative real random variable \(Z\) and \(a>0\), \(\mathbb{P}\left(Z\geq a\right)\leq\frac{\mathbb{E}\left[Z\right]}{a}\). 
\end{proposition}
\begin{proof}
	See that \(\mathbb{E}\left[Z\right]=\mathbb{E}\left[Z\mathbf{1}_{Z\geq a}+Z\mathbf{1}_{Z<a}\right]\geq\mathbb{E}\left[Z\mathbf{1}_{Z\geq a}\right]\geq a\mathbb{E}\left[\mathbf{1}_{Z\geq a}\right]=a\mathbb{P}\left(Z\geq a\right)\). 
\end{proof}
\subsection{Hoeffding's inequality}\label{SShoeffding}
Hoeffding's inequality is a concentration result for sums of bounded random variables. We first state the real version, due to \citet{hoeffding1963probability}. 
\begin{proposition}[Hoeffding's inequality]\label{Phoeffdingreal}
	Let \(Z_1,...,Z_n\) be independent real random variables such that for all \(i\), \(\mathbb{E}[Z_i]=0\) and \(\lvert Z_i\rvert\leq c_i\) almost surely for some constants \(c_i>0\). Then writing \(S_n=\sum^n_{i=1}Z_i\) and \(b^2=\sum^n_{i=1}c_i^2\),
	\[\mathbb{P}\left(S_n\geq a\right)\leq e^{-\frac{a^2}{2b^2}}\]
	for any \(a>0\), or reformulated, for any \(t>0\),
	\[\mathbb{P}\left(S_n\geq b\sqrt{2t}\right)\leq e^{-t}.\]
\end{proposition}
\begin{proof}
	Let \(\lambda>0\), and define, for each \(i\), \(A_i=\frac{c_i-Z_i}{2c_i}\). Then almost surely, \(0\leq A_i\leq1\), \(Z_i=A_i(-c_i)+(1-A_i)c_i\), and by the convexity of the function \(u\mapsto e^{\lambda u}\), we have \(e^{\lambda Z_i}\leq A_ie^{-\lambda c_i}+(1-A_i)e^{\lambda c_i}\). But then \(\mathbb{E}[A_i]=\frac{1}{2}\), so
	\[\mathbb{E}\left[e^{\lambda Z_i}\right]\leq\frac{1}{2}e^{-\lambda c_i}+\frac{1}{2}e^{\lambda c_i}=\sum^\infty_{k=0}\frac{\left(\lambda c_i\right)^{2k}}{(2k)!}\leq\sum^\infty_{k=0}\frac{\left(\lambda c_i\right)^{2k}}{2^kk!}=e^{\frac{1}{2}\left(\lambda c_i\right)^2}.\]
	Now see that, by Markov's inequality,
	\[\mathbb{P}\left(S_n\geq a\right)=\mathbb{P}\left(e^{\lambda S_n}\geq e^{\lambda a}\right)\leq e^{-\lambda a}\mathbb{E}\left[e^{\lambda S_n}\right]\leq e^{-\lambda a}e^{\frac{\lambda^2b^2}{2}}.\]
	Now let \(\lambda=\frac{a}{b^2}\) for the first inequality, and then take \(a=b\sqrt{2t}\) for the reformulation. 
\end{proof}
In \citet{pinelis1992approach}, Hoeffding's inequality was extended to martingales in Banach spaces with certain smoothness properties (see also \citet[Eqn. (3)]{rosasco2010learning} and \citet[p.217, Corollary 6.15]{steinwart2008support}). As we only require the result for sums of independent random variables taking values in a separable Hilbert space \(\mathcal{Y}\), we give the corresponding simplified statement and proof. First, we state its expectation form, then use it to prove the probability inequality. 
\begin{proposition}\label{Phoeffdinghilbertexpectation}
	Let \(Y_1,...,Y_n\) be independent random variables in \(\mathcal{Y}\), such that for all \(i\), \(\mathbb{E}[Y_i]=0\) and \(\lVert Y_i\rVert_\mathcal{Y}\leq c_i\) almost surely for some constants \(c_i>0\). Then writing \(S_n=\sum^n_{i=1}Y_i\), for any \(\lambda>0\), we have
	\[\mathbb{E}\left[\cosh\left(\lambda\left\lVert S_n\right\rVert_\mathcal{Y}\right)\right]\leq\prod^n_{i=1}e^{\lambda^2c^2_i}.\]
\end{proposition}
\begin{proof}
	Denote by \(\mathcal{F}_0=\{\emptyset,\Omega\}\) the trivial \(\sigma\)-algebra, and for each \(i=1,...,n\), let \(\mathcal{F}_i=\sigma(Y_1,...,Y_i)\), the \(\sigma\)-algebra generated by \(Y_1,...,Y_i\). For \(\lambda>0\), consider the real-valued stochastic process \(F_\lambda(t)\) indexed by \(t\in\mathbb{R}\) given by
	\[F_\lambda(t)=\mathbb{E}\left[\cosh\left(\lambda\left\lVert S_{n-1}+tY_n\right\rVert_\mathcal{Y}\right)\mid\mathcal{F}_{n-1}\right].\]
	If we define maps \(H:\mathbb{R}\rightarrow\mathcal{Y}\) and \(J:\mathcal{Y}\rightarrow\mathbb{R}\) by \(H(t)=tY_n\) and \(J(y)=\lambda\left\lVert S_{n-1}+y\right\rVert_\mathcal{Y}\) respectively, the derivative of \(F_\lambda\) with respect to \(t\) can be calculated from the chain rule as
	\[F'_\lambda(t)=\mathbb{E}\left[(J\circ H)'(t)\sinh\left(\lambda\left\lVert S_{n-1}+tY_n\right\rVert_\mathcal{Y}\right)\mid\mathcal{F}_{n-1}\right].\]
	Now, \citet[p.100, Example 7.3]{precup2002methods} tells us that \((J\circ H)'(t)=(H^*\circ J'\circ H)(t)\). We can easily compute the adjoint \(H^*(y)=\langle y,Y_n\rangle_\mathcal{Y}\) and the Fr\'echet derivative \(J'(y)=\frac{\lambda S_{n-1}+\lambda y}{\left\lVert S_{n-1}+y\right\rVert_\mathcal{Y}}\), so we have 
	\[F'_\lambda(t)=\mathbb{E}\left[\left\langle Y_n,\frac{\lambda S_{n-1}+\lambda tY_n}{\left\lVert S_{n-1}+tY_n\right\rVert_\mathcal{Y}}\right\rangle_\mathcal{Y}\sinh\left(\lambda\left\lVert S_{n-1}+tY_n\right\rVert_\mathcal{Y}\right)\mid\mathcal{F}_{n-1}\right].\]
	Then see that \(F'_\lambda(0)=0\):
	\begin{alignat*}{2}
		F'_\lambda(0)=\sinh\left(\lambda\left\lVert S_{n-1}\right\rVert_\mathcal{Y}\right)\left\langle\mathbb{E}\left[Y_n\mid\mathcal{F}_{n-1}\right],\frac{\lambda S_{n-1}}{\left\lVert S_{n-1}\right\rVert_\mathcal{Y}}\right\rangle_\mathcal{Y}=0,
	\end{alignat*}
	since, by independence and zero-mean assumptions, \(\mathbb{E}[Y_n\mid\mathcal{F}_{n-1}]=0\). We want to compute the second derivative with respect to \(t\), and in order to do so, we need the derivative of the first factor in the expectation of \(F'_\lambda\). Define a map \(K:\mathcal{Y}\rightarrow\mathbb{R}\) by \(K(y)=\langle Y_n,S_{n-1}+y\rangle_\mathcal{Y}\). Then the Fr\'echet derivative of \(K\) can easily be seen to be \(K'(y)=Y_n\). Then using the quotient rule,
	\begin{alignat*}{2}
		\frac{d}{dt}\left\langle Y_n,\frac{\lambda S_{n-1}+\lambda tY_n}{\left\lVert S_{n-1}+tY_n\right\rVert_\mathcal{Y}}\right\rangle_\mathcal{Y}&=\frac{\lambda\left\lVert Y_n\right\rVert^2_\mathcal{Y}}{\left\lVert S_{n-1}+tY_n\right\rVert_\mathcal{Y}}-\frac{\left\langle Y_n,S_{n-1}+tY_n\right\rangle_\mathcal{Y}}{\left\lVert S_{n-1}+tY_n\right\rVert_\mathcal{Y}^2}\left\langle Y_n,\frac{\lambda S_{n-1}+\lambda tY_n}{\left\lVert S_{n-1}+tY_n\right\rVert_\mathcal{Y}}\right\rangle_\mathcal{Y}\\
		&\leq\frac{\lambda\left\lVert Y_n\right\rVert^2_\mathcal{Y}}{\left\lVert S_{n-1}+tY_n\right\rVert_\mathcal{Y}}. 
	\end{alignat*}
	Then see that, using the elementary inequality \(\sinh u\leq u\cosh u\),
	\begin{alignat*}{2}
		F''_\lambda(t)&\leq\mathbb{E}\left[\cosh\left(\lambda\left\lVert S_{n-1}+tY_n\right\rVert_\mathcal{Y}\right)\left(\left\langle Y_n,\frac{\lambda S_{n-1}+\lambda tY_n}{\left\lVert S_{n-1}+tY_n\right\rVert_\mathcal{Y}}\right\rangle_\mathcal{Y}^2+\lambda^2\left\lVert Y_n\right\rVert^2_\mathcal{Y}\right)\mid\mathcal{F}_{n-1}\right]\\
		&\leq\mathbb{E}\left[\cosh\left(\lambda\left\lVert S_{n-1}+tY_n\right\rVert_\mathcal{Y}\right)\left(2\lambda^2\left\lVert Y_n\right\rVert^2_\mathcal{Y}\right)\mid\mathcal{F}_{n-1}\right]\enspace\text{by the Cauchy-Schwarz inequality}\\
		&\leq2\lambda^2c_n^2\mathbb{E}\left[\cosh\left(\lambda\left\lVert S_{n-1}+tY_n\right\rVert_\mathcal{Y}\right)\mid\mathcal{F}_{n-1}\right]\qquad\text{by the almost sure bound on }\left\lVert Y_n\right\rVert_\mathcal{Y}\\
		&=2\lambda^2c_n^2F_\lambda(t).
	\end{alignat*}
	Define \(G_\lambda(t)=\frac{1}{2\lambda^2c^2_n}F''_\lambda(t)-F_\lambda(t)\). Then by the preceding argument, \(G_\lambda(t)\leq0\) for all \(t\in\mathbb{R}\). But consider the differential equation
	\[F''_\lambda(t)=2\lambda^2c^2_n\left(F_\lambda(t)+G_\lambda(t)\right),\qquad F'_\lambda(0)=0.\tag{*}\]
	We claim that \(F(t)=F_\lambda(0)\cosh\left(\sqrt{2}\lambda c_nt\right)+\int^{\sqrt{2}\lambda c_nt}_0G_\lambda\left(\frac{s}{\sqrt{2}\lambda c_n}\right)\sinh\left(\sqrt{2}\lambda c_nt-s\right)ds\) solves the differential equation (*). Indeed, we clearly have \(F(0)=F_\lambda(0)\); further, we have
	\[F'(t)=\sqrt{2}\lambda c_nF_\lambda(0)\sinh\left(\sqrt{2}\lambda c_nt\right)+\sqrt{2}\lambda c_n\int^{\sqrt{2}\lambda c_nt}_0G_\lambda\left(\frac{s}{\sqrt{2}\lambda c_n}\right)\cosh\left(\sqrt{2}\lambda c_nt-s\right)ds\]
	which clearly satisfies \(F'(0)=0\); and finally, 
	\begin{alignat*}{2}
		F''(t)&=2\lambda^2c_n^2F_\lambda(0)\cosh\left(\sqrt{2}\lambda c_nt\right)\\
		&\qquad+2\lambda^2c_n^2\int^{\sqrt{2}\lambda c_nt}_0G_\lambda\left(\frac{s}{\sqrt{2}\lambda c_n}\right)\sinh\left(\sqrt{2}\lambda c_nt-s\right)ds+2\lambda^2c_n^2G_\lambda(t)\\
		&=2\lambda^2c^2_n\left(F(t)+G_\lambda(t)\right),
	\end{alignat*}
	Hence this \(F\) is the solution to (*), and hence we have
	\begin{alignat*}{2}
		F_\lambda(1)&=F_\lambda(0)\cosh\left(\sqrt{2}\lambda c_n\right)+\int^{\sqrt{2}\lambda c_n}_0G_\lambda\left(\frac{s}{\sqrt{2}\lambda c_n}\right)\sinh\left(\sqrt{2}\lambda c_n-s\right)ds\\
		&\leq F_\lambda(0)\cosh\left(\sqrt{2}\lambda c_n\right)\qquad\text{since }G_\lambda\leq0\\
		&\leq F_\lambda(0)e^{\lambda^2c_n^2}\qquad\text{by }\cosh u\leq e^{\frac{1}{2}u^2}.
	\end{alignat*}
	Now see that
	\begin{alignat*}{3}
		\mathbb{E}\left[\cosh\left(\lambda\left\lVert S_n\right\rVert_\mathcal{Y}\right)\right]&=\mathbb{E}\left[F_\lambda(1)\right]&&\text{by the law of iterated expectations}\\
		&\leq e^{\lambda^2c_n^2}\mathbb{E}\left[\cosh\left(\lambda\left\lVert S_{n-1}\right\rVert_\mathcal{Y}\right)\right]&&\text{by above}\\
		&\leq\prod^n_{i=1}e^{\lambda^2c_i^2}&&\text{by the same argument on }i=1,...,n-1.
	\end{alignat*}
\end{proof}
\begin{proposition}[Hoeffding's inequality in Hilbert spaces]\label{Phoeffdinghilbert}
	Let \(Y_1,...,Y_n\) be independent random variables in \(\mathcal{Y}\), such that for all \(i\), \(\mathbb{E}[Y_i]=0\) and \(\lVert Y_i\rVert_\mathcal{Y}\leq c_i\) almost surely for some constants \(c_i>0\). Then writing \(S_n=\sum^n_{i=1}Y_i\) and \(b^2=\sum^n_{i=1}c_i^2\),
	\[\mathbb{P}\left(\left\lVert S_n\right\rVert_\mathcal{Y}\geq a\right)\leq2e^{-\frac{a^2}{4b^2}}\]
	for any \(a>0\), or reformulated, for any \(t>0\),
	\[\mathbb{P}\left(\left\lVert S_n\right\rVert_\mathcal{Y}\geq2b\sqrt{t}\right)\leq2e^{-t}.\]
\end{proposition}
\begin{proof}
	See that
	\begin{alignat*}{3}
		\mathbb{P}\left(\left\lVert S_n\right\rVert_\mathcal{Y}\geq a\right)&=\mathbb{P}\left(\cosh\left(\lambda\left\lVert S_n\right\rVert_\mathcal{Y}\right)\geq\cosh\left(\lambda a\right)\right)\\
		&\leq\frac{1}{\cosh\left(\lambda a\right)}\mathbb{E}\left[\cosh\left(\lambda\left\lVert S_n\right\rVert_\mathcal{Y}\right)\right]&&\text{by Markov's inequality}\\
		&\leq\frac{1}{\cosh\left(\lambda a\right)}\prod^n_{i=1}e^{\lambda^2c^2_i}&&\text{by Proposition \ref{Phoeffdinghilbertexpectation}}\\
		&\leq2e^{\lambda^2b^2-\lambda a}&&\text{since }\cosh u\geq\frac{1}{2}e^u.
	\end{alignat*}
	Now let \(\lambda=\frac{a}{2b^2}\) for the first inequality, and then take \(a=2b\sqrt{t}\) for the reformulation. 
\end{proof}
\subsection{Gaussian Measures in Hilbert Spaces}\label{SSgaussian}
Next, we consider concentration of the Gaussian measure. In the real case, the Gaussian measure with mean \(\mu\) and variance \(q\) is defined as the measure that is absolutely continuous with respect to the Lebesgue measure and has density \(\frac{1}{\sqrt{2\pi q}}e^{-\frac{1}{2q}\left(x-\mu\right)^2}\). The Gaussian measure with mean 0 and variance 1 is called the standard Gaussian measure. For a real variable with the standard Gaussian distribution, the following concentration inequality can easily be derived.
\begin{lemma}\label{Lconcentrationrealgaussian}
	Let \(Z\) have the standard Gaussian distribution. Then for any \(a>0\), \(\mathbb{P}(Z\geq a)\leq e^{-\frac{1}{2}a^2}\). 
\end{lemma}
\begin{proof}
	See that, for any \(a>0\), Markov's inequality gives
	\begin{alignat*}{2}
		\mathbb{P}\left(Z\geq a\right)&=\mathbb{P}\left(e^{aZ}\geq e^{a^2}\right)\\
		&\leq e^{-a^2}\mathbb{E}\left[e^{aZ}\right]\\
		&=e^{-a^2}\int_\mathbb{R}\frac{1}{\sqrt{2\pi}}e^{-\frac{1}{2}(z-a)^2+\frac{1}{2}a^2}dz\\
		&=e^{\frac{1}{2}a^2-a^2}\\
		&=e^{-\frac{1}{2}a^2}.
	\end{alignat*}
\end{proof}
The definition of Gaussian measures can be extended to the separable Hilbert space \(\mathcal{Y}\). 
\begin{definition}{\citet[pp.46-47]{da2014stochastic}}]\label{Dgaussianhilbert}
	A random variable \(Y\) in \(\mathcal{Y}\) is Gaussian if, for any \(y\in\mathcal{Y}\), \(\langle Y,y\rangle_\mathcal{Y}\) is a real Gaussian random variable (with some mean and variance). 
\end{definition}
The next two lemmas are concerned with the \textit{mean} and \textit{covariance operator} of a \(\mathcal{Y}\)-valued Gaussian random variable. 
\begin{lemma}\label{Lgaussianintegrable}
	If \(Y\) is a Gaussian random variable in \(\mathcal{Y}\), then \(\mathbb{E}[\lVert Y\rVert^2_\mathcal{Y}]<\infty\). As a consequence, \(Y\) is Bochner integrable, and we call \(\mu=\mathbb{E}\left[Y\right]\in\mathcal{Y}\) the \textnormal{mean} of \(Y\). 
\end{lemma}
\begin{proof}
	First, consider the map \(M:\mathcal{Y}\rightarrow\mathbb{R}\) defined by \(M(y)=\mathbb{E}\left[\left\langle y,Y\right\rangle_\mathcal{Y}\right]\), which is clearly linear. For each \(n\in\mathbb{N}\), we can define the set \(U_n=\left\{y\in\mathcal{Y}:\mathbb{E}\left[\left\lvert\left\langle y,Y\right\rangle_\mathcal{Y}\right\rvert\right]\leq n\right\}\). Since, for each \(y\), \(\left\langle y,Y\right\rangle_\mathcal{Y}\) is Gaussian, \(\mathbb{E}\left[\left\lvert\left\langle y,Y\right\rangle_\mathcal{Y}\right\rvert\right]<\infty\), and so \(\mathcal{Y}=\cup^\infty_{n=1}U_n\). Each \(U_n\) is closed, so by the Baire category theorem \citep[p.76, Theorem 1']{bollobas1999linear}, there exist some \(n_0\in\mathbb{N}\), \(y_0\in U_{n_0}\) and \(r_0>0\) such that
	\[\mathbb{E}\left[\left\lvert\left\langle y_0+y,Y\right\rangle_\mathcal{Y}\right\rvert\right]\leq n_0\qquad\text{for all }y\in\mathcal{Y}\text{ with }\left\lVert y\right\rVert_\mathcal{Y}\leq r_0.\]
	So for any non-zero \(y\in\mathcal{Y}\),
	\[\left\lvert M(y)\right\rvert\leq\mathbb{E}\left[\left\lvert\left\langle y,Y\right\rangle_\mathcal{Y}\right\rvert\right]\leq\frac{\left\lVert y\right\rVert_\mathcal{Y}}{r_0}\mathbb{E}\left[\left\lvert\left\langle y_0,Y\right\rangle_\mathcal{Y}\right\rvert+\left\lvert\left\langle y_0+\frac{r_0y}{\left\lVert y\right\rVert_\mathcal{Y}},Y\right\rangle_\mathcal{Y}\right\rvert\right]\leq\frac{2\left\lVert y\right\rVert_\mathcal{Y}n_0}{r_0},\]
	which implies that \(M\) is continuous. By Riesz representation theorem \citep[p.137, Theorem 9]{bollobas1999linear}, there exists some \(\tilde{\mu}\in\mathcal{Y}\) such that \(M(y)=\left\langle\tilde{\mu},y\right\rangle_\mathcal{Y}\) for all \(y\in\mathcal{Y}\). Note that \(Y-\tilde{\mu}\) is also Gaussian, since, for any \(y\in\mathcal{Y}\), \(\left\langle Y-\tilde{\mu},y\right\rangle_\mathcal{Y}=\left\langle Y,y\right\rangle_\mathcal{Y}-\left\langle\tilde{\mu},y\right\rangle_\mathcal{Y}\) is a real Gaussian variable, and in particular, for any \(y\in\mathcal{Y}\), \(\mathbb{E}\left[\left\langle Y-\tilde{\mu},y\right\rangle_\mathcal{Y}\right]=\mathbb{E}\left[\left\langle Y,y\right\rangle_\mathcal{Y}\right]-\left\langle\tilde{\mu},y\right\rangle_\mathcal{Y}=M(y)-M(y)=0\). 
	
	Now let \(\{e_j\}_{j=1}^\infty\) be an orthonormal basis of \(\mathcal{Y}\). For each \(j\), from above, \(Z_j\vcentcolon=\left\langle Y-\tilde{\mu},e_j\right\rangle_\mathcal{Y}\) is a zero-mean real Gaussian variable, say with variance \(q_j\). Then we have \(\mathbb{E}\left[\left\lVert Y-\tilde{\mu}\right\rVert^2_\mathcal{Y}\right]=\sum^\infty_{j=1}\mathbb{E}\left[Z_j^2\right]=\sum^\infty_{j=1}q_j\). Now, there exists some \(c>0\) such that \(\mathbb{P}\left(\left\lVert Y-\tilde{\mu}\right\rVert_\mathcal{Y}>c\right)\leq\frac{1}{8}\). Then see that, for each \(j=1,2,...\), writing \(c_j=\mathbb{E}\left[Z_j^2\mathbf{1}_{\left\lVert Y-\tilde{\mu}\right\rVert_\mathcal{Y}\leq c}\right]\), the sequence \(\{c_j\}_{j=1}^\infty\) is summable, since
	\[\sum^\infty_{j=1}c_j=\sum^\infty_{j=1}\mathbb{E}\left[Z_j^2\mathbf{1}_{\left\lVert Y-\tilde{\mu}\right\rVert_\mathcal{Y}\leq c}\right]=\mathbb{E}\left[\left\lVert Y-\tilde{\mu}\right\rVert_\mathcal{Y}^2\mathbf{1}_{\left\lVert Y-\tilde{\mu}\right\rVert_\mathcal{Y}\leq c}\right]\leq c^2<\infty.\]
	For each \(j\in\mathbb{N}\) such that \(c_j>0\), \(\tilde{Z}_j=\frac{Z_j}{\sqrt{c_j}}\) is a zero-mean real Gaussian random variable with variance \(\frac{q_j}{c_j}\). Hence, the characteristic function of \(\tilde{Z}_j\) is \(\mathbb{E}\left[e^{i\tilde{Z}_j}\right]=e^{-\frac{q_j}{2c_j}}\), and so
	\begin{alignat*}{2}
		1-e^{-\frac{q_j}{2c_j}}&=\mathbb{E}\left[1-\cos\tilde{Z}_j\right]\\
		&=\mathbb{E}\left[\left(1-\cos\tilde{Z}_j\right)\mathbf{1}_{\left\lVert Y-\tilde{\mu}\right\rVert_\mathcal{Y}\leq c}\right]+\mathbb{E}\left[\left(1-\cos\tilde{Z}_j\right)\mathbf{1}_{\left\lVert Y-\tilde{\mu}\right\rVert_\mathcal{Y}>c}\right]\\
		&\leq\frac{1}{2}\mathbb{E}\left[\tilde{Z}_j^2\mathbf{1}_{\left\lVert Y-\tilde{\mu}\right\rVert_\mathcal{Y}\leq c}\right]+2\mathbb{P}\left(\left\lVert Y-\tilde{\mu}\right\rVert_\mathcal{Y}>c\right)\text{by }1-\cos u\leq\frac{1}{2}u^2\text{ and }1-\cos u\leq2\\
		&=\frac{1}{2c_j}\mathbb{E}\left[Z_j^2\mathbf{1}_{\left\lVert Y-\tilde{\mu}\right\rVert_\mathcal{Y}\leq c}\right]+\frac{1}{4}\quad\text{since }c\text{ was defined to give }\mathbb{P}\left(\left\lVert Y-\tilde{\mu}\right\rVert_\mathcal{Y}>c\right)\leq\frac{1}{8}\\
		&\leq\frac{3}{4}\qquad\qquad\text{by the definition of }c_j
	\end{alignat*}
	\begin{alignat*}{2}
		\implies\qquad&e^{-\frac{q_j}{2c_j}}\geq\frac{1}{4}\\
		\implies\qquad&\frac{q_j}{2c_j}\leq e^{\frac{q_j}{2c_j}}\leq4\\
		\implies\qquad&q_j\leq8c_j.
	\end{alignat*}
	For \(j\in\mathbb{N}\) such that \(c_j=\mathbb{E}\left[Z_j^2\mathbf{1}_{\left\lVert Y-\tilde{\mu}\right\rVert_\mathcal{Y}\leq c}\right]=0\), this means that we have \(\mathbb{P}\left(Z_j=0\right)\geq\mathbb{P}\left(\left\lVert Y-\tilde{\mu}\right\rVert_\mathcal{Y}\leq c\right)\geq\frac{7}{8}\), but \(Z_j\) is Gaussian, so we must have \(Z_j=0\), i.e. \(q_j=0\). Hence, we have \(q_j\leq8c_j\) in this case too, and since \(\{c_j\}_{j=1}^\infty\) is summable,
	\[\mathbb{E}\left[\left\lVert Y-\tilde{\mu}\right\rVert^2_\mathcal{Y}\right]=\sum^\infty_{j=1}q_j\leq8\sum^\infty_{j=1}c^2_j<\infty.\]
	We can now finish the proof by noting that
	\[\mathbb{E}\left[\left\lVert Y\right\rVert^2_\mathcal{Y}\right]\leq2\mathbb{E}\left[\left\lVert Y-\tilde{\mu}\right\rVert^2_\mathcal{Y}\right]+2\left\lVert\tilde{\mu}\right\rVert^2_\mathcal{Y}<\infty,\]
	using the elementary Hilbert space inequality \(\lVert a+b\rVert^2\leq2\lVert a\rVert^2+2\lVert b\rVert^2\). 
\end{proof}
Denote by \(\mathscr{L}(\mathcal{Y})\) the Banach space of continuous linear operators from \(\mathcal{Y}\) into itself, with the operator norm. 
\begin{lemma}\label{Lgaussiancovariance}
	For a \(\mathcal{Y}\)-valued Gaussian variable \(Y\) with mean \(\mu\), the random operator \(\left(Y-\mu\right)\otimes\left(Y-\mu\right):\mathcal{Y}\rightarrow\mathcal{Y}\) defined by \(\left(Y-\mu\right)\otimes\left(Y-\mu\right)(y)=\left\langle Y-\mu,y\right\rangle_\mathcal{Y}\left(Y-\mu\right)\) is continuous and linear, and as a random variable taking values in \(\mathscr{L}(\mathcal{Y})\), is Bochner integrable. We call \(\Phi=\mathbb{E}\left[\left(Y-\mu\right)\otimes\left(Y-\mu\right)\right]\in\mathscr{L}(\mathcal{Y})\) the \textnormal{covariance operator} of \(Y\). The covariance operator \(\Phi\) is self-adjoint and trace-class. 
\end{lemma}
\begin{proof}
	Linearity is obvious. For continuity, see that, for any \(y\in\mathcal{Y}\), the Cauchy-Schwarz inequality gives
	\[\left\lVert\left(Y-\mu\right)\otimes\left(Y-\mu\right)(y)\right\rVert_\mathcal{Y}=\left\lvert\left\langle Y-\mu,y\right\rangle_\mathcal{Y}\right\rvert\left\lVert Y-\mu\right\rVert_\mathcal{Y}\leq\left\lVert Y-\mu\right\rVert_\mathcal{Y}^2\left\lVert y\right\rVert_\mathcal{Y}.\]
	Now see that, by the Cauchy-Schwarz inequality and Lemma \ref{Lgaussianintegrable}, 
	\[\mathbb{E}\left[\left\lVert\left(Y-\mu\right)\otimes\left(Y-\mu\right)\right\rVert_\text{op}\right]=\mathbb{E}\left[\left\lVert Y-\mu\right\rVert_\mathcal{Y}\sup_{y\in\mathcal{Y},\left\lVert y\right\rVert\leq1}\left\lvert\left\langle Y-\mu,y\right\rangle_\mathcal{Y}\right\rvert\right]\leq\mathbb{E}\left[\left\lVert Y-\mu\right\rVert_\mathcal{Y}^2\right]<\infty,\]
	as \(Y-\mu\) is a Gaussian variable in \(\mathcal{Y}\). Clearly, for any \(y_1,y_2\in\mathcal{Y}\),
	\[\left\langle y_1,\Phi y_2\right\rangle_\mathcal{Y}=\mathbb{E}\left[\left\langle y_1,Y-\mu\right\rangle_\mathcal{Y}\left\langle y_2,Y-\mu\right\rangle_\mathcal{Y}\right]=\left\langle \Phi y_1,y_2\right\rangle_\mathcal{Y},\]
	so \(\Phi\) is self-adjoint. Now see that, for any orthonormal basis \(\{e_j\}_{j=1}^\infty\) of \(\mathcal{Y}\),
	\[\text{Tr}\Phi=\sum^\infty_{j=1}\left\langle\Phi e_j,e_j\right\rangle_\mathcal{Y}=\sum^\infty_{j=1}\mathbb{E}\left[\left\langle Y-\mu,e_j\right\rangle_\mathcal{Y}^2\right]=\mathbb{E}\left[\left\lVert Y-\mu\right\rVert_\mathcal{Y}^2\right]<\infty,\]
	so \(Q\) is trace-class. 
\end{proof}
Some authors (e.g. \citet[p.24]{bharucha1972random}) refer to the quantity
\[\text{Tr}\Phi=\mathbb{E}\left[\left\lVert\left(Y-\mu\right)\otimes\left(Y-\mu\right)\right\rVert_\text{op}\right]=\mathbb{E}\left[\left\lVert Y-\mu\right\rVert^2_\mathcal{Y}\right]\]
as the \say{variance} of \(Y\) . 

For a random variable \(Y\) on \(\mathcal{Y}\), its \textit{characteristic function} is defined as the functional \(\varphi_Y:\mathcal{Y}\rightarrow\mathbb{C}\) defined by \(\varphi_Y(y)=\mathbb{E}\left[e^{i\left\langle Y,y\right\rangle_\mathcal{Y}}\right]\) \citep[pp.34-35]{da2014stochastic}. As for real variables, the characteristic function uniquely determines the distribution of the random variable \citep[p.35, Proposition 2.5(i)]{da2014stochastic}. Clearly, the characteristic function of a Gaussian variable \(Y\) with mean \(\mu\) and covariance operator \(\Phi\) is given by
\[\varphi_Y(y)=e^{i\left\langle\mu,y\right\rangle_\mathcal{Y}-\frac{1}{2}\left\langle\Phi y,y\right\rangle_\mathcal{Y}},\qquad y\in\mathcal{Y},\]
so a Gaussian distribution is uniquely determined by its mean and covariance operator. 

The next result gives a concentration result for Gaussian random variables in separable Hilbert spaces. 
\begin{proposition}\label{Pconcentrationgaussian}
	Suppose that \(Y\) is a Gaussian random variable in \(\mathcal{Y}\) with mean 0 and covariance operator \(\Phi\). Then for any \(0<\lambda<\frac{1}{2\textnormal{Tr}\Phi}\),
	\[\mathbb{E}\left[e^{\lambda\left\lVert Y\right\rVert^2_\mathcal{Y}}\right]\leq\frac{1}{\sqrt{1-2\lambda\textnormal{Tr}\Phi}}\]
	and consequently, 
	\[\mathbb{P}\left(\left\lVert Y\right\rVert_\mathcal{Y}\geq a\right)\leq2e^{-\frac{3a^2}{8\textnormal{Tr}\Phi}}.\]
\end{proposition}
\begin{proof}
	By Lemma \ref{Lgaussiancovariance}, the covariance operator \(\Phi\) is self-adjoint and trace-class, so it is compact \citep[p.89, Theorem 18.11(b)]{conway2000course}. Then the spectral theorem \citep[p.46, Theorem 5.1]{conway1990course} tells us that there exists an orthonormal basis \(\{e_j\}_{j=1}^\infty\) of \(\mathcal{Y}\) of eigenvectors of \(\Phi\), with corresponding eigenvalues \(\lambda_j\). For each \(j\in\mathbb{N}\), \(Z_j=\left\langle Y,e_j\right\rangle_\mathcal{Y}\) is a zero-mean real Gaussian variable with variance \(\mathbb{E}\left[Z_j^2\right]=\left\langle\Phi e_j,e_j\right\rangle_\mathcal{Y}=\lambda_j\), and moreover, for \(j\neq k\), \(\mathbb{E}\left[Z_jZ_k\right]=\left\langle\Phi e_j,e_k\right\rangle_\mathcal{Y}=\left\langle\lambda_je_j,e_k\right\rangle_\mathcal{Y}=0\). This means that, for any \(N\in\mathbb{N}\), \((Z_1,...,Z_N)\) are mutually independent. Then see that, for any \(0<\lambda<\frac{1}{2\text{Tr}\Phi}\),
	\begin{alignat*}{2}
		\mathbb{E}\left[e^{\lambda\sum^N_{j=1}Z_j^2}\right]&=\prod^N_{j=1}\mathbb{E}\left[e^{\lambda Z_j^2}\right]\\
		&=\prod^N_{j=1}\frac{1}{\sqrt{2\pi\lambda_j}}\int^\infty_{-\infty}e^{\lambda z^2-\frac{z^2}{2\lambda_j}}dz\\
		&=\prod^N_{j=1}\frac{1}{\sqrt{1-2\lambda_j\lambda}}\\
		&\leq\frac{1}{\sqrt{1-2\lambda\sum^N_{j=1}\lambda_j}},
	\end{alignat*}
	and letting \(N\rightarrow\infty\), we obtain
	\[\mathbb{E}\left[e^{\lambda\left\lVert Y\right\rVert^2_\mathcal{Y}}\right]\leq\frac{1}{\sqrt{1-2\lambda\text{Tr}\Phi}}.\]
	Now, for any \(a>0\) and \(0<\lambda<\frac{1}{2\text{Tr}\Phi}\), Markov's inequality gives
	\[\mathbb{P}\left(\left\lVert Y\right\rVert_\mathcal{Y}\geq a\right)=\mathbb{P}\left(e^{\lambda\left\lVert Y\right\rVert_\mathcal{Y}^2}\geq e^{\lambda a^2}\right)\leq e^{-\lambda a^2}\mathbb{E}\left[e^{\lambda\left\lVert Y\right\rVert^2_\mathcal{Y}}\right]\leq e^{-\lambda a^2}\frac{1}{\sqrt{1-2\lambda\text{Tr}\Phi}}.\]
	Let \(\lambda=\frac{3}{8\text{Tr}\Phi}\) to finish the proof. 
\end{proof}

\section{Differential Calculus}\label{Sdiffcalcfull}
Recall that \(\mathcal{Y}\) is a Hilbert space. Suppose that \(U\) is an open subset of \(\mathbb{R}^d\), and denote the Euclidean norm in \(\mathbb{R}^d\) by \(\lVert\cdot\rVert\). We say that \(f_1,f_2:U\rightarrow\mathcal{Y}\) are \textit{tangent} at a point \(a\in U\) \citep[p.28]{cartan1967calcul} if the quantity
\[m(r)=\sup_{\left\lVert x-a\right\rVert\leq r}\left\lVert f_1(x)-f_2(x)\right\rVert_\mathcal{Y},\]
which is defined for \(r>0\) small enough (since \(U\) is open), satisfies the condition
\[\lim_{r\rightarrow0}\frac{m(r)}{r}=0,\qquad\text{which we also write as}\qquad m(r)=o(r).\]
We say that the map \(g:U\rightarrow\mathcal{Y}\) is \textit{differentiable} at \(a\in U\) if \(g\) is continuous at \(a\) and there exists a linear map \(g'(a):\mathbb{R}^d\rightarrow\mathcal{Y}\) such that the maps \(x\mapsto g(x)-g(a)\) and \(x\mapsto g'(a)(x-a)\) are tangent at \(a\) \citep[p.29]{cartan1967calcul}. This condition is also written as
\[\left\lVert g(x)-g(a)-g'(a)(x-a)\right\rVert_\mathcal{Y}=o(\left\lVert x-a\right\rVert).\] 
This immediately implies that \(g'(a)\) is continuous, so \(g'(a)\) belongs to \(\mathscr{L}(\mathbb{R}^d,\mathcal{Y})\), the space of continuous linear operators from \(\mathbb{R}^d\) into \(\mathcal{Y}\). We call \(g'(a)\in\mathcal{Y}\) the \textit{derivative} of \(g\) at \(a\). We say that \(g\) is differentiable on \(U\) if \(g\) is differentiable at every point in \(U\), and the map \(g':U\rightarrow\mathcal{Y}\) is called the \textit{derivative map} of \(g\). We say that \(g\) is \textit{continuously differentiable}, or \textit{of class \(C^1\)}, if \(g\) is differentiable at every point of \(U\) and the map \(g':U\rightarrow\mathcal{Y}\) is continuous \citep[p.30]{cartan1967calcul}. 

Let \(g:U\rightarrow\mathcal{Y}\) be a continuous map. For each \(a=(a_1,...,a_d)\in U\) and each \(l=1,...,d\), consider the inclusion \(\lambda_l:\mathbb{R}\rightarrow\mathbb{R}^d\) defined by
\[\lambda_l(x_l)=(a_1,...,a_{l-1},x_l,a_{l+1},...,a_d).\]
The composition \(g\circ\lambda_l\) is defined on an open subset \(\lambda_l^{-1}(\mathcal{X})\subset\mathbb{R}\), which contains \(a_l\). If \(g\) is differentiable at \(a\), then for each \(l=1,...,d\), the map \(g\circ\lambda_l\) is differentiable at \(a_l\) \citep[p.38, Proposition 2.6.1]{cartan1967calcul}. The derivative of \(g\circ\lambda_l\) at \(a\) is called the \textit{partial derivative} of \(g\), denoted by \(\partial_lg(a)\), and lives in \(\mathscr{L}(\mathbb{R},\mathcal{Y})\). But \(\mathscr{L}(\mathbb{R},\mathcal{Y})\) is isometrically isomorphic to \(\mathcal{Y}\) \citep[p.20, Exemple 1]{cartan1967calcul}, so we can view \(\partial_lg(a)\) as an element of \(\mathcal{Y}\). Moreover,
\[g'(a)(h)=g'(a)(h_1,...,h_d)=\sum^d_{l=1}h_l\partial_lg(a),\qquad\text{for }h=(h_1,...,h_d)\in\mathbb{R}^d.\]
\citet[p.40, Proposition 2.6.2]{cartan1967calcul} tells us that \(g\) is of class \(C^1\) if and only if \(\partial_lg:U\rightarrow\mathcal{Y}\) is continuous for each \(l=1,...,d\). 

Next, we consider higher-order derivatives. For an integer \(m\), a map \(F:(\mathbb{R}^d)^m\rightarrow\mathcal{Y}\) is \textit{\(m\)-linear} if, for each \(k=1,...,m\) and any \(a^{(1)},...,a^{(k-1)},a^{(k+1)},...,a^{(m)}\in\mathbb{R}^d\), the map \(x\mapsto F(a^{(1)},...,a^{(k-1)},x,a^{(k+1)},...,a^{(m)})\) is linear from \(\mathbb{R}^d\) into \(\mathcal{Y}\) \citep[p.24]{cartan1967calcul}. We say that \(F\) is an \(m\)-linear map from \(\mathbb{R}^d\) into \(\mathcal{Y}\), and denote by \(\mathscr{L}_m(\mathbb{R}^d;\mathcal{Y})\) the space of all continuous \(m\)-linear maps from \(\mathbb{R}^d\) into \(\mathcal{Y}\)\footnote{Beware that \(\mathscr{L}_m(\mathbb{R}^d;\mathcal{Y})\), the space of continuous \(m\)-linear maps from \(\mathbb{R}^d\) into \(\mathcal{Y}\), is different to \(\mathscr{L}((\mathbb{R}^d)^m,\mathcal{Y})\), the space of continuous linear maps from \((\mathbb{R}^d)^m\) into \(\mathcal{Y}\).}. The space \(\mathscr{L}_m(\mathbb{R}^d;\mathcal{Y})\) can then be equipped with a natural operator norm defined by
\[\left\lVert F\right\rVert_\text{op}=\sup_{\lVert x^{(1)}\rVert\leq1,...,\lVert x^{(m)}\rVert\leq1}\left\lVert F(x^{(1)},...,x^{(m)})\right\rVert_\mathcal{Y}.\]
For any integer \(m\), \citet[p.88, Theorem 4.4]{coleman2012calculus} tells us that \(\Psi_m:\mathscr{L}(\mathbb{R}^d,\mathscr{L}_{m-1}(\mathbb{R}^d;\mathcal{Y}))\rightarrow\mathscr{L}_m(\mathbb{R}^d;\mathcal{Y})\) defined by \(\Psi_m(F)(x^{(1)},x^{(2)},...,x^{(m)})=F(x^{(1)})(x^{(2)},...,x^{(m)})\) is an isometric isomorphism. 

We say that \(g:U\rightarrow\mathcal{Y}\) is \textit{twice differentiable at \(a\in U\)} if the derivative map \(g':U\rightarrow\mathscr{L}(\mathbb{R}^d,\mathcal{Y})\) is differentiable at \(a\). We denote by \(g''(a)=g^{(2)}(a)\in\mathscr{L}(\mathbb{R}^d,\mathscr{L}(\mathbb{R}^d,\mathcal{Y}))\simeq\mathscr{L}_2(\mathbb{R}^d;Y)\) the \textit{second derivative of \(g\) at \(a\)}. We say that \(g\) is \textit{twice differentiable on \(U\)} if it is twice differentiable at all points in \(U\). Then we have a map \(g^{(2)}:U\rightarrow\mathscr{L}_2(\mathbb{R}^d,\mathcal{Y})\). We say that \(g\) is \textit{twice continuously differentiable on \(U\)}, or \textit{of class \(C^2\) on \(U\)}, if \(g\) is twice differentiable and if the map \(g^{(2)}\) is continuous \citep[p.64]{cartan1967calcul}. By continuing in this way, we say that \(g\) is \textit{\(m\)-times differentiable at \(a\in U\)} if \(g^{(m-1)}:U\rightarrow\mathscr{L}_{m-1}(\mathbb{R}^d;\mathcal{Y})\) is differentiable at \(a\), define the \textit{\(m^\text{th}\) derivative} \(g^{(m)}(a)\in\mathscr{L}_m(\mathbb{R}^d;\mathcal{Y})\) of \(g\) at \(a\) as the derivative of \(g^{(m-1)}\) at \(a\), and say that \(g\) is \(m\)-times differentiable on \(U\) if it is \(m\)-times differentiable at all points in \(U\). We say that \(g\) is \textit{of class \(C^m\) on \(U\)} if \(g\) is \(m\)-times differentiable at all points in \(U\) and the map \(g^{(m)}:U\rightarrow\mathscr{L}_m(\mathbb{R}^d;\mathcal{Y})\) is continuous; we say that \(g\) is \textit{of class \(C^\infty\)} if it is of class \(C^m\) for all \(m\in\mathbb{N}\) \citep[pp.69--70]{cartan1967calcul}. 

Similarly, for \(l_1\in\{1,...,d\}\), if the partial derivative \(\partial_{l_1}g:U\rightarrow\mathcal{Y}\) is defined in some neighbourhood of \(x\in U\) and is differentiable, then for \(l_2\in\{1,...,d\}\) (which may or may not be distinct from \(l_1\)), we may define the second partial derivative \(\partial_{l_1}\partial_{l_2}g(a)\in\mathcal{Y}\). If \(l_1=l_2=l\), then we write \(\partial_l\partial_lg=\partial^2_lg\). Analogously to the first partial derivative, we have a formula that expresses the second derivative as a sum of second partial derivatives:
\[g''(a)((h^{(1)}_1,...,h^{(1)}_d),(h^{(2)}_1,...,h^{(2)}_d))=\sum_{l_1,l_2=1}^dh^{(1)}_{l_1}h^{(2)}_{l_2}\partial_{l_1}\partial_{l_2}g(a),\]
where \(h^{(1)}=(h^{(1)}_1,...,h^{(1)}_d),h^{(2)}=(h^{(2)}_1,...,h^{(2)}_d)\in\mathbb{R}^d\) \citep[p.68, (5.2.5)]{cartan1967calcul}. Continuing in the same way, we can define the \(m^\text{th}\) partial derivative \(\partial_{l_1}...\partial_{l_m}g(a)\in\mathcal{Y}\). Then writing \(\mathbf{h}=(h^{(1)},...,h^{(m)})\in(\mathbb{R}^d)^m\), we have
\[g^{(m)}(a)(\mathbf{h})=\sum_{l_1,...,l_m=1}^dh^{(1)}_{l_1}...h^{(m)}_{l_m}\partial_{l_1}...\partial_{l_m}g(a).\]
Finally, we state the extension of Taylor's theorem to functions with values in \(\mathcal{Y}\), with Lagrange's form of the remainder. To this end, for \(a,b\in\mathbb{R}^d\), define the \textit{segment} joining \(a\) and \(b\) as the set \citep[p.51]{coleman2012calculus}. 
\[[a,b]=\{x\in\mathbb{R}^d:x=va+(1-v)b,v\in[0,1]\}.\]
\begin{theorem}[{\citet[p.77, Th\'eor\`eme 5.6.2]{cartan1967calcul}}]
	Suppose that \(g:U\rightarrow\mathcal{Y}\) is \((m+1)\)-times differentiable, that the segment \([a,a+h]\) is contained in \(U\) and that, for some \(K>0\), we have
	\[\left\lVert g^{(m+1)}(x)\right\rVert_\textnormal{op}\leq K\qquad\text{for all }x\in U.\]
	Then
	\[\left\lVert g(a+h)-\sum^m_{k=0}\frac{1}{k!}g^{(k)}(a)((h)^k)\right\rVert_\mathcal{Y}\leq K\frac{\left\lVert h\right\rVert^{m+1}}{(m+1)!},\]
	where we wrote \((h)^k=(h,...,h)\in(\mathbb{R}^d)^k\) for \(k=1,...,m\).
\end{theorem}
Write \(\mathbb{N}_0=\{0,1,2,...\}\), and for \(p=(p_1,...,p_d)\in\mathbb{N}^d_0\), write \([p]\vcentcolon=p_1+...+p_d\). Then we denote the \(p^\text{th}\) partial derivative \(\partial^{p_1}_1...\partial^{p_d}_dg(a)\) of \(g\) at \(a\in U\) as \(D^pg(a)\in\mathcal{Y}\). This is possible since the order of partial differentiation is immaterial by repeated application of \citet[p.69, Proposition 5.2.2]{cartan1967calcul}. Hence, for each \(k=1,...,m+1\), we have
\[g^{(k)}(a)((h)^k)=\sum^d_{l_1,...,l_k=1}h_{l_1}...h_{l_k}\partial_{l_1}...\partial_{l_k}g(a)=\sum_{[p]=k}\frac{k!h^p}{p!}D^pg(a),\]
where we wrote \(h^p\) as a shorthand for \(h_1^{p_1}...h_d^{p_d}\) and \(p!\) for \(p_1!...p_d!\). Hence, using partial derivatives, we can express Taylor's theorem above as
\[\left\lVert g(a+h)-\sum_{[p]\leq m}\frac{h^p}{p!}D^pg(a)\right\rVert_\mathcal{Y}\leq K\frac{\left\lVert h\right\rVert^{m+1}}{(m+1)!}.\]
\section{Empirical Process Theory with Vector-Valued Functions}\label{Sempiricalprocessesfull}
In this Section, we state and prove some basic empirical process-theoretic results, adapted to our setting of vector-valued functions. Although technically new, the ideas and proofs carry over from the real case with ease, by applying vector-valued concentration inequalities from Section \ref{SconcentrationHilbfull}. 

\subsection{Symmetrisation}\label{SSsymmetrisation}
Symmetrisation is an indispensable technique in empirical process theory. Let \(X'_1,...,X'_n\) be another set of independent copies of \(X\), independent of \(X_1,...,X_n\). Denote by \(P'_n\) the empirical measure on \(X'_1,...,X'_n\), i.e. \(P'_n=\frac{1}{n}\sum^n_{i=1}\delta_{X'_i}\). 
\begin{lemma}\label{Lsymmetrisationcopies}
	We have
	\[\mathbb{E}\left[\left\lVert P_n-P\right\rVert_\mathcal{G}\right]\leq\mathbb{E}\left[\left\lVert P_n-P'_n\right\rVert_\mathcal{G}\right].\]
\end{lemma}
\begin{proof}
	Denote by \(\mathcal{F}_n\) the \(\sigma\)-algebra generated by \(X_1,...,X_n\). Then for each \(g\in\mathcal{G}\), we have
	\[\mathbb{E}\left[P_ng\mid\mathcal{F}_n\right]=P_ng\qquad\text{and}\qquad\mathbb{E}\left[P'_ng\mid\mathcal{F}_n\right]=Pg,\]
	and so
	\[(P_n-P)g=\mathbb{E}\left[\left(P_n-P'_n\right)g\mid\mathcal{F}_n\right].\]
	Now see that
	\begin{alignat*}{3}
		\left\lVert P_n-P\right\rVert_\mathcal{G}&=\sup_{g\in\mathcal{G}}\left\lVert\mathbb{E}\left[\left(P_n-P'_n\right)g\mid\mathcal{F}_n\right]\right\rVert_\mathcal{Y}\\
		&\leq\sup_{g\in\mathcal{G}}\mathbb{E}\left[\left\lVert\left(P_n-P'_n\right)g\right\rVert_\mathcal{Y}\mid\mathcal{F}_n\right]&&\text{by Jensen's inequality}\\
		&\leq\mathbb{E}\left[\sup_{g\in\mathcal{G}}\left\lVert\left(P_n-P'_n\right)g\right\rVert_\mathcal{Y}\mid\mathcal{F}_n\right].
	\end{alignat*}
	Now take expectations on both sides and apply the law of iterated expectations arrive at the result. 
\end{proof}
We let \(\{\sigma_i\}_{i=1}^n\) be a \textit{Rademacher sequence}, i.e. a sequence of independent random variables \(\sigma_i\) with
\[\mathbb{P}\left(\sigma_i=1\right)=\mathbb{P}\left(\sigma_i=-1\right)=\frac{1}{2},\qquad\text{for all }i=1,...,n.\]
We define the symmetrised empirical measures \(P^\sigma_n=\frac{1}{n}\sum^n_{i=1}\sigma_i\delta_{X_i}\) and \(P'^\sigma_n=\frac{1}{n}\sum^n_{i=1}\sigma_i\delta_{X'_i}\), and denote
\[P^\sigma_ng=\frac{1}{n}\sum^n_{i=1}\sigma_ig(X_i)\qquad\text{and}\qquad P'^\sigma_ng=\frac{1}{n}\sum^n_{i=1}\sigma_ig(X'_i).\]
\begin{lemma}[Symmetrisation with means]\label{Lsymmetrisationmean}
	We have
	\[\mathbb{E}\left[\left\lVert P_n-P\right\rVert_\mathcal{G}\right]\leq2\mathbb{E}\left[\left\lVert P^\sigma_n\right\rVert_\mathcal{G}\right]\]
\end{lemma}
\begin{proof}
	Note that \(\left\lVert P_n-P'_n\right\rVert_\mathcal{G}\) has the same distribution as \(\left\lVert P^\sigma_n-P'^\sigma_n\right\rVert_\mathcal{G}\), since, for each \(i=1,...,n\) and \(g\in\mathcal{G}\). \(g(X_i)-g(X'_i)\) and \(\sigma_i\left(g(X_i)-g(X'_i)\right)\) have the same distribution. Hence, the triangle inequality gives us
	\begin{alignat*}{2}
		\mathbb{E}\left[\left\lVert P_n-P'_n\right\rVert_\mathcal{G}\right]=\mathbb{E}\left[\left\lVert P^\sigma_n-P'^\sigma_n\right\rVert_\mathcal{G}\right]\leq\mathbb{E}\left[\left\lVert P^\sigma_n\right\rVert_\mathcal{G}+\left\lVert P'^\sigma_n\right\rVert_\mathcal{G}\right]=2\mathbb{E}\left[\left\lVert P^\sigma_n\right\rVert_\mathcal{G}\right].
	\end{alignat*}
	Now apply Lemma \ref{Lsymmetrisationcopies}. 
\end{proof}
\begin{lemma}[Symmetrisation with probabilities]\label{Lsymmetrisationprobability}
	Let \(a>0\). Suppose that for all \(g\in\mathcal{G}\), 
	\[\mathbb{P}\left(\left\lVert\left(P_n-P\right)g\right\rVert_\mathcal{Y}>\frac{a}{2}\right)\leq\frac{1}{2}.\]
	Then
	\[\mathbb{P}\left(\left\lVert P_n-P\right\rVert_\mathcal{G}>a\right)\leq4\mathbb{P}\left(\left\lVert P^\sigma_n\right\rVert_\mathcal{G}>\frac{a}{4}\right).\]
\end{lemma}
\begin{proof}
	Denote again by \(\mathcal{F}_n\) the \(\sigma\)-algebra generated by \(X_1,...,X_n\). If \(\left\lVert P_n-P\right\rVert_\mathcal{G}>a\), then we know that for some random function \(g_*\) depending on \(X_1,...,X_n\), \(\left\lVert\left(P_n-P\right)g_*\right\rVert_\mathcal{Y}>a\). Because \(X'_1,...,X'_n\) are independent of \(\mathcal{F}_n\), 
	\[\mathbb{P}\left(\left\lVert\left(P'_n-P\right)g_*\right\rVert_\mathcal{Y}>\frac{a}{2}\mid\mathcal{F}_n\right)=\mathbb{P}\left(\left\lVert\left(P_n-P\right)g_*\right\rVert_\mathcal{Y}>\frac{a}{2}\right)\leq\frac{1}{2}.\tag{*}\]
	Then see that, 
	\begin{alignat*}{3}
		\mathbb{P}\left(\left\lVert P_n-P\right\rVert_\mathcal{G}>a\right)&\leq\mathbb{P}\left(\left\lVert\left(P_n-P\right)g_*\right\rVert_\mathcal{Y}>a\right)\\
		&=\mathbb{E}\left[\mathbf{1}\left\{\left\lVert\left(P_n-P\right)g_*\right\rVert_\mathcal{Y}>a\right\}\right]\\
		&\leq2\mathbb{E}\left[\mathbb{P}\left(\left\lVert\left(P'_n-P\right)g_*\right\rVert_\mathcal{Y}\leq\frac{a}{2}\mid\mathcal{F}_n\right)\mathbf{1}\left\{\left\lVert\left(P_n-P\right)g_*\right\rVert_\mathcal{Y}>a\right\}\right]&&\text{by (*)}\\
		&=2\mathbb{E}\left[\mathbb{P}\left(\left\lVert\left(P'_n-P\right)g_*\right\rVert_\mathcal{Y}\leq\frac{a}{2}\text{ and }\left\lVert\left(P_n-P\right)g_*\right\rVert_\mathcal{Y}>a\mid\mathcal{F}_n\right)\right]\\
		&=2\mathbb{P}\left(\left\lVert\left(P'_n-P\right)g_*\right\rVert_\mathcal{Y}\leq\frac{a}{2}\text{ and }\left\lVert\left(P_n-P\right)g_*\right\rVert_\mathcal{Y}>a\right).
	\end{alignat*}
	But if the two inequalities in the probability on the last line hold, then the reverse triangle inequality gives us
	\[\frac{a}{2}<\left\lVert\left(P_n-P\right)g_*\right\rVert_\mathcal{Y}-\left\lVert\left(P'_n-P\right)g_*\right\rVert_\mathcal{Y}\leq\left\lVert\left(P_n-P'_n\right)g_*\right\rVert_\mathcal{Y},\]
	so
	\begin{alignat*}{2}
		\mathbb{P}\left(\left\lVert P_n-P\right\rVert_\mathcal{G}>a\right)&\leq2\mathbb{P}\left(\left\lVert\left(P_n-P'_n\right)g_*\right\rVert_\mathcal{Y}>\frac{a}{2}\right)\\
		&\leq2\mathbb{P}\left(\left\lVert P_n-P'_n\right\rVert_\mathcal{G}>\frac{a}{2}\right)\\
		&=2\mathbb{P}\left(\left\lVert P^\sigma_n-P'^\sigma_n\right\rVert_\mathcal{G}>\frac{a}{2}\right)\\
		&\leq2\mathbb{P}\left(\left\lVert P^\sigma_n\right\rVert_\mathcal{G}>\frac{a}{4}\text{ or }\left\lVert P'^\sigma_n\right\rVert_\mathcal{G}>\frac{a}{4}\right)\\
		&\leq4\mathbb{P}\left(\left\lVert P^\sigma_n\right\rVert_\mathcal{G}>\frac{a}{4}\right).
	\end{alignat*}
\end{proof}
A simple application of the above symmetrisation argument and Hoeffding's inequality in Hilbert spaces (Proposition \ref{Phoeffdinghilbert}) shows that finite function classes are Glivenko-Cantelli.
\begin{lemma}\label{Lfiniteclass}
	Let \(\mathcal{G}=\left\{g_1,...,g_N\right\}\in L^1(\mathcal{X},P;\mathcal{Y})\) be a finite class of functions with cardinality \(N>1\). Then we have
	\[\left\lVert P_n-P\right\rVert_\mathcal{G}\rightarrow0.\]
\end{lemma}
\begin{proof}
	Take any \(K>0\). Define the function \(G:\mathcal{X}\rightarrow\mathbb{R}\) by \(G(x)=\max_{1\leq j\leq N}\left\lVert g_j(x)\right\rVert_\mathcal{Y}\). Since each \(\left\lVert g_j\right\rVert_\mathcal{Y}\) is integrable, and we have a finite collection, \(G\) is also integrable. Then, for each \(j=1,...,N\), define the function \(\tilde{g}_j:\mathcal{X}\rightarrow\mathcal{Y}\) by \(\tilde{g}_j=g_j\mathbf{1}\left\{G\leq K\right\}\). Then for all \(i=1,...,n\), letting \(\sigma_i\) be independent Rademacher variables again, we have
	\[\mathbb{E}\left[\sigma_i\tilde{g}_j(X_i)\right]=0\qquad\text{and}\qquad\left\lVert\sigma_i\tilde{g}_j(X_i)\right\rVert_\mathcal{Y}\leq K\text{ almost surely.}\]
	Hence, for each \(j=1,...,N\), by Hoeffding's inequality (Proposition \ref{Phoeffdinghilbert}), for any \(t>0\), we have
	\[\mathbb{P}\left(\left\lVert P^\sigma_n\tilde{g}_j\right\rVert_\mathcal{Y}\geq2K\sqrt{\frac{t}{n}}\right)=\mathbb{P}\left(\left\lVert\sum^n_{i=1}\sigma_i\tilde{g}_j(X_i)\right\rVert_\mathcal{Y}\geq2K\sqrt{nt}\right)\leq2e^{-t}.\]
	By the union bound, for any \(t>0\), we have
	\begin{alignat*}{2}
		\mathbb{P}\left(\max_{1\leq j\leq N}\left\lVert P^\sigma_n\tilde{g}_j\right\rVert_\mathcal{Y}\geq2K\sqrt{\frac{t+\log N}{n}}\right)&\leq N\max_{1\leq j\leq N}\mathbb{P}\left(\left\lVert P^\sigma_n\tilde{g}_j\right\rVert_\mathcal{Y}\geq2K\sqrt{\frac{t+\log N}{n}}\right)\\
		&\leq2e^{-t}.
	\end{alignat*}
	Now see that, for each \(j=1,...,N\), Chebyshev's inequality gives
	\[\mathbb{P}\left(\left\lVert\left(P_n-P\right)\tilde{g}_j\right\rVert_\mathcal{Y}>4K\sqrt{\frac{t+\log N}{n}}\right)\leq\frac{n\mathbb{E}\left[\left\lVert\left(P_n-P\right)\tilde{g}_j\right\rVert_\mathcal{Y}^2\right]}{16K^2\left(t+\log N\right)}\leq\frac{1}{16\left(t+\log N\right)}\leq\frac{1}{2},\]
	where the last inequality follows since \(8t+8\log N\geq8\log2\geq1\). Now apply Lemma \ref{Lsymmetrisationprobability} to see that
	\begin{alignat*}{2}
		\mathbb{P}\left(\max_{1\leq j\leq N}\left\lVert\left(P_n-P\right)\tilde{g}_j\right\rVert_\mathcal{Y}>8K\sqrt{\frac{t+\log N}{n}}\right)&\leq4\mathbb{P}\left(\max_{1\leq j\leq N}\left\lVert P^\sigma_n\tilde{g}_j\right\rVert_\mathcal{Y}>2K\sqrt{\frac{t+\log N}{n}}\right)\\
		&\leq8e^{-t}.
	\end{alignat*}
	This tells us that
	\[\max_{1\leq j\leq N}\left\lVert\left(P_n-P\right)\tilde{g}_j\right\rVert_\mathcal{Y}\stackrel{P}{\rightarrow}0.\]
	Finally, see that
	\[\left\lVert P_n-P\right\rVert_\mathcal{G}\leq\max_{1\leq j\leq N}\left\lVert\left(P_n-P\right)\tilde{g}_j\right\rVert_\mathcal{Y}+\max_{1\leq j\leq N}\left\lVert\left(P_n-P\right)g_j\mathbf{1}\left\{G>K\right\}\right\rVert_\mathcal{Y}.\]
	Here, the first term converges to 0 in probability for any \(K>0\), as shown above, and the second term decomposes as 
	\begin{alignat*}{2}
		\max_{1\leq j\leq N}\left\lVert\left(P_n-P\right)g_j\mathbf{1}\left\{G>K\right\}\right\rVert_\mathcal{Y}&\leq\left(P_n+P\right)G\mathbf{1}\left\{G>K\right\}\\
		&=\left(P_n-P\right)G\mathbf{1}\left\{G>K\right\}+2PG\mathbf{1}\left\{G>K\right\}\\
		&\leq\left(P_n-P\right)G+2PG\mathbf{1}\left\{G>K\right\}.
	\end{alignat*}
	Here, the first term converges to 0 in probability by the weak law of large numbers, and the second term converges to 0 as \(K\rightarrow\infty\), by \citet[p.71, Lemma 3.10]{cinlar2011probability}. 
\end{proof}

\subsection{Uniform law of large numbers}\label{SSulln}
We start with a definition. 
\begin{definition}[Adapted from {\citet[p.26, Definition 3.1]{vandegeer2000empirical}}]\label{Denvelope}
	The function \(G:\mathcal{X}\rightarrow\mathbb{R}\) defined by \(G(\cdot)=\sup_{g\in\mathcal{G}}\left\lVert g(\cdot)\right\rVert_\mathcal{Y}\) is called the \textit{envelope} of \(\mathcal{G}\). 
\end{definition}
The following is a uniform law of large numbers based on conditions on the entropy \(H(\delta,\mathcal{G},\lVert\cdot\rVert_{1,P_n})\) and the envelope \(G\). 
\begin{theorem}\label{Tglivenkocantelli}
	Suppose that
	\[G\in L^1(\mathcal{X},P;\mathbb{R})\qquad\text{and}\qquad\frac{1}{n}H(\delta,\mathcal{G},\lVert\cdot\rVert_{1,P_n})\stackrel{P}{\rightarrow}0\enspace\text{for each }\delta>0.\]
	Then \(\mathcal{G}\) is a Glivenko Cantelli class, i.e. \(\left\lVert P_n-P\right\rVert_\mathcal{G}\stackrel{P}{\rightarrow}0\). 
\end{theorem}
\begin{proof}
	Take any \(K>0\) and \(\delta>0\). Denote again by \(\mathcal{F}_n\) the \(\sigma\)-algebra generated by \(X_1,...,X_n\), and define \(\mathcal{G}_K=\{g\mathbf{1}\{G\leq K\}:g\in\mathcal{G}\}\). Let \(g_1,...,g_N\), with \(N=N(\delta,\mathcal{G},\lVert\cdot\rVert_{1,P_n})\), be a minimal \(\delta\)-covering of \(\mathcal{G}\). Then \(N\) is a random variable, that is measurable with respect to \(\mathcal{F}_n\). Moreover, writing \(\tilde{g}_j=g_j\mathbf{1}\{G\leq K\}\) for each \(j=1,...,N\), \(\tilde{g}_1,...,\tilde{g}_N\) form a \(\delta\)-covering of \(\mathcal{G}_K\), since, for any \(\tilde{g}=g\mathbf{1}\{G\leq K\}\in\mathcal{G}_K\) for \(g\in\mathcal{G}\), there exists \(j\in\{1,...,N\}\) with \(\lVert g-g_j\rVert_{1,P_n}\leq\delta\), so \(\lVert\tilde{g}-\tilde{g}_j\rVert_{1,P_n}\leq\lVert g-g_j\rVert_{1,P_n}\leq\delta\). 
	
	Note that, when \(\left\lVert\tilde{g}-\tilde{g}_j\right\rVert_{1,P_n}=P_n\left\lVert\tilde{g}-\tilde{g}_j\right\rVert_\mathcal{Y}\leq\delta\), we have
	\[\left\lVert P^\sigma_n\tilde{g}\right\rVert_\mathcal{Y}\leq\left\lVert P^\sigma_n\tilde{g}_j\right\rVert_\mathcal{Y}+\left\lVert P^\sigma_n\tilde{g}-P^\sigma_n\tilde{g}_j\right\rVert_\mathcal{Y}\leq\left\lVert P^\sigma_n\tilde{g}_j\right\rVert_\mathcal{Y}+P_n\left\lVert\tilde{g}-\tilde{g}_j\right\rVert_\mathcal{Y}\leq\left\lVert P^\sigma_n\tilde{g}_j\right\rVert_\mathcal{Y}+\delta.\]
	So for any \(\tilde{g}\in\mathcal{G}_K\),
	\[\left\lVert P^\sigma_n\tilde{g}\right\rVert_\mathcal{Y}\leq\max_{1\leq j\leq N}\left\lVert P^\sigma_n\tilde{g}_j\right\rVert_\mathcal{Y}+\delta.\tag{*}\]
	By Hoeffding's inequality and union bound (as in the proof of Lemma \ref{Lfiniteclass}, since \(N\) is measurable with respect to \(\mathcal{F}_n\)), for any \(t>0\), we have
	\[\mathbb{P}\left(\max_{1\leq j\leq N}\left\lVert P^\sigma_n\tilde{g}_j\right\rVert_\mathcal{Y}\geq2K\sqrt{\frac{t+\log N}{n}}\mid\mathcal{F}_n\right)\leq2e^{-t}.\]
	We then apply (*) and integrate both sides (to remove the conditioning on \(\mathcal{F}_n\)) to see that, for any \(t>0\),
	\[\mathbb{P}\left(\left\lVert P^\sigma_n\right\rVert_{\mathcal{G}_K}\geq\delta+2K\sqrt{\frac{t+\log N}{n}}\right)\leq2e^{-t}.\]
	Then see that, using the elementary inequality \(\sqrt{a}+\sqrt{b}\geq\sqrt{a+b}\), 
	\begin{alignat*}{2}
		&\mathbb{P}\left(\left\lVert P^\sigma_n\right\rVert_{\mathcal{G}_K}\geq2\delta+2K\sqrt{\frac{t}{n}}\right)\\
		&\leq\mathbb{P}\left(\left\lVert P^\sigma_n\right\rVert_{\mathcal{G}_K}\geq\delta+2K\sqrt{\frac{t}{n}}+2K\sqrt{\frac{\log N}{n}}\right)+\mathbb{P}\left(2K\sqrt{\frac{\log N}{n}}\geq\delta\right)\\
		&\leq\mathbb{P}\left(\left\lVert P^\sigma_n\right\rVert_{\mathcal{G}_K}\geq\delta+2K\sqrt{\frac{t+\log N}{n}}\right)+\mathbb{P}\left(2K\sqrt{\frac{1}{n}H(\delta,\mathcal{G},\lVert\cdot\rVert_{1,P_n})}\geq\delta\right)\\
		&\leq2e^{-t}+\mathbb{P}\left(2K\sqrt{\frac{1}{n}H(\delta,\mathcal{G},\lVert\cdot\rVert_{1,P_n})}\geq\delta\right). 
	\end{alignat*}
	Also, by Chebyshev's inequality, for each \(\tilde{g}\in\mathcal{G}_K\), we have, for any \(t\geq\frac{1}{8}\)
	\begin{alignat*}{2}
		\mathbb{P}\left(\left\lVert\left(P_n-P\right)\tilde{g}\right\rVert_\mathcal{Y}>4\delta+4K\sqrt{\frac{t}{n}}\right)&\leq\mathbb{P}\left(\left\lVert\left(P_n-P\right)\tilde{g}\right\rVert_\mathcal{Y}>4K\sqrt{\frac{t}{n}}\right)\\
		&\leq\frac{n\mathbb{E}\left[\left\lVert\left(P_n-P\right)\tilde{g}\right\rVert_\mathcal{Y}^2\right]}{16K^2t}\\
		&\leq\frac{1}{16t}\\
		&\leq\frac{1}{2}.
	\end{alignat*}
	Hence, we can apply symmetrisation with probabilities again (Lemma \ref{Lsymmetrisationprobability}) to see that, for any \(t\geq\frac{1}{8}\),
	\begin{alignat*}{2}
		\mathbb{P}\left(\left\lVert P_n-P\right\rVert_{\mathcal{G}_K}\geq8\delta+8K\sqrt{\frac{t}{n}}\right)&\leq4\mathbb{P}\left(\left\lVert P^\sigma_n\right\rVert_{\mathcal{G}_K}\geq2\delta+2K\sqrt{\frac{t}{n}}\right)\\
		&\leq2^{-t}+\mathbb{P}\left(2K\sqrt{\frac{1}{n}H(\delta,\mathcal{G},\lVert\cdot\rVert_{1,P_n})}\geq\delta\right).
	\end{alignat*}
	Here, since \(\delta>0\) was arbitrary and \(\frac{1}{n}H(\delta,\mathcal{G},\lVert\cdot\rVert_{1,P_n})\stackrel{P}{\rightarrow}0\) by hypothesis, we have that \(\mathcal{G}_K\) is Glivenko Cantelli. 
	
	Finally, see that
	\[\left\lVert P_n-P\right\rVert_\mathcal{G}\leq\sup_{\tilde{g}\in\mathcal{G}_K}\left\lVert\left(P_n-P\right)\tilde{g}\right\rVert_\mathcal{Y}+\sup_{g\in\mathcal{G}}\left\lVert\left(P_n-P\right)g\mathbf{1}\left\{G>K\right\}\right\rVert_\mathcal{Y}.\]
	Here, the first term converges to 0 in probability for any \(K>0\), as shown above, and the second term decomposes as 
	\begin{alignat*}{2}
		\sup_{g\in\mathcal{G}}\left\lVert\left(P_n-P\right)g\mathbf{1}\left\{G>K\right\}\right\rVert_\mathcal{Y}&\leq\left(P_n+P\right)G\mathbf{1}\left\{G>K\right\}\\
		&=\left(P_n-P\right)G\mathbf{1}\left\{G>K\right\}+2PG\mathbf{1}\left\{G>K\right\}\\
		&\leq\left(P_n-P\right)G+2PG\mathbf{1}\left\{G>K\right\}.
	\end{alignat*}
	Here, the first term converges to 0 in probability by the weak law of large numbers, and the second term converges to 0 as \(K\rightarrow\infty\), by \citet[p.71, Lemma 3.10]{cinlar2011probability}, since \(G\) is integrable by hypothesis. 
\end{proof}
\subsection{Chaining and asymptotic equicontinuity with empirical entropy}\label{SSchaining}
In this subsection we show that, with additional conditions on the entropy of \(\mathcal{G}\) (which we assume to be totally bounded with respect to the appropriate metric) and a technique called \say{chaining}, we can derive explicit finite-sample bounds, and show the asymptotic continuity of the empirical process indexed by \(\mathcal{G}\) (see Definition \ref{Dempiricalprocess}). As before, we work conditionally on the samples, and denote the \(\sigma\)-algebra generated by \(X_1,...,X_n\) as \(\mathcal{F}_n\). 

Suppose that \(\mathcal{G}\) has an envelope \(G\in L^2(\mathcal{X},P;\mathbb{R})\) (see Definition \ref{Denvelope}). Then the quantity \(R=\sup_{g\in\mathcal{G}}\left\lVert g\right\rVert_{2,P}\) is finite, since
\[R^2=\sup_{g\in\mathcal{G}}\mathbb{E}\left[\left\lVert g(X)\right\rVert_\mathcal{Y}^2\right]\leq\mathbb{E}\left[\sup_{g\in\mathcal{G}}\left\lVert g(X)\right\rVert_\mathcal{Y}^2\right]=\mathbb{E}\left[G^2\right]<\infty.\]
Similarly, the quantity \(R_n=\sup_{g\in\mathcal{G}}\left\lVert g\right\rVert_{2,P_n}\) is almost surely finite. We call \(R\) and \(R_n\) the \textit{theoretical radius} and \textit{empirical radius} of \(\mathcal{G}\), respectively. Note that \(R_n\) is a random quantity, measurable with respect to \(\mathcal{F}_n\). 

Let us fix \(S\in\mathbb{N}\). To ease the notation, for \(s=0,1,...,S\), write \(N_s=N(2^{-s}R_n,\mathcal{G},\lVert\cdot\rVert_{2,P_n})\) for the \(2^{-s}R_n\)-covering number of \(\mathcal{G}\) with respect to the \(\lVert\cdot\rVert_{2,P_n}\)-metric, which we assume to be finite. Let \(\{g^s_j\}^{N_s}_{j=1}\subset\mathcal{G}\) be a \(2^{-s}R_n\)-covering set of \(\mathcal{G}\) with respect to the \(\lVert\cdot\rVert_{2,P_n}\)-metric. Note that \(\{g^0\}=\{0\}\) is an \(R_n\)-covering set of \(\mathcal{G}\), since, for any \(g\in\mathcal{G}\), \(\left\lVert g\right\rVert_{2,P_n}\leq R_n\). Similarly, write \(H_s=\log N_s\) for each \(s=0,1,...,S\), for the corresponding entropy. Note that the quantities \(N_s\) and \(H_s\), as well as the covering set \(\{g^s_j\}^{N_s}_{j=1}\), are random quantities that are measurable with respect to \(\mathcal{F}_n\). 

Now fix \(g\in\mathcal{G}\). Then define
\begin{alignat*}{2}
	g^{S+1}&\vcentcolon=\argmin_{\{g^{S+1}_j\}_{j=1}^{N_{S+1}}}\left\{\left\lVert g-g_j^{S+1}\right\rVert_{2,P_n}\right\}\\
	g^S&\vcentcolon=\argmin_{\{g^S_j\}_{j=1}^{N_S}}\left\{\left\lVert g^{S+1}-g_j^S\right\rVert_{2,P_n}\right\}\\
	\vdots&\qquad\vdots\\
	g^s&\vcentcolon=\argmin_{\{g^s_j\}_{j=1}^{N_s}}\left\{\left\lVert g^{s+1}-g_j^s\right\rVert_{2,P_n}\right\}\\
	\vdots&\qquad\vdots\\
	g^0&\vcentcolon=0.
\end{alignat*}
\begin{proposition}[Chaining]\label{Pchaining}
	We fix \(S\in\mathbb{N}\). Define
	\[J_n\vcentcolon=\sum^S_{s=0}2^{-s}R_n\sqrt{2H_{s+1}}.\]
	\begin{enumerate}[(i)]
		\item For all \(t>0\), 
		\[\mathbb{P}\left(\sup_{g\in\mathcal{G}}\left\lVert\sum^S_{s=0}P^\sigma_n\left(g^{s+1}-g^s\right)\right\rVert_\mathcal{Y}\geq\frac{\sqrt{2}J_n}{\sqrt{n}}+6R_n\sqrt{\frac{1+t}{n}}\mid\mathcal{F}_n\right)\leq2e^{-t}.\]
		\item Suppose that \(\varepsilon_1,...,\varepsilon_n\) are i.i.d. Gaussian random variables in \(\mathcal{Y}\) with mean 0 and covariance operator \(Q\). Without loss of generality (by rescaling if necessary), assume \(\textnormal{Tr}Q=1\). For each \(g\in\mathcal{G}\), we can consider the following inner product:
		\[\left\langle\varepsilon,g\right\rangle_{2,P_n}=\frac{1}{n}\sum^n_{i=1}\left\langle\varepsilon_i,g(X_i)\right\rangle_\mathcal{Y}.\]
		Then for all \(t>0\),
		\[\mathbb{P}\left(\sup_{g\in\mathcal{G}}\sum^S_{s=0}\left\langle\varepsilon,g^{s+1}-g^s\right\rangle_{2,P_n}\geq\frac{J_n}{\sqrt{n}}+4R_n\sqrt{\frac{1+t}{n}}\mid\mathcal{F}_n\right)\leq e^{-t}.\]
	\end{enumerate}
\end{proposition}
\begin{proof}
	\begin{enumerate}[(i)]
		\item Fix \(s\in\{0,1,...,S\}\) and \(k\in\{1,...,N_{s+1}\}\). Denote
		\[g^{s+1,s}_k=\argmin_{\{g^s_j\}_{j=1}^{N_s}}\left\{\left\lVert g^{s+1}_k-g^s_j\right\rVert_{2,P_n}\right\}.\]
		Then
		\[\left\lVert P_n^\sigma\left(g^{s+1}_k-g^{s+1,s}_k\right)\right\rVert_\mathcal{Y}\leq\frac{1}{n}\sum^n_{i=1}\left\lVert g^{s+1}_k\left(X_i\right)-g^{s+1,s}_k\left(X_i\right)\right\rVert_\mathcal{Y},\]
		where
		\[\sqrt{\sum^n_{i=1}\left\lVert g^{s+1}_k\left(X_i\right)-g^{s+1,s}_k\left(X_i\right)\right\rVert_\mathcal{Y}^2}=\sqrt{n}\left\lVert g^{s+1}_k-g^{s+1,s}_k\right\rVert_{2,P_n}\leq\sqrt{n}2^{-s}R_n,\]
		since the \(\{g^s_j\}_{j=1}^{N_s}\) form a \(2^{-s}R_n\)-covering of \((\mathcal{G},\lVert\cdot\rVert_{2,P_n})\). Hence, noting that \(R_n\) is measurable with respect to \(\mathcal{F}_n\), Hoeffding's inequality (Proposition \ref{Phoeffdinghilbert}) gives, for any \(t>0\),
		\[\mathbb{P}\left(\left\lVert P^\sigma_n\left(g^{s+1}_k-g^{s+1,s}_k\right)\right\rVert_\mathcal{Y}\geq 2^{-(s-1)}R_n\sqrt{\frac{t}{n}}\mid\mathcal{F}_n\right)\leq2e^{-t}.\]
		Therefore (by the union bound), for each \(s=0,1,...,S\) and all \(t>0\),
		\[\mathbb{P}\left(\max_{k\in\{1,...,N_{s+1}\}}\left\lVert P^\sigma_n\left(g^{s+1}_k-g^{s+1,s}_k\right)\right\rVert_\mathcal{Y}\geq2^{-(s-1)}R_n\sqrt{\frac{H_{s+1}+t}{n}}\mid\mathcal{F}_n\right)\leq2e^{-t}.\]
		Fix \(t\) and for \(s=0,1,...,S\), let
		\begin{alignat*}{2}
			\alpha_s\vcentcolon&=2^{-(s-1)}R_n\left(\sqrt{H_{s+1}}+\sqrt{(1+s)(1+t)}\right)\\
			&\geq2^{-(s-1)}R_n\left(\sqrt{H_{s+1}+(1+s)(1+t)}\right),
		\end{alignat*}
		using \(\sqrt{a}+\sqrt{b}\geq\sqrt{a+b}\). Then using \(\sum^S_{s=0}2^{-(s-1)}\sqrt{1+s}\leq6\),
		\begin{alignat*}{2}
			\sum^S_{s=0}\alpha_s&=\sqrt{2}J_n+\sum^S_{s=0}2^{-(s-1)}R_n\sqrt{(1+s)(1+t)}\\
			&\leq\sqrt{2}J_n+6R_n\sqrt{1+t}.
		\end{alignat*}
		Therefore
		\begin{alignat*}{2}
			&\mathbb{P}\left(\sup_{g\in\mathcal{G}}\left\lVert\sum^S_{s=0}P^\sigma_n\left(g^{s+1}-g^s\right)\right\rVert_\mathcal{Y}\geq\frac{\sqrt{2}J_n}{\sqrt{n}}+6R_n\sqrt{\frac{1+t}{n}}\mid\mathcal{F}_n\right)\\
			&\leq\mathbb{P}\left(\sup_{g\in\mathcal{G}}\left\lVert\sum^S_{s=0}P^\sigma_n\left(g^{s+1}-g^s\right)\right\rVert_\mathcal{Y}\geq\frac{1}{\sqrt{n}}\sum^S_{s=0}\alpha_s\mid\mathcal{F}_n\right)\\
			&\leq\mathbb{P}\left(\sum^S_{s=0}\sup_{g\in\mathcal{G}}\left\lVert P^\sigma_n\left(g^{s+1}-g^s\right)\right\rVert_\mathcal{Y}\geq\frac{1}{\sqrt{n}}\sum^S_{s=0}\alpha_s\mid\mathcal{F}_n\right)\\
			&\leq\sum^S_{s=0}\mathbb{P}\left(\sup_{g\in\mathcal{G}}\left\lVert P^\sigma_n\left(g^{s+1}-g^s\right)\right\rVert_\mathcal{Y}\geq\frac{1}{\sqrt{n}}\alpha_s\mid\mathcal{F}_n\right)\\
			&=\sum^S_{s=0}\mathbb{P}\left(\max_{k=1,...,N_{s+1}}\left\lVert P^\sigma_n\left(g^{s+1}_k-g^{s+1,s}_k\right)\right\rVert_\mathcal{Y}\geq\frac{1}{\sqrt{n}}\alpha_s\mid\mathcal{F}_n\right)\\
			&\leq2\sum^S_{s=0}e^{-(1+s)(1+t)}\\
			&\leq2e^{-t}.
		\end{alignat*}
		\item Fix \(s\in\left\{0,1,...,S\right\}\) and \(k\in\left\{1,...,N_{s+1}\right\}\). Denote
		\[g^{s+1,s}_k=\argmin_{\{g^s_j\}_{j=1}^{N_s}}\left\{\left\lVert g^{s+1}_k-g^s_j\right\rVert_{2,P_n}\right\}.\]
		Let \(\lambda>0\) be arbitrary. Then Markov's inequality gives us, for any \(t>0\),
		\begin{alignat*}{2}
			&\mathbb{P}\left(\left\langle\varepsilon,g^{s+1}_k-g^{s+1,s}_k\right\rangle_{2,P_n}\geq2^{-s}R_n\sqrt{\frac{2t}{n}}\mid\mathcal{F}_n\right)\\
			&\leq e^{-\lambda2^{-s}R_n\sqrt{\frac{2t}{n}}}\mathbb{E}\left[e^{\frac{\lambda}{n}\sum^n_{i=1}\left\langle\varepsilon_i,g^{s+1}_k(X_i)-g^{s+1,s}_k(X_i)\right\rangle_\mathcal{Y}}\mid\mathcal{F}_n\right]\\
			&=e^{-\lambda2^{-s}R_n\sqrt{\frac{2t}{n}}}\prod^n_{i=1}\mathbb{E}\left[e^{\frac{\lambda}{n}\left\langle\varepsilon_i,g^{s+1}_k(X_i)-g^{s+1,s}_k(X_i)\right\rangle_\mathcal{Y}}\mid\mathcal{F}_n\right].
		\end{alignat*}
		Here, since \(\varepsilon_i\) is a \(\mathcal{Y}\)-valued Gaussian random variable with mean 0 and covariance operator \(Q\) for each \(i=1,...,n\), the distribution of the real variable \(\frac{\lambda}{n}\left\langle\varepsilon_i,g^{s+1}_k(X_i)-g^{s+1,s}_k(X_i)\right\rangle_\mathcal{Y}\) conditioned on \(\mathcal{F}_n\) is real Gaussian with mean \(0\) and variance
		\[\frac{\lambda^2}{n^2}\mathbb{E}\left[\left\langle g^{s+1}_k(X_i)-g^{s+1,s}_k(X_i),\varepsilon_i\right\rangle_\mathcal{Y}^2\mid\mathcal{F}_n\right]\leq\frac{\lambda^2}{n^2}\left\lVert g^{s+1}_k(X_i)-g^{s+1,s}_k(X_i)\right\rVert^2_\mathcal{Y},\]
		which follows from the Cauchy-Schwarz inequality and the fact that \(\mathbb{E}\left[\left\lVert\varepsilon_i\right\rVert^2_\mathcal{Y}\right]=\text{Tr}Q=1\). Hence, 
		\begin{alignat*}{2}
			&\mathbb{P}\left(\left\langle\varepsilon,g^{s+1}_k-g^{s+1,s}_k\right\rangle_{2,P_n}\geq2^{-s}R_n\sqrt{\frac{2t}{n}}\mid\mathcal{F}_n\right)\\
			&\leq e^{-\lambda2^{-s}R_n\sqrt{\frac{2t}{n}}}\prod^n_{i=1}e^{\frac{\lambda^2}{2n^2}\left\lVert g^{s+1}_k(X_i)-g^{s+1,s}_k(X_i)\right\rVert^2_\mathcal{Y}}\\
			&=e^{-\lambda2^{-s}R_n\sqrt{\frac{2t}{n}}}e^{\frac{\lambda^2}{2n^2}\sum^n_{i=1}\left\lVert g^{s+1}_k(X_i)-g^{s+1,s}_k(X_i)\right\rVert^2_\mathcal{Y}}\\
			&=e^{-\lambda2^{-s}R_n\sqrt{\frac{2t}{n}}}e^{\frac{\lambda^2}{2n}\left\lVert g^{s+1}_k-g^{s+1,s}_k\right\rVert^2_{2,P_n}}\\
			&\leq e^{-\lambda2^{-s}R_n\sqrt{\frac{2t}{n}}}e^{\frac{\lambda^2}{2n}\left(2^{-s}R_n\right)^2}.
		\end{alignat*}
		Now let \(\lambda=\frac{\sqrt{2nt}}{2^{-s}R_n}\) to see that
		\[\mathbb{P}\left(\left\langle\varepsilon,g^{s+1}_k-g^{s+1,s}_k\right\rangle_{2,P_n}\geq2^{-s}R_n\sqrt{\frac{2t}{n}}\mid\mathcal{F}_n\right)\leq e^{-t}.\]
		Therefore, by the union bound, for each \(s=0,1,...,S\) and all \(t>0\), 
		\[\mathbb{P}\left(\max_{k\in\{1,...,N_{s+1}\}}\left\langle\varepsilon,g^{s+1}_k-g^{s+1,s}_k\right\rangle_{2,P_n}\geq2^{-s}R_n\sqrt{\frac{2(t+H_{s+1})}{n}}\mid\mathcal{F}_n\right)\leq e^{-t}.\]
		Fix \(t\) and for \(s=0,1,...,S\), let
		\[\alpha_s\vcentcolon=2^{-s}R_n\left(\sqrt{2H_{s+1}}+\sqrt{2(1+s)(1+t)}\right)\geq2^{-s}R_n\sqrt{2\left(H_{s+1}+(1+s)(1+t)\right)}\]
		using \(\sqrt{a}+\sqrt{b}\geq\sqrt{a+b}\). Then using \(\sum^\infty_{s=0}2^{-s}\sqrt{2(1+s)}\leq4\),
		\[\sum^\infty_{s=0}\alpha_s=J_n+\sum^\infty_{s=0}2^{-s}R_n\sqrt{2(1+s)(1+t)}\leq J_n+4R_n\sqrt{1+t}.\]
		Then
		\begin{alignat*}{2}
			&\mathbb{P}\left(\sup_{g\in\mathcal{G}}\sum^S_{s=0}\left\langle\varepsilon,g^{s+1}-g^s\right\rangle_{2,P_n}\geq\frac{J_n}{\sqrt{n}}+4R_n\sqrt{\frac{1+t}{n}}\mid\mathcal{F}_n\right)\\
			&\leq\mathbb{P}\left(\sum^S_{s=0}\sup_{g\in\mathcal{G}}\left\langle\varepsilon,g^{s+1}-g^s\right\rangle_{2,P_n}\geq\frac{1}{\sqrt{n}}\sum^S_{s=0}\alpha_s\mid\mathcal{F}_n\right)\\
			&\leq\sum^S_{s=0}\mathbb{P}\left(\sup_{g\in\mathcal{G}}\left\langle\varepsilon,g^{s+1}-g^s\right\rangle_{2,P_n}\geq\frac{1}{\sqrt{n}}\alpha_s\mid\mathcal{F}_n\right)\\
			&=\sum^S_{s=0}\mathbb{P}\left(\max_{k=1,...,N_{s+1}}\left\langle\varepsilon,g^{s+1}_k-g^{s+1,s}_k\right\rangle_{2,P_n}\geq\frac{1}{\sqrt{n}}\alpha_s\mid\mathcal{F}_n\right)\\
			&\leq\sum^S_{s=0}e^{-(1+s)(1+t)}\\
			&\leq e^{-t}.
		\end{alignat*}
	\end{enumerate}
\end{proof}
Recall from Definition \ref{Dempiricalprocess} the empirical process, \(\left\{\nu_n(g)=\sqrt{n}\left(P_n-P\right)g:g\in\mathcal{G}\right\}\). Under additional conditions, we can use the previous lemma to show its asymptotic equicontinuity. We continue to assume that the envelope \(G=\sup_{g\in\mathcal{G}}\left\lVert g\right\rVert_\mathcal{Y}\) satisfies \(G\in L^2(\mathcal{X},P;\mathbb{R})\). 
\begin{theorem}\label{Tasymptoticequicontinuity}
	Suppose that \(\mathcal{G}\) satisfies the \say{uniform entropy condition}, i.e. there exists a decreasing function \(H:\mathbb{R}\rightarrow\mathbb{R}\) satisfying
	\[\int^1_0\sqrt{H(u)}du<\infty\]
	such that, for all \(u>0\) and any probability distribution \(Q\) with finite support, 
	\[H(u\left\lVert G\right\rVert_{2,Q},\mathcal{G},\lVert\cdot\rVert_{2,Q})\leq H(u).\]
	Then the empirical process \(\nu_n\) is asymptotically equicontinuous.
\end{theorem}
\begin{proof}
	Take any arbitrary \(g_0\in\mathcal{G}\). We will show that \(\nu_n\) is asymptotically equicontinuous at \(g_0\). Take arbitrary \(\epsilon_1,\epsilon_2>0\), and fix \(S\in\mathbb{N}\). Define, for \(\delta>0\), the closed \(\delta\)-ball around the origin:
	\[\mathcal{B}(\delta)\vcentcolon=\left\{g\in\mathcal{G}:\left\lVert g\right\rVert_{2,P}\leq\delta\right\}.\]
	Then clearly, the theoretical radius of \(\mathcal{B}(\delta)\) is \(\sup_{g\in\mathcal{B}(\delta)}\left\lVert g\right\rVert_{2,P}=\delta\). Denote the empirical radius of \(\mathcal{B}(\delta)\) by \(R_{n,\delta}=\sup_{g\in\mathcal{B}(\delta)}\left\lVert g\right\rVert_{2,P_n}\), and analogously to the proof of Proposition \ref{Pchaining}, define
	\[J_{n,\delta}\vcentcolon=\sum^S_{s=0}2^{-s}R_{n,\delta}\sqrt{2H\left(2^{-(s+1)}R_{n,\delta},\mathcal{B}(\delta),\lVert\cdot\rVert_{2,P_n}\right)}.\]
	Also define
	\[\mathcal{J}(\rho)\vcentcolon=8\int^\rho_0\sqrt{2H(u)}du,\qquad\rho>0,\]
	which is bounded for any finite \(\rho>0\), by the uniform entropy condition. 
	
	Define \(A\in\mathcal{F}\) as the event on which \(R_{n,\delta}\leq2\delta\) and \(\left\lVert G\right\rVert_{2,P_n}\leq2\left\lVert G\right\rVert_{2,P}\). Then on this event, we have
	\begin{alignat*}{2}
		J_{n,\delta}&=\sum^S_{s=0}2^{-s}R_{n,\delta}\sqrt{2H\left(2^{-(s+1)}R_{n,\delta},\mathcal{B}(\delta),\lVert\cdot\rVert_{2,P_n}\right)}\\
		&\leq4\int^{R_{n,\delta}}_0\sqrt{2H(u,\mathcal{B}(\delta),\lVert\cdot\rVert_{2,P_n})}du\\
		&\leq4\int^{2\delta}_0\sqrt{2H\left(u,\mathcal{G},\lVert\cdot\rVert_{2,P_n}\right)}du\qquad\text{since }R_{n,\delta}\leq2\delta\text{ on }A\text{ and }\mathcal{B}(\delta)\subseteq\mathcal{G}\\
		&\leq4\int^{2\delta}_0\sqrt{2H\left(\frac{u}{\left\lVert G\right\rVert_{2,P_n}}\right)}du\qquad\text{by the uniform entropy condition}\\
		&\leq4\int^{2\delta}_0\sqrt{2H\left(\frac{u}{2\left\lVert G\right\rVert_{2,P}}\right)}du\qquad\text{since }\left\lVert G\right\rVert_{2,P_n}\leq2\left\lVert G\right\rVert_{2,P}\text{ on }A\text{ and }H\text{ is decreasing.}\\
		&=8\left\lVert G\right\rVert_{2,P}\int^{\frac{\delta}{\left\lVert G\right\rVert_{2,P}}}_0\sqrt{2H(u)}du\qquad\text{by substitution}\\
		&=\left\lVert G\right\rVert_{2,P}\mathcal{J}\left(\frac{\delta}{\left\lVert G\right\rVert_{2,P}}\right).
	\end{alignat*}
	On \(A\), we also have
	\begin{alignat*}{2}
		\sup_{g\in\mathcal{B}(\delta)}\left\lVert P^\sigma_n\left(g-g^{S+1}\right)\right\rVert_\mathcal{Y}&\leq\sup_{g\in\mathcal{B}(\delta)}\left\lVert g-g^{S+1}\right\rVert_{1,P_n}\\
		&\leq\sup_{g\in\mathcal{B}(\delta)}\left\lVert g-g^{S+1}\right\rVert_{2,P_n}\\
		&\leq2^{-(S+1)}R_{n,\delta}\\
		&\leq2^{-S}\delta.\tag{*}
	\end{alignat*}
	So on \(A\), noting that
	\begin{alignat*}{2}
		\left\lVert P^\sigma_n\right\rVert_{\mathcal{B}(\delta)}&=\sup_{g\in\mathcal{B}(\delta)}\left\lVert P^\sigma_n\left(g-g^{S+1}\right)+\sum^S_{s=0}P^\sigma_n\left(g^{s+1}-g^s\right)\right\rVert_\mathcal{Y}\\
		&\leq\sup_{g\in\mathcal{B}(\delta)}\left\lVert P^\sigma_n\left(g-g^{S+1}\right)\right\rVert_\mathcal{Y}+\sup_{g\in\mathcal{B}(\delta)}\left\lVert\sum^S_{s=0}P^\sigma_n\left(g^{s+1}-g^s\right)\right\rVert_\mathcal{Y},
	\end{alignat*}
	we have, for all \(t>0\),
	\begin{alignat*}{2}
		&\mathbb{P}\left(\left\lVert P^\sigma_n\right\rVert_{\mathcal{B}(\delta)}\geq\frac{\sqrt{2}\left\lVert G\right\rVert_{2,P}\mathcal{J}\left(\frac{\delta}{\left\lVert G\right\rVert_{2,P}}\right)}{\sqrt{n}}+12\delta\sqrt{\frac{1+t}{n}}+2^{-S}\delta\mid\mathcal{F}_n\right)\\
		&=\mathbb{P}\left(\sup_{g\in\mathcal{B}(\delta)}\left\lVert P^\sigma_n\left(g-g^{S+1}\right)\right\rVert_\mathcal{Y}+\sup_{g\in\mathcal{B}(\delta)}\left\lVert\sum^S_{s=0}P^\sigma_n\left(g^{s+1}-g^s\right)\right\rVert_\mathcal{Y}\right.\\
		&\qquad\qquad\qquad\left.\geq\frac{\sqrt{2}J_{n,\delta}}{n}+6R_{n,\delta}\sqrt{\frac{1+t}{n}}+2^{-S}\delta\mid\mathcal{F}_n\right)\\
		&\leq\mathbb{P}\left(\sup_{g\in\mathcal{B}(\delta)}\left\lVert\sum^S_{s=0}P^\sigma_n\left(g^{s+1}-g^s\right)\right\rVert_\mathcal{Y}\geq\frac{\sqrt{2}J_{n,\delta}}{\sqrt{n}}+6R_{n,\delta}\sqrt{\frac{1+t}{n}}\mid\mathcal{F}_n\right)\\
		&\leq2e^{-t},
	\end{alignat*}
	where the term \(\mathbb{P}\left(\sup_{g\in\mathcal{B}(\delta)}\left\lVert P^\sigma_n\left(g-g^{S+1}\right)\right\rVert_\mathcal{Y}\geq2^{-S}\delta\mid\mathcal{F}_n\right)\) vanishes by (*) and the last inequality follows Proposition \ref{Pchaining}(i). Then we can de-symmetrise using Lemma \ref{Lsymmetrisationprobability}:
	\begin{alignat*}{2}
		&\mathbb{P}\left(\left\lVert P_n-P\right\rVert_{\mathcal{B}(\delta)}\geq\frac{4\sqrt{2}\left\lVert G\right\rVert_{2,P}\mathcal{J}\left(\frac{\delta}{\left\lVert G\right\rVert_{2,P}}\right)}{\sqrt{n}}+48\delta\sqrt{\frac{1+t}{n}}+2^{-(S-2)}\delta\right)\\
		&\leq4\mathbb{P}\left(\left\lVert P^\sigma_n\right\rVert_{\mathcal{B}(\delta)}\geq\frac{\sqrt{2}\left\lVert G\right\rVert_{2,P}\mathcal{J}\left(\frac{\delta}{\left\lVert G\right\rVert_{2,P}}\right)}{\sqrt{n}}+12\delta\sqrt{\frac{1+t}{n}}+2^{-S}\delta\right)\\
		&\leq8e^{-t}+4\mathbb{P}\left(R_{n,\delta}>2\delta\text{ or }\left\lVert G\right\rVert_{2,P_n}>2\left\lVert G\right\rVert_{2,P}\right)\\
		&=8e^{-t}+4\mathbb{P}\left(\sup_{g\in\mathcal{B}(\delta)\cup\{G\}}\left\lVert g\right\rVert^2_{2,P_n}>4\sup_{g\in\mathcal{B}(\delta)\cup\{G\}}\left\lVert g\right\rVert^2_{2,P}\right).
	\end{alignat*}
	Now let \(t=\log\left(\frac{8}{\epsilon_2}\right)\) and \(S\) large enough such that \(2^{-(S-2)}\leq\frac{1}{\sqrt{n}}\), and \(\delta\) small enough such that
	\[4\sqrt{2}\left\lVert G\right\rVert_{2,P}\mathcal{J}\left(\frac{\delta}{\left\lVert G\right\rVert_{2,P}}\right)+48\delta\sqrt{1+\log\left(\frac{8}{\epsilon_2}\right)}+\delta\leq\epsilon_1.\]
	Then
	\[\mathbb{P}\left(\sqrt{n}\left\lVert P_n-P\right\rVert_{\mathcal{B}(\delta)}>\epsilon_1\right)\leq\epsilon_2+4\mathbb{P}\left(\sup_{g\in\mathcal{B}(\delta)\cup\{G\}}\left\lVert g\right\rVert^2_{2,P_n}>4\sup_{g\in\mathcal{B}(\delta)\cup\{G\}}\left\lVert g\right\rVert^2_{2,P}\right).\]
	Hence, for any \(g\in\mathcal{G}\) such that \(\left\lVert g-g_0\right\rVert_{2,P}\leq\delta\), 
	\begin{alignat*}{2}
		\mathbb{P}\left(\left\lVert\nu_n(g)-\nu_n(g_0)\right\rVert_\mathcal{Y}>\epsilon_1\right)&=\mathbb{P}\left(\sqrt{n}\left\lVert\left(P_n-P\right)\left(g-g_0\right)\right\rVert_\mathcal{Y}>\epsilon_1\right)\\
		&\leq\mathbb{P}\left(\sqrt{n}\left\lVert P_n-P\right\rVert_{\mathcal{B}(\delta)}>\epsilon_1\right)\\
		&\leq\epsilon_2+4\mathbb{P}\left(\sup_{g\in\mathcal{B}(\delta)\cup\{G\}}\left\lVert g\right\rVert^2_{2,P_n}>4\sup_{g\in\mathcal{B}(\delta)\cup\{G\}}\left\lVert g\right\rVert^2_{2,P}\right).
	\end{alignat*}
	Here, by the uniform law of large numbers on \(\mathcal{B}(\delta)\cap\{G\}\) (Theorem \ref{Tglivenkocantelli}), the second term converges to \(0\) as \(n\rightarrow\infty\). Hence, as \(\epsilon_1\) and \(\epsilon_2\) were arbitrary, we have asymptotic equicontinuity. 
\end{proof}

\subsection{Peeling and Least-Squares Regression with Fixed Design and Gaussian Noise}\label{SSpeelingappendix}
\begin{theorem}\label{Tleastsquares}
	Suppose that \(\varepsilon_1,...,\varepsilon_n\) are i.i.d. with Gaussian distribution with mean 0 and covariance operator \(Q\) (c.f. Definition \ref{Dgaussianhilbert} and Lemmas \ref{Lgaussianintegrable} and \ref{Lgaussiancovariance}), and that \(\textnormal{Tr}\,Q=1\). Further, suppose that
	\[J(\delta)\vcentcolon=4\int^\delta_0\sqrt{2H(u,\mathcal{B}_{2,P_n}(\delta),\lVert\cdot\rVert_{2,P_n})}du<\infty,\quad\text{for each }\delta>0\text{, and }\frac{J(\delta)}{\delta^2}\text{ is decreasing in }\delta,\]
	where \(\mathcal{B}_{2,P_n}(\delta)\vcentcolon=\{g\in\mathcal{G}:\left\lVert g\right\rVert_{2,P_n}\leq\delta\}\). Then for all \(t\geq\frac{3}{8}\) and all \(\delta_n\) satisfying
	\[\sqrt{n}\delta_n^2\geq8\left(J(\delta_n)+4\delta_n\sqrt{1+t}+\delta_n\sqrt{\frac{8}{3}t}\right),\]
	we have
	\[\mathbb{P}\left(\left\lVert\hat{g}_n-g_0\right\rVert_{2,P_n}>\delta_n\right)\leq\left(1+\frac{2}{e-1}\right)e^{-t}.\]
\end{theorem}
\begin{proof}
	First, recall the notation
	\[\left\langle\varepsilon,g\right\rangle_{2,P_n}=\frac{1}{n}\sum^n_{i=1}\left\langle\varepsilon_i,g(X_i)\right\rangle_\mathcal{Y}\]
	from Proposition \ref{Pchaining}(ii), and note that we have the following basic inequality
	\[\left\lVert\hat{g}_n-g_0\right\rVert^2_{2,P_n}\leq2\left\langle\varepsilon,\hat{g}_n-g_0\right\rangle_{2,P_n},\tag{*}\]
	which follows from the fact that \(\hat{g}_n\) minimises \(\left\lVert Y_i-g(X_i)\right\rVert^2_{2,P_n}\) over \(g\in\mathcal{G}\), giving
	\[\left\lVert\varepsilon_i-(g_0-\hat{g}_n)\right\rVert^2_{2,P_n}=\left\lVert Y_i-\hat{g}_n(X_i)\right\rVert_{2,P_n}^2\leq\left\lVert Y_i-g_0(X_i)\right\rVert_{2,P_n}^2=\left\lVert\varepsilon_i\right\rVert_{2,P_n}^2.\]
	We use a technique called the \say{peeling device}, first introduced by \citet{vandegeer2000empirical}. See that
	\begin{alignat*}{2}
		\mathbb{P}&\left(\left\lVert\hat{g}_n-g_0\right\rVert_{2,P_n}>\delta_n\right)=\mathbb{P}\left(\bigcup_{j=1}^\infty\left\{2^{j-1}\delta_n<\left\lVert\hat{g}_n-g_0\right\rVert_{2,P_n}\leq2^j\delta_n\right\}\right)\\
		&\leq\sum^\infty_{j=1}\mathbb{P}\left(2^{j-1}\delta_n<\left\lVert\hat{g}_n-g_0\right\rVert_{2,P_n}\leq2^j\delta_n\right)\qquad\text{by the union bound}\\
		&=\sum^\infty_{j=1}\mathbb{P}\left(\left\{2^{j-1}\delta_n<\left\lVert\hat{g}_n-g_0\right\rVert_{2,P_n}\right\}\bigcap\left\{\hat{g}_n-g_0\in\mathcal{B}_n(2^j\delta_n)\right\}\right)\\
		&\leq\sum^\infty_{j=1}\mathbb{P}\left(\left\{\left(2^{j-1}\delta_n\right)^2<2\left\langle\varepsilon,\hat{g}_n-g_0\right\rangle_{2,P_n}\right\}\bigcap\left\{\hat{g}_n-g_0\in\mathcal{B}_n(2^j\delta_n)\right\}\right)\qquad\text{by (*)}\\
		&\leq\sum^\infty_{j=1}\mathbb{P}\left(\sup_{g\in\mathcal{B}_n(2^j\delta_n)}2\left\langle\varepsilon,g\right\rangle_{2,P_n}>\left(2^{j-1}\delta_n\right)^2\right)\\
		&=\sum^\infty_{j=1}\mathbb{P}\left(\sup_{g\in\mathcal{B}_n(2^j\delta_n)}\left\langle\varepsilon,g\right\rangle_{2,P_n}>\frac{1}{8}\left(2^j\delta_n\right)^2\right).
	\end{alignat*}
	Now, applying the hypothesis on \(\delta_n\), we see that, for each \(j\),
	\begin{alignat*}{2}
		\frac{1}{8}\left(2^j\delta_n\right)^2&\geq\frac{(2^j)^2J(\delta_n)}{\sqrt{n}}+4(2^j)^2\delta_n\sqrt{\frac{1+t}{n}}+\frac{\sqrt{\frac{8}{3}t}(2^j)^2\delta_n}{\sqrt{n}}\\
		&\geq\frac{J(2^j\delta_n)}{\sqrt{n}}+4(2^j\delta_n)\sqrt{\frac{1+t+j}{n}}+\frac{\sqrt{\frac{8}{3}t}(2^j)^2\delta_n}{\sqrt{n}}\\
		&\geq\frac{J_n}{\sqrt{n}}+4(2^j\delta_n)\sqrt{\frac{1+t+j}{n}}+\frac{\sqrt{\frac{8}{3}t}(2^j)^2\delta_n}{\sqrt{n}}
	\end{alignat*}
	where we used the fact that \(\frac{J(\delta)}{\delta^2}\) is decreasing in \(\delta\) and \(\sqrt{1+t+j}\leq2^j\sqrt{1+t}\), and \(J_n\) is defined as in Proposition \ref{Pchaining} with \(\mathcal{G}=\mathcal{B}_n(2^j\delta_n)\) and \(R_n=2^j\delta_n\). On the other hand, we can write, for any \(S\in\mathbb{N}\),
	\[\left\langle\varepsilon,g\right\rangle_{2,P_n}=\left\langle\varepsilon,g-g^{S+1}\right\rangle_{2,P_n}+\sum^S_{s=0}\left\langle\varepsilon,g^{s+1}-g^s\right\rangle_{2,P_n},\]
	using the chaining notation in Section \ref{SSchaining}. Hence,
	\begin{alignat*}{2}
		&\mathbb{P}\left(\left\lVert\hat{g}_n-g_0\right\rVert_{2,P_n}>\delta_n\right)\\
		&\leq\sum^\infty_{j=1}\mathbb{P}\left(\sup_{g\in\mathcal{B}_n(2^j\delta_n)}\left\langle\varepsilon,g-g^{S+1}\right\rangle_{2,P_n}>\frac{\sqrt{\frac{8}{3}t}(2^j)^2\delta_n}{\sqrt{n}}\right)\\
		&\quad+\sum^\infty_{j=1}\mathbb{P}\left(\sup_{g\in\mathcal{B}_n(2^j\delta_n)}\sum^S_{s=0}\left\langle\varepsilon,g^{s+1}-g^s\right\rangle_{2,P_n}>\frac{J_n}{\sqrt{n}}+4(2^j\delta_n)\sqrt{\frac{1+t+j}{n}}\right)\\
		&\leq\sum^\infty_{j=1}\mathbb{P}\left(\frac{2^j}{2^{S+1}}\delta_n\left\lVert\varepsilon\right\rVert_{2,P_n}>\frac{\sqrt{\frac{8}{3}t}2^{2j}\delta_n}{\sqrt{n}}\right)+\sum^\infty_{j=1}e^{-(t+j)}\enspace\text{by Proposition \ref{Pchaining}(ii)}\\
		&=\sum^\infty_{j=1}\mathbb{P}\left(\left\lVert\varepsilon\right\rVert_{2,P_n}>2^j\sqrt{\frac{8}{3}t}\right)+\frac{1}{e-1}e^{-t}\enspace\text{letting }S\text{ such that }\sqrt{n}\leq2^{S+1}\\
		&\leq\sum^\infty_{j=1}\mathbb{P}\left(\left\lVert\varepsilon\right\rVert_{2,P_n}>2^j+\sqrt{\frac{8}{3}t}\right)+\frac{1}{e-1}e^{-t}\enspace\text{since }t\geq\frac{3}{8}\\
		&\leq\sum^\infty_{j=1}\mathbb{P}\left(\frac{1}{n}\sum^n_{i=1}\left\lVert\varepsilon_i\right\rVert_\mathcal{Y}^2>2^{2j}+\frac{8}{3}t\right)+\frac{1}{e-1}e^{-t}\\
		&\leq\sum^\infty_{j=1}e^{-\frac{3}{8}2^{2j}-t}\mathbb{E}\left[e^{\frac{3}{8}\frac{1}{n}\sum^n_{i=1}\left\lVert\varepsilon_i\right\rVert^2_\mathcal{Y}}\right]+\frac{1}{e-1}e^{-t}\qquad\text{by Markov's inequality}\\
		&\leq\sum^\infty_{j=1}e^{-\frac{3}{8}2^{2j}-t}\prod^n_{i=1}\mathbb{E}\left[e^{\frac{3}{8}\frac{1}{n}\left\lVert\varepsilon_i\right\rVert^2_\mathcal{Y}}\right]+\frac{1}{e-1}e^{-t}\qquad\text{by independence}\\
		&\leq\sum^\infty_{j=1}e^{-\frac{3}{8}2^{2j}-t}\mathbb{E}\left[e^{\frac{3}{8}\left\lVert\varepsilon_1\right\rVert^2_\mathcal{Y}}\right]+\frac{1}{e-1}e^{-t}\qquad\text{by Jensen's inequality}\\
		&\leq\sum^\infty_{j=1}e^{-\frac{3}{8}2^{2j}-t}+\frac{1}{e-1}e^{-t}\qquad\text{by Proposition \ref{Pconcentrationgaussian}}\\
		&\leq e^{-t}\left(e^{-\frac{3}{4}}-e^{-1}+\sum^\infty_{j=1}e^{-j}+\frac{1}{e-1}\right)\\
		&\leq e^{-t}\left(1+\frac{2}{e-1}\right).
	\end{alignat*}
\end{proof}
\subsection{Rademacher Complexities}\label{SSrademacherappendix}
In this Section, we discuss the extension of the concept of Rademacher complexities to classes of vector-valued functions in more depth (c.f. Section \ref{SSrademacher}). We first give the definition of Rademacher complexities of classes of real-valued functions. 
\begin{definition}[{\citet[Definition 2]{bartlett2002rademacher}}]\label{Drademacherreal}
	Suppose \(\mathcal{G}\) is a class of real-valued functions \(\mathcal{X}\rightarrow\mathbb{R}\). Then the \textit{empirical (or conditional) Rademacher complexity} of \(\mathcal{G}\) is defined as
	\[\hat{\mathfrak{R}}_n(\mathcal{G})=\mathbb{E}\left[\sup_{g\in\mathcal{G}}\left\lvert\frac{1}{n}\sum^n_{i=1}\sigma_ig(X_i)\right\rvert\mid X_1,...,X_n\right],\]
	where the expectation is taken with respect to the Rademacher variables \(\{\sigma_i\}_{i=1}^n\). The Rademacher complexity of \(\mathcal{G}\) is defined as
	\[\mathfrak{R}_n(G)=\mathbb{E}\left[\hat{\mathfrak{R}}_n(G)\right].\]
\end{definition}
Since this seminal definition, it was realised that the absolute value around \(\frac{1}{n}\sum^n_{i=1}\sigma_ig(X_i)\) was unnecessary (see, for example, \citet[paragraph between Corollary 4 and Lemma 5]{meir2003generalization} or \citet[last paragraph of Section 1]{maurer2016vector}). However, in order to facilitate the following direct extension to classes of vector-valued functions, we retain the absolute value sign. 
\begin{definition}\label{Drademachervector}
	Suppose \(\mathcal{G}\) is a class of \(\mathcal{X}\rightarrow\mathcal{Y}\) functions. Then the empirical (or conditional) Rademacher complexity of \(\mathcal{G}\) is defined as
	\[\hat{\mathfrak{R}}_n(\mathcal{G})=\mathbb{E}\left[\sup_{g\in\mathcal{G}}\left\lVert\frac{1}{n}\sum^n_{i=1}\sigma_ig(X_i)\right\rVert_\mathcal{Y}\mid X_1,...,X_n\right]=\mathbb{E}\left[\left\lVert P^\sigma_ng\right\rVert_\mathcal{G}\mid X_1,...,X_n\right],\]
	using the notation from Section \ref{Sempiricalprocesstheory}. The Rademacher complexity of \(\mathcal{G}\) is defined as
	\[\mathfrak{R}_n(G)=\mathbb{E}\left[\hat{\mathfrak{R}}_n(G)\right].\]
\end{definition}
Note that our definition is different to the \say{vector-valued Rademacher complexity} already in use in the literature, mostly for \(\mathcal{Y}\) being a finite-dimensional Euclidean space (\citeauthor{yousefi2018local}, \citeyear{yousefi2018local}, Definition 1; \citeauthor{li2019learning}, \citeyear{li2019learning}, Definition 3), but also for \(\mathcal{Y}=l_2\), the space of square-summable sequences \citep{maurer2016vector}. These papers define the \say{Rademacher complexity} of vector-valued function classes not as in Definition \ref{Drademachervector}, where we have one Rademacher variable \(\sigma_i\) per sample \(X_i\), but introduce a Rademacher variable for every coordinate of \(\mathcal{Y}\). The resulting quantity looks something like
\[\mathbb{E}\left[\sup_{g\in\mathcal{G}}\frac{1}{n}\sum^n_{i=1}\sum_k\sigma^k_ig_k(X_i)\mid X_1,...,X_n\right],\]
where \(g_k\) is the \(k^\text{th}\) coordinate of \(g\) with respect to a basis, and \(\{\sigma^k_i\}_{i,k}\) are Rademacher random variables. For convenience, in what follows, we call this the \say{coordinate-wise Rademacher complexity}, and denote it by \(\hat{\mathfrak{R}}^\text{coord}_n(\mathcal{G})\). 

While we recognise the usefulness of this definition, especially thanks to the contraction result shown in \citet{maurer2016vector}, \citet{cortes2016structured}, \citet{zatarain2019vector} and \citet{foster2019vector}, for several reasons, we insist on using Definition \ref{Drademachervector}. Firstly, as it is clear from the definition, and as admitted by \citet[paragraph just above Conjecture 2]{maurer2016vector}, Definition \ref{Drademachervector} is a more natural definition in view of the real-valued Rademacher complexity. Moreover, our work in Section \ref{SSsymmetrisation} uses the empirical symmetrised measure \(\frac{1}{n}\sum^n_{i=1}\sigma_i\delta_{X_i}\) to good effect and in a way that directly generalises from the real-valued case, which suggests that Definition \ref{Drademachervector} is natural. Finally, and perhaps most critically, the coordinate-wise Rademacher complexity is not independent of the choice of the basis of \(\mathcal{Y}\). For a simple counterexample, let \(\mathcal{X}=\mathcal{Y}=\mathbb{R}^2\), and \(\mathcal{G}=\{g_1,g_2\}\), where \(g_1\) is the orthogonal projection onto the line \(y=x\), and \(g_2\) is the orthogonal projection onto the line \(y=-x\). This means that, letting \(X_1=\begin{pmatrix}1\\0\end{pmatrix}\) and \(X_2=\begin{pmatrix}0\\1\end{pmatrix}\), we have
\[g_1(X_1)=\begin{pmatrix}\frac{1}{2}\\\frac{1}{2}\end{pmatrix},\quad g_1(X_2)=\begin{pmatrix}\frac{1}{2}\\\frac{1}{2}\end{pmatrix},\quad g_2(X_1)=\begin{pmatrix}\frac{1}{2}\\-\frac{1}{2}\end{pmatrix},\quad g_2(X_2)=\begin{pmatrix}-\frac{1}{2}\\\frac{1}{2}\end{pmatrix}.\]
Then the coordinate-wise Rademacher complexity of \(\mathcal{G}\) with respect to the standard basis \(\{X_1,X_2\}\) is
\begin{alignat*}{2}
	\hat{\mathfrak{R}}_n^\text{coord}(\mathcal{G})&=\mathbb{E}\left[\sup_{g\in\mathcal{G}}\sum^2_{i=1}\sum^2_{k=1}\sigma^k_ig_k(X_i)\right]\\
	&=\mathbb{E}\left[\sup_{g\in\mathcal{G}}\left\{\sigma^1_1\left(g(X_1)\right)_1+\sigma^2_1\left(g(X_1)\right)_2+\sigma^1_2\left(g(X_2)\right)_1+\sigma^2_2\left(g(X_2)\right)_2\right\}\right]\\
	&=\mathbb{E}\left[\frac{\sigma^1_1}{2}+\frac{\sigma^2_2}{2}+\sup_{g\in\mathcal{G}}\left\{\sigma^2_1\left(g(X_1)\right)_2+\sigma^1_2\left(g(X_2)\right)_1\right\}\right]\\
	&=\sup_{g\in\mathcal{G}}\left\{\left(g(X_1)\right)_2+\left(g(X_2)\right)_1\right\}+\sup_{g\in\mathcal{G}}\left\{-\left(g(X_1)\right)_2+\left(g(X_2)\right)_1\right\}\\
	&\qquad+\sup_{g\in\mathcal{G}}\left\{\left(g(X_1)\right)_2-\left(g(X_2)\right)_1\right\}+\sup_{g\in\mathcal{G}}\left\{-\left(g(X_1)\right)_2-\left(g(X_2)\right)_1\right\}\\
	&=1+0+0+1\\
	&=2.
\end{alignat*}
But if we use the orthonormal basis \(\left\{\begin{pmatrix}\frac{1}{\sqrt{2}}\\\frac{1}{\sqrt{2}}\end{pmatrix},\begin{pmatrix}-\frac{1}{\sqrt{2}}\\\frac{1}{\sqrt{2}}\end{pmatrix}\right\}\), then we have
\begin{alignat*}{2}
	&(g_1(X_1))_1=\frac{1}{\sqrt{2}},\quad(g_1(X_1))_2=0,\quad(g_1(X_2))_1=\frac{1}{\sqrt{2}}\quad(g_1(X_2))_2=0\\
	&(g_2(X_1))_1=0,\qquad(g_2(X_1))_2=-\frac{1}{\sqrt{2}},\qquad(g_2(X_2))_1=0,\quad(g_2(X_2))_2=\frac{1}{\sqrt{2}}.
\end{alignat*}
So the complexity with respect to the standard basis \(\{X_1,X_2\}\) is
\begin{alignat*}{2}
	&\mathbb{E}\left[\sup_{g\in\mathcal{G}}\sum^2_{i=1}\sum^2_{k=1}\sigma^k_ig_k(X_i)\right]\\
	&=\mathbb{E}\left[\sup_{g\in\mathcal{G}}\left\{\sigma^1_1\left(g(X_1)\right)_1+\sigma^2_1\left(g(X_1)\right)_2+\sigma^1_2\left(g(X_2)\right)_1+\sigma^2_2\left(g(X_2)\right)_2\right\}\right]\\
	&=\sup_{g\in\mathcal{G}}\left\{\left(g(X_1)\right)_1+\left(g(X_1)\right)_2+\left(g(X_2)\right)_1+\left(g(X_2)\right)_2\right\}\\
	&\qquad+\sup_{g\in\mathcal{G}}\left\{\left(g(X_1)\right)_1+\left(g(X_1)\right)_2+\left(g(X_2)\right)_1-\left(g(X_2)\right)_2\right\}\\
	&\qquad+\sup_{g\in\mathcal{G}}\left\{\left(g(X_1)\right)_1+\left(g(X_1)\right)_2-\left(g(X_2)\right)_1+\left(g(X_2)\right)_2\right\}\\
	&\qquad+\sup_{g\in\mathcal{G}}\left\{\left(g(X_1)\right)_1-\left(g(X_1)\right)_2+\left(g(X_2)\right)_1+\left(g(X_2)\right)_2\right\}\\
	&\qquad+\sup_{g\in\mathcal{G}}\left\{-\left(g(X_1)\right)_1+\left(g(X_1)\right)_2+\left(g(X_2)\right)_1+\left(g(X_2)\right)_2\right\}\\
	&\qquad+\sup_{g\in\mathcal{G}}\left\{\left(g(X_1)\right)_1+\left(g(X_1)\right)_2-\left(g(X_2)\right)_1-\left(g(X_2)\right)_2\right\}\\
	&\qquad+\sup_{g\in\mathcal{G}}\left\{\left(g(X_1)\right)_1-\left(g(X_1)\right)_2+\left(g(X_2)\right)_1-\left(g(X_2)\right)_2\right\}\\
	&\qquad+\sup_{g\in\mathcal{G}}\left\{-\left(g(X_1)\right)_1+\left(g(X_1)\right)_2+\left(g(X_2)\right)_1-\left(g(X_2)\right)_2\right\}\\
	&\qquad+\sup_{g\in\mathcal{G}}\left\{\left(g(X_1)\right)_1-\left(g(X_1)\right)_2-\left(g(X_2)\right)_1+\left(g(X_2)\right)_2\right\}\\
	&\qquad+\sup_{g\in\mathcal{G}}\left\{-\left(g(X_1)\right)_1+\left(g(X_1)\right)_2-\left(g(X_2)\right)_1+\left(g(X_2)\right)_2\right\}\\
	&\qquad+\sup_{g\in\mathcal{G}}\left\{-\left(g(X_1)\right)_1-\left(g(X_1)\right)_2+\left(g(X_2)\right)_1+\left(g(X_2)\right)_2\right\}\\
	&\qquad+\sup_{g\in\mathcal{G}}\left\{\left(g(X_1)\right)_1-\left(g(X_1)\right)_2-\left(g(X_2)\right)_1-\left(g(X_2)\right)_2\right\}\\
	&\qquad+\sup_{g\in\mathcal{G}}\left\{-\left(g(X_1)\right)_1+\left(g(X_1)\right)_2-\left(g(X_2)\right)_1-\left(g(X_2)\right)_2\right\}\\
	&\qquad+\sup_{g\in\mathcal{G}}\left\{-\left(g(X_1)\right)_1-\left(g(X_1)\right)_2+\left(g(X_2)\right)_1-\left(g(X_2)\right)_2\right\}\\
	&\qquad+\sup_{g\in\mathcal{G}}\left\{-\left(g(X_1)\right)_1-\left(g(X_1)\right)_2-\left(g(X_2)\right)_1+\left(g(X_2)\right)_2\right\}\\
	&\qquad+\sup_{g\in\mathcal{G}}\left\{-\left(g(X_1)\right)_1-\left(g(X_1)\right)_2-\left(g(X_2)\right)_1-\left(g(X_2)\right)_2\right\}\\
	&=\sqrt{2}+\sqrt{2}+0+\sqrt{2}+0+0+\sqrt{2}+0+\sqrt{2}+0+\sqrt{2}+0-\sqrt{2}+0+0+0\\
	&=5\sqrt{2}.
\end{alignat*}
Hence, we see that the coordinate-wise Rademacher complexity is not independent of the chosen orthonormal basis. We deem this to be a critical issue with the coordinate-wise Rademacher complexity, because it is intuitively clear that the \say{complexity} of a function class should not depend on the choice of the basis of the output space. This is especially pertinent in our context, considering that our interest is primarily in the case when the output space \(\mathcal{Y}\) is infinite-dimensional in which there may be no \say{standard basis}. 

One of the main ways of bounding the Rademacher complexity of real-valued function classes is to use the entropy. We show that the Rademacher complexity of vector-valued function classes \(\mathcal{G}\) can be bounded using the entropy, a vector-valued analogue of \citet[p.338, Lemma 27.4]{shalev2014understanding}. We use the chaining notation in Section \ref{SSchaining}, and also use Proposition \ref{Phoeffdinghilbertexpectation}, the expectation form of vector-valued Hoeffding's inequality.
\begin{theorem}\label{Trademacherboundentropy}
	Let \(S\in\mathbb{N}\) be any (large) integer. The empirical Rademacher complexity is bounded as
	\[\hat{\mathfrak{R}}_n(\mathcal{G})\leq2^{-(S+1)}R_n+\frac{2}{\sqrt{n}}J_n,\]
	where we recall that \(R_n=\sup_{g\in\mathcal{G}}\lVert g\rVert_{2,P_n}\) is the empirical radius and \(J_n=\sum^S_{s=0}2^{-s}R_n\sqrt{2H_{s+1}}\) is the uniform entropy bound. 
\end{theorem}
\begin{proof}
	See that
	\begin{alignat*}{2}
		\hat{\mathfrak{R}}_n(\mathcal{G})&=\mathbb{E}\left[\sup_{g\in\mathcal{G}}\left\lVert\frac{1}{n}\sum^n_{i=1}\sigma_ig(X_i)\right\rVert_\mathcal{Y}\mid\mathcal{F}_n\right]\\
		&=\mathbb{E}\left[\sup_{g\in\mathcal{G}}\left\lVert P^\sigma_ng\right\rVert_\mathcal{Y}\mid\mathcal{F}_n\right]\\
		&=\mathbb{E}\left[\sup_{g\in\mathcal{G}}\left\lVert P^\sigma_n\left(g-g^{S+1}\right)+\sum^S_{s=0}P^\sigma_n\left(g^{s+1}-g^s\right)\right\rVert_\mathcal{Y}\mid\mathcal{F}_n\right]\\
		&\leq\mathbb{E}\left[\sup_{g\in\mathcal{G}}\left\lVert P^\sigma_n\left(g-g^{S+1}\right)\right\rVert_\mathcal{Y}\mid\mathcal{F}_n\right]+\mathbb{E}\left[\sup_{g\in\mathcal{G}}\left\lVert\sum^S_{s=0}P^\sigma_n\left(g^{s+1}-g^s\right)\right\rVert_\mathcal{Y}\mid\mathcal{F}_n\right]\\
		&\leq\sup_{g\in\mathcal{G}}\frac{1}{n}\sum^n_{i=1}\left\lVert g(X_i)-g^{S+1}(X_i)\right\rVert_\mathcal{Y}+\sum^S_{s=0}\mathbb{E}\left[\sup_{g\in\mathcal{G}}\left\lVert P^\sigma_n\left(g^{s+1}-g^s\right)\right\rVert_\mathcal{Y}\mid\mathcal{F}_n\right]\\
		&\leq\sup_{g\in\mathcal{G}}\left\lVert g-g^{S+1}\right\rVert_{2,P_n}+\sum^S_{s=0}\mathbb{E}\left[\max_{k\in\{1,...,N_{s+1}\}}\left\lVert P^\sigma_n\left(g^{s+1}_k-g^{s+1,s}_k\right)\right\rVert_\mathcal{Y}\mid\mathcal{F}_n\right]\\
		&\leq2^{-(S+1)}R_n+\sum^S_{s=0}\frac{1}{\lambda_s}\log\left(\mathbb{E}\left[\sum^{N_{s+1}}_{k=1}e^{\lambda_s\left\lVert P^\sigma_n\left(g^{s+1}_k-g^{s+1,s}_k\right)\right\rVert_\mathcal{Y}}\mid\mathcal{F}_n\right]\right)\qquad(a)\\
		&\leq2^{-(S+1)}R_n+\sum^S_{s=0}\frac{1}{\lambda_s}\log\left(\sum^{N_{s+1}}_{k=1}\mathbb{E}\left[2\cosh\left(\lambda_s\left\lVert P^\sigma_n\left(g^{s+1}_k-g^{s+1,s}_k\right)\right\rVert_\mathcal{Y}\right)\mid\mathcal{F}_n\right]\right)(b)\\
		&\leq2^{-(S+1)}R_n+\sum^S_{s=0}\frac{1}{\lambda_s}\log\left(2\sum_{k=1}^{N_{s+1}}e^{\frac{\lambda_s^2}{n}(2^{-s}R_n)^2}\right)\qquad(c)\\
		&=2^{-(S+1)}R_n+\sum^S_{s=0}\frac{1}{\lambda_s}\log\left(2N_{s+1}e^{\frac{\lambda_s^2}{n}(2^{-s}R_n)^2}\right)\\
		&=2^{-(S+1)}R_n+\sum^S_{s=0}\frac{1}{\lambda_s}\left(H_{s+1}+\log2\right)+\frac{\lambda_s}{n}\sum^S_{s=0}(2^{-s}R_n)^2\\
		&\leq2^{-(S+1)}R_n+\sum^S_{s=0}\frac{1}{\lambda_s}2H_{s+1}+\frac{\lambda_s}{n}\sum^S_{s=0}(2^{-s}R_n)^2\qquad(d)\\
		&=2^{-(S+1)}R_n+\frac{2}{\sqrt{n}}\sum^S_{s=0}2^{-s}R_n\sqrt{2H_{s+1}}\qquad(e)\\
		&=2^{-(S+1)}R_n+\frac{2}{\sqrt{n}}J_n
	\end{alignat*}
	where, in (a), we used Jensen's inequality and the fact that the sum of positive numbers is greater than their maximum; in (b), we used the basic fact \(e^x\leq2\cosh x\); in (c), we used Proposition \ref{Phoeffdinghilbertexpectation}; in (d), we used the fact that \(H_{s+1}\geq\log2\); and in (e), we let
	\[\lambda_s=\frac{\sqrt{2nH_{s+1}}}{2^{-s}R_n}.\]
\end{proof}

When the Rademacher complexity is used in empirical risk minimisation for real-valued function classes \(\mathcal{F}\), what we end up using is not the Rademacher complexity \(\mathfrak{R}_n(\mathcal{F})\) of the function class itself, but that of the composition of the loss with the function class. The same is true for vector-valued empirical risk minimisation problems. More precisely, suppose we have a loss function \(\mathcal{L}:\mathcal{Y}\times\mathcal{Y}\rightarrow\mathbb{R}\), and we denote by \(\hat{g}_n\) the solution of the following empirical risk minimisation problem:
\[\hat{g}_n=\argmin_{g\in\mathcal{G}}\frac{1}{n}\sum^n_{i=1}\mathcal{L}(Y_i,g(X_i))=\argmin_{g\in\mathcal{G}}\hat{\mathcal{R}}_n(g).\]
Denote by \(g^*\) the minimiser of the population risk:
\[g^*\vcentcolon=\argmin_{g\in\mathcal{G}}\mathbb{E}\left[\mathcal{L}(Y,g(X))\right]=\argmin_{g\in\mathcal{G}}\mathcal{R}(g).\]
We want to know how fast \(\mathcal{R}(\hat{g}_n)\) converges to the minimal risk \(\mathcal{R}(g^*)\) as the sample size \(n\) increases. Here, actually, the standard result concerning Rademacher complexities applies directly -- we will quote the following result. 
\begin{theorem}[{\citet[p.328, Theorem 26.5]{shalev2014understanding}}]\label{Tshalev}
	Assume that for all \((x,y)\in\mathcal{X}\times\mathcal{Y}\) and \(g\in\mathcal{G}\), we have \(\lvert\mathcal{L}(y,g(x))\rvert\leq c\) for some constant \(c>0\). Then with probability at least \(1-\delta\), we have
	\[\mathcal{R}(\hat{g}_n)-\mathcal{R}(g^*)\leq2\mathfrak{R}_n(\mathcal{L}\circ\mathcal{G})+5c\sqrt{\frac{2\log\left(\frac{8}{\delta}\right)}{n}}\]
	where we used the notation \(\mathcal{L}\circ\mathcal{G}\) for the class of functions \(\mathcal{X}\times\mathcal{Y}\rightarrow\mathbb{R}\) defined as
	\[\mathcal{L}\circ\mathcal{G}\vcentcolon=\left\{(x,y)\mapsto\mathcal{L}(y,g(x)):g\in\mathcal{G}\right\}.\]
\end{theorem}

Now, the question is how to obtain a meaningful bound on the Rademacher complexity \(\mathfrak{R}_n(\mathcal{L}\circ\mathcal{G})\) as \(n\rightarrow\infty\). When \(\mathcal{G}\) is a class of real-valued functions, the Contraction Lemma \citep[p.331, Lemma 26.9]{shalev2014understanding} tells us that if, for each \(Y_i\in\mathbb{R}\), the map \(y\mapsto\mathcal{L}(Y_i,y)\) is \(c\)-Lipschitz, then \(\mathfrak{R}_n(\mathcal{L}\circ\mathcal{G})\) is bounded by \(c\mathfrak{R}_n(\mathcal{G})\), so it is meaningful to work with \(\mathfrak{R}_n(\mathcal{G})\). However, an analogue of this result when \(\mathcal{G}\) is a class of \(\mathcal{Y}\)-valued functions is shown to be impossible via a counterexample, in \citet[Section 6]{maurer2016vector}. 

As mentioned above, one of the main ways of bounding the Rademacher complexity is to use entropy. As our end goal is to bound the Rademacher complexity of \(\mathcal{L}\circ\mathcal{G}\), there are two ways of going about this task with entropy. For real-valued function classes \(\mathcal{F}\), what is commonly done is to bound the Rademacher complexity of \(\mathcal{L}\circ\mathcal{F}\) with the Rademacher complexity of \(\mathcal{F}\) using contraction, then to bound the Rademacher complexity of \(\mathcal{F}\) by an expression involving the entropy, using chaining. As discussed before, contraction becomes difficult with vector-valued function classes. But we propose a different way that avoids contraction of Rademacher complexities. We can first bound the Rademacher complexity of \(\mathcal{L}\circ\mathcal{G}\) with an expression involving the entropy of \(\mathcal{L}\circ\mathcal{G}\), and use the following contraction result of entropies. 
\begin{lemma}\label{Llipschitzentropy}
	Suppose that for each \(Y\in\mathcal{Y}\), the \(\mathcal{Y}\rightarrow\mathbb{R}\) map \(y\mapsto\mathcal{L}(Y,y)\) is \(c\)-Lipschitz for some constant \(c>0\), i.e. for \(y_1,y_2\in\mathcal{Y}\), \(\lvert\mathcal{L}(Y,y_1)-\mathcal{L}(Y,y_2)\rvert\leq c\lVert y_1-y_2\rVert_\mathcal{Y}\). Then for any \(\delta>0\), we have
	\[H(c\delta,\mathcal{L}\circ\mathcal{G},\lVert\cdot\rVert_{2,P_n})\leq H(\delta,\mathcal{G},\lVert\cdot\rVert_{2,P_n}).\]
\end{lemma}
\begin{proof}
	To ease the notation, write \(N=N(\delta,\mathcal{G},\lVert\cdot\rVert_{2,P_n})\), and let \(g_1,...,g_N\) be a minimal \(\delta\)-covering of \(\mathcal{G}\). Then for any \(\mathcal{L}\circ g\in\mathcal{L}\circ\mathcal{G}\), there exists some \(g_j\), \(j\in\{1,...,N\}\) with \(\lVert g-g_j\rVert_{2,P_n}=(\frac{1}{n}\sum^n_{i=1}\lVert g(X_i)-g_j(X_i)\rVert^2_\mathcal{Y})^{1/2}\leq\delta\). Then by the Lipschitz condition on \(\mathcal{L}\),
	\begin{alignat*}{2}
		\left\lVert\mathcal{L}\circ g-\mathcal{L}\circ g_j\right\rVert_{2,P_n}&=\left(\frac{1}{n}\sum^n_{i=1}\left\lvert\mathcal{L}(Y_i,g(X_i))-\mathcal{L}(Y_i,g_j(X_i))\right\rvert^2\right)^{\frac{1}{2}}\\
		&\leq\left(\frac{1}{n}\sum^n_{i=1}c^2\left\lVert g(X_i)-g_j(X_i)\right\rVert_\mathcal{Y}^2\right)^{\frac{1}{2}}\\
		&=c\left\lVert g-g_j\right\rVert_{2,P_n}\\
		&\leq c\delta.
	\end{alignat*}
	Hence \(\mathcal{L}\circ g_1,...,\mathcal{L}\circ g_N\) is a \(c\delta\)-covering of \(\mathcal{L}\circ\mathcal{G}\), i.e.
	\[N(c\delta,\mathcal{L}\circ\mathcal{G},\lVert\cdot\rVert_{2,P_n})\leq N(\delta,\mathcal{G},\lVert\cdot\rVert_{2,P_n}).\]
	Now finish the proof by taking logarithms of both sides. 
\end{proof}
So for empirical risk minimisation problems with appropriate loss functions, it does make sense to consider the entropy of vector-valued function classes \(\mathcal{G}\), while it remains as future work to investigate the use of the Rademacher complexity of \(\mathcal{G}\). 
\end{document}